\documentclass[10pt,a4paper]{article}

\usepackage[T1]{fontenc}
\usepackage[all]{xy}
\usepackage{amsthm}
\usepackage{amssymb}
\usepackage{amsmath}
\usepackage{bbm}
\usepackage{graphics,graphicx}
\usepackage{ esint }
\usepackage{eepic}
\usepackage{epsfig}
\usepackage{times}
\usepackage{stmaryrd}
\usepackage{hyperref}
\usepackage{array}
\usepackage{color}

\usepackage[toc,page]{appendix}

\usepackage{tikz}
\usetikzlibrary{patterns}

\newtheorem{theorem}{Theorem}
\newtheorem{lemma}[theorem]{Lemma}
\newtheorem{Proposition}[theorem]{Proposition}
\newtheorem{definition}[theorem]{Definition}
\newtheorem{remark}[theorem]{Remark}
\newtheorem{corollary}[theorem]{Corollary}

\usepackage{caption} 
\usepackage{tikz}
\newcommand{\R}{\mathbb{R}}

\newcommand{\N}{\mathbb{N}}

\renewcommand{\O}{\mathcal{O}}
\newcommand{\C}{\mathcal{C}}
\newcommand{\K}{\mathcal{K}}
\newcommand{\F}{\mathcal{F}}

\newcommand{\D}{\mathcal{D}}
\newcommand{\Om}{\Omega}
\newcommand{\eps}{\varepsilon}

\newcounter{mnotecount}[section]

\newcommand{\rmnote}[1]{}

\title{Improved description of Blaschke--Santal\'o diagrams via numerical shape optimization}
 \author{Ilias Ftouhi}
 \date{
    \today
}

\begin{document}
\maketitle






\begin{abstract}
We propose a method based on the combination of theoretical results on Blaschke--Santal\'o diagrams and numerical shape optimization techniques to obtain improved description of Blaschke--Santal\'o diagrams in the class of planar convex sets. To illustrate our approach, we study three relevant diagrams involving the perimeter $P$, the diameter $d$, the area $A$ and the first eigenvalue of the Laplace operator with Dirichlet boundary condition $\lambda_1$. The first diagram is a purely geometric one involving the triplet $(P,d,A)$ and the two other diagrams involve geometric and spectral functionals, namely $(P,\lambda_1,A)$ and $(d,\lambda_1,A)$ where a strange phenomenon of non-continuity of the extremal shapes is observed.

\end{abstract}












\section{Introduction}\label{s:introduction}

A Blaschke--Santaló diagram is a tool that allows one to visualize all possible inequalities relating three quantities depending on the shape of a set: it was named as a reference to the works of W. Blaschke \cite{blaschke} and L. Santal\'o \cite{santalo}, where the authors were looking for the description of inequalities involving three geometrical quantities for a given convex set. Afterward, these diagrams have been extensively studied especially for the class of convex sets and more recently for triplets involving geometric and spectral functionals. For various examples on purely geometric functionals, we refer to the following non-exhaustive list of works \cite{cifre3,MR3653891,branden,delyon2,ftouhi_cheeger,ftouhi_inequality}, and for examples involving both geometric and spectral quantities, we refer to \cite{freitas_antunes,MR3068840,BBF,zbMATH07369278,ftouh,ftouhi_henrot,MR4338207,LZ,zbMATH06736468,vdBBP}. \vspace{2mm}

\textcolor{black}{Let us first present an abstract setting for Blaschke--Santal\'o diagrams. We consider  $\F$ a class of subsets of $\R^2$ and $J_i:\Om\in\F\longmapsto J_i(\Om)\in \R$, where $i\in\{1,2,3\}$ three shape functionals. The Blaschke--Santal\'o diagram corresponding to the triplet $(J_1,J_2,J_3)$ is defined as the following three-dimensional region
$$\{(J_1(\Om),J_2(\Om),J_3(\Om))\ |\ \Om\in\F\}\subset \R^3.$$
If one assumes the functionals $J_1$, $J_2$ and $J_3$ to be homogeneous (i.e., for every $i\in\{1,2,3\}$ there exists $\alpha_i\in\R$ such that $J_i(t \Om)= t^{\alpha_i}J_i(\Om)$ for $t>0$ and $\Om\in\F$), which is the case for various relevant functionals including those considered in the present paper, then it is sufficient to study a two-dimensional diagram where the value of $J_3$ is prescribed as a constraint.} 

Throughout the paper, we will adopt the following definition for Blaschke--Santal\'o diagrams:
\begin{definition}
Let $\F$ be a class of subsets of $\R^2$ and $J_i:\Om\in\F\longmapsto J_i(\Om)\in \R$, where $i\in\{1,2,3\}$, three homogeneous shape functionals. The Blaschke--Santal\'o diagram of the triplet $(J_1,J_2,J_3)$ for the class $\F$ is given by the following planar region: 
$$\D := \{(J_1(\Om),J_2(\Om))\ |\ \Om\in \F\ \text{and}\ J_3(\Om)=1\}\subset \R^2.$$
\end{definition}
In the present paper, we are interested in the class $\K$ of bounded planar convex sets with nonempty interior (i.e., convex bodies). We develop an approach based on the combination of theoretical and numerical results to provide an improved description of Blaschke--Santal\'o diagrams. We recall that a complete description of such diagrams is equivalent to finding all the possible inequalities relating the involved functionals. Unfortunately, such analytical description can be quite difficult to obtain, especially when spectral functionals are involved, see for example \cite[Conjecture 1]{ftouh} and the discussion therein. Thus, it is interesting to develop numerical tools that allow to obtain approximations of Blaschke--Santal\'o diagrams, which will surely help to develop the intuition and state some interesting conjectures. A natural idea is to generate a large number of random convex sets (more precisely convex polygons) and compute the values of the considered functionals on each set. Even if this process allows to obtain a cloud of dots that approaches the diagram, the results are not satisfactory since it is observed that some regions of the diagrams (those corresponding to smooth shapes) are quite sparse. We refer to \cite{freitas_antunes},   \cite[Section 3.1]{ftouh} and \cite[Section 5.3]{LZ} for some examples.  For a nice approximation of Blaschke--Santal\'o diagram with using Lloyd's algorithm we refer to \cite{oudet_bogosel}. 

In the present paper, we propose to combine a theoretical result of simple connectedness of the diagrams (see Theorems \ref{th:main_1} and \ref{th:main_2}) with the obtained numerical solutions of some relevant shape optimization problems to provide an improved approximation of the diagrams which is considerably better than the one obtained by the process of random generation of convex polygons, see Section \ref{s:bs_diagrams} and the figures therein. We illustrate our strategy by applying it to diagrams involving the perimeter $P$, the diameter $d$, the area $A$ and the first Dirichlet eigenvalue of the Laplacian $\lambda_1$ defined as follows
$$\lambda_1(\Om):= \min\left\{\frac{\int_{\Om} |\nabla u|^2dx}{\int_\Om u^2 dx},\ u\in H^1_0(\Om)\backslash\{0\}\right\},$$
where $H^1_0(\Om)$ denotes the completion for the $H^1$-norm of the space $C^\infty_c(\Om)$ of infinitely differentiable functions of
compact support in $\Om$, even though our strategy can be applied to many other examples.  

More precisely, we study the following diagrams: 
\begin{itemize}
    \item The purely geometric diagram $\D_1$ corresponding to the triplet $(P,A,d)$:
    \begin{equation}\label{def:D_1}
    \D_1:= \{(P(\Om),A(\Om))\ |\ \text{$\Omega\in\K$}\ \text{and}\ d(\Om)=1\}.    
    \end{equation}
    Up to our knowledge $\D_1$ is not completely characterized yet, even though, several partial (but advanced) results are known, see \cite{MR1544679,MR2410988}. 
    \item The diagram $\D_2$ corresponding to the triplet $(P,\lambda_1,A)$:
    \begin{equation}\label{def:D_2}
    \D_2:= \{(P(\Om),\lambda_1(\Om))\ |\ \text{$\Omega\in\K$}\ \text{and}\ A(\Om)=1\},    
    \end{equation}
that has been introduced for the first time in \cite{freitas_antunes} and extensively studied in \cite{ftouh}. 
    \item The diagram $\D_3$ corresponding to the triplet $(d,\lambda_1,A)$:
    \begin{equation}\label{def:D_3}
     \D_3:= \{(d(\Om),\lambda_1(\Om))\ |\ \text{$\Omega\in\K$}\ \text{and}\ A(\Om)=1\}.  
    \end{equation}
This diagram has not yet been considered in the literature and we deemed it important to consider it since a strange phenomenon of non-continuity of the extremal shapes is observed,  see Section \ref{ss:dla} and more precisely Figure \ref{fig:optimal_DLA}. 
\end{itemize}

In \cite{ftouh}, the authors prove that the diagram $\D_2$ is given by the set of points located between the graphs of two continuous and strictly increasing functions. In particular, this shows that the diagrams are simply connected (i.e., contain no holes). In the present paper, we prove results in the same flavor for the diagrams $\D_1$ and $\D_3$. Once such theoretical results are proved, it remains to numerically solve the following shape optimization problems: 
\begin{equation}\label{prob:convex}
    \min\slash \max \{(J_1(\Om)\ |\ \text{$\Omega\in \K$},\ J_2(\Om)=c_0\ \text{and}\ J_3(\Om)=1\},
\end{equation}
where $J_1$, $J_2$ and $J_3$ are the involved shape functionals and $c_0>0$ is a positive constant. This process allows to obtain improved descriptions of the diagrams. 

\color{black}
For the diagram $\D_1$, the obtained result is the following: 
\begin{theorem}\label{th:main_1}
We define 
\begin{equation}\label{eq:phi}
\begin{array}{cccl}
   \phi : &(2,\pi]& \longrightarrow&\mathbb{R}\\
   & x &\longmapsto &\inf\left\{A(\Om)\ |\ \Omega\in \mathcal{K},\ d(\Om)=1\ \text{and}\ P(\Omega)= x\right\},
\end{array}
\end{equation}
\begin{equation}\label{eq:psi}
\begin{array}{cccl}
 \psi : &(2,\pi]& \longrightarrow&\mathbb{R}\\
   & x &\longmapsto& \sup\left\{A(\Om)\ |\ \Omega\in \mathcal{K},\ d(\Om)=1\ \text{and}\ P(\Omega)= x\right\}.
\end{array}
\end{equation}
The functions $\phi$ and $\psi$ are continuous and strictly increasing. Moreover,we have 
$$\D_1 = \{(x,y)\ |\ x\in(2,\pi]\ \ \text{and}\ \ \phi(x)\leq y\leq \psi(x)\}.$$
\end{theorem}

As for the diagram $\D_3$, we obtained the following result: 
\begin{theorem}\label{th:main_2}
We define 
\begin{equation}\label{eq:f}
\begin{array}{cccl}
   f : &[\frac{2}{\sqrt{\pi}},+\infty)& \longrightarrow&\mathbb{R}\\
   & x &\longmapsto &\inf\left\{\lambda_1(\Om)\ |\ \Omega\in \mathcal{K},\ A(\Om)=1\ \text{and}\ d(\Omega)= x\right\},
\end{array}
\end{equation}
\begin{equation}\label{eq:g}
\begin{array}{cccl}
 g : &[\frac{2}{\sqrt{\pi}},+\infty)& \longrightarrow&\mathbb{R}\\
   & x &\longmapsto& \sup\left\{\lambda_1(\Om)\ |\ \Omega\in \mathcal{K},\ A(\Om)=1\ \text{and}\ d(\Omega)= x\right\}.
\end{array}
\end{equation}
The functions $f$ and $g$ are continuous and $f$ is also strictly increasing. Moreover, we have 
$$\D_3 = \{(x,y)\ |\ x\in[\frac{2}{\sqrt{\pi}},+\infty)\ \ \text{and}\ \ f(x)\leq y\leq g(x)\}.$$
\end{theorem}

Before presenting the outline of the paper, let us provide some comments on the results of Theorems \ref{th:main_1} and \ref{th:main_2}: 
\begin{itemize}
    \item The functions $\phi$ and $\psi$ (resp. $f$ and $g$) are well defined on the interval $(2,\pi]$ (resp. $[\frac{2}{\sqrt{\pi}},+\infty)$). This can be easily seen by considering the family $(S_{1,d})_{d\ge 2}$ of stadiums of diameter $d\ge 2$ and inradius $1$. Indeed, it is straightforward to check that the functions $F:d\ge 2\longmapsto P(S_{1,d})/d(S_{1,d})$ and $G:d\ge 2\longmapsto d(S_{1,d})/\sqrt{A(S_{1,d})}$ are continuous and satisfy 
    $$F([2,+\infty)) = (2,\pi]\ \ \ \text{and}\ \ \ G([2,+\infty))= \left[\frac{2}{\sqrt{\pi}},+\infty\right).$$
    \item For the diagram $\D_3$, numerical simulations suggest that the function $g$ is also increasing. Nevertheless, we are not able to prove it because our method is based on some perturbation properties of the functionals under convexity constraint (see Lemma \ref{lem:perturbation}). More precisely, we would like to prove that there is no local maximizer of the first Dirichlet eigenvalue among convex sets of prescribed area. Such result seems to be difficult to obtain for the moment as the only result in the same direction that we are aware about is that a local maximizer (if it exists) should be a polygon, see \cite{lnp_new}. 
\end{itemize}
\color{black}


 \vspace{3mm}

\textbf{Outline of the paper:} The paper is organized as follows: Section \ref{s:proof_main} is devoted to the proof of Theorems \ref{th:main_1} and \ref{th:main_2}, we also present and prove some interesting results on the diagram $\D_3$ which has (up to our knowledge) never been considered in the literature, see Section \ref{ss:dla}. In Section \ref{s:numerical_shape_optimization_methods}, we present and review four parametrizations used to handle the convexity constraint and solve problems of the type \eqref{prob:convex}. Three of these parametrizations are quite classical, this is the case of the ones based on the support functions (Section \eqref{ss:support_function}), see for example \cite{antunes2019parametric,support_premier,bogosel:hal-03607776}, the gauge function (Section \eqref{ss:gauge_function}), see \cite[Section 5]{bogosel:hal-03607776} and on the vertices of polygonal approximations of the sets (Section \eqref{ss:vertices}), see \cite{MR4089508}, and the last one based on the radial function (Section \eqref{ss:radial_function}), which is up to our knowledge original as we did not manage to find a work where the convexity of the sets is modeled via some quadratic inequality constraints as in \eqref{ineq:conv_radial}. In Section \ref{s:bs_diagrams}, we present the improved descriptions of the diagrams $\D_1$, $\D_2$ and $\D_3$. Finally, in the appendix, we compute the shape derivative of the diameter which has been successfully used in the numerical simulations.


\section{Proof of Theorems \ref{th:main_1} and \ref{th:main_2} and other theoretical results}\label{s:proof_main}

\color{black}
In the whole paper, we adopt the following notations: 
\begin{itemize}
    \item $M(x,y)$ stands for a point of coordinates $x$ and $y$. 
    \item $\K$ stands for the class of bounded planar convex sets with nonempty interior (i.e., convex bodies). 
    \item If $n, m\in \mathbb{N}$ such that $n\leq m$, we denote by $\llbracket n,m \rrbracket$ the set of integers $\{n,n+1,\dots,m\}$. 
    \item We say that a Blaschke-Santal\'o diagram $\D$ is vertically convex if for every two points $A(x, yA)$ and $B(x, yB)$ (with the same abscissa) that belong to $\D$, the (vertical) segment $[AB]$ also belongs to the diagram.
\end{itemize}
\color{black}

The diagram $\D_2$ has been studied in detail in \cite{ftouh}. Therefore, in the present paper, we focus on the diagrams $\D_1$ and $\D_3$. 

\subsection{Some classical definitions and results of convex geometry} 
Before stating the proofs, we recall some classical definitions and results of convex geometry that we use in our proofs. Since in the present paper we are interested in the planar case, the definitions and propositions are stated in dimension $2$, for a more general presentation and other results, we refer the reader to \cite{schneider}.  

\begin{definition}\label{def:minkowski_sums}
Let $K$ and $Q$ be two subsets of $\mathbb{R}^2$ and $t> 0$, we define
\begin{itemize}
    \item the Minkowski sum of the sets $K$ and $Q$ by $$K+ Q:= \{x+y\ |\ x\in K\ \text{and}\ y\in Q\},$$
    \item and the dilatation of the set $K$ by the positive coefficient $t$ by $$t K:= \{t x\ |\ x\in K\}.$$ 
\end{itemize}
\end{definition}

\begin{theorem}[Steiner formula \cite{steiner}]\label{prop:steiner}
Let $\Om$ be a planar convex body and $B_1$ a disk of unit radius. We have 
$$A(\Om+t B_1) = A(\Om)+ t P(\Om)+ \pi t^2.$$
\end{theorem}

The following result is classical and holds for any dimension. It can be found for example in \cite[Theorem 7.1.1]{schneider}.  
\begin{theorem}[Brunn--Minkowski inequality]\label{th:brunn}
Let $\Om$ and $\Om'$ be two elements of $\K$ and $t\in [0,1]$. We have 
$$A(t\Om +(1-t)\Om')^{1/2} \ge tA(\Om)^{1/2} +(1-t)A(\Om')^{1/2},$$
with equality if and only if $\Om$ and $\Om'$ are homothetic. 
\end{theorem}

\subsection{The diagram \texorpdfstring{$\D_3$}{D3} }\label{ss:proof_dla}
Unlike for $\D_1$, the diagram $\D_3$ does not involve two functionals that are linear with respect to the Minkowski sums and dilatation. This makes the strategy used in Section \ref{ss:proof_pad} impossible to adopt. The idea here is to adapt the strategy introduced and developed in \cite{ftouh} for the diagram $\D_2$ involving the triplet $(P,\lambda_1,A)$ to the case of the triplet $(d,\lambda_1,A)$. 

\subsubsection{A perturbation Lemma}
Before presenting the proof of Theorem \ref{th:main_2} et us state and prove a perturbation lemma corresponding to the functionals studied in the diagram $\D_3$. 

\begin{lemma}\label{lem:perturbation}
(Perturbation Lemma) We denote by $\K_1$ the class of planar convex bodies of unit area endowed with the Hausdorff distance. We have 
\begin{enumerate}
    \item the disk is the only local minimizer of the diameter in $\K_1$.
    \item There is no local maximizer of the diameter in $\K_1$.
    \item The disk is the only local minimizer of $\lambda_1$ in $\K_1$.
\end{enumerate}
\end{lemma}

\begin{proof}
\begin{enumerate}
    \item In order to prove that the disk is the only local minimizer of the diameter in the class  $\K_1$, we exhibit a small perturbation which consists in adding a small disk (in the Minkowski sense) to the set and then re-scaling in order preserve the area constraint. We then show that this process allows to strictly decrease the diameter locally when the set is not a disk.  
    
    By the isodiametric inequality, the disk is the only global minimizer of the diameter in $\K_1$ (it is then a local minimizer). We denote by $B_1$ a disk of radius $1$. Let us consider $\Om\in\K_1$ different from the disk. For every $t\ge 0$, we consider 
    $$\Om_t:= \frac{\Om+t B_1}{\sqrt{A(\Om+t B_1)}}\in \K_1.$$
   The sequence $(\Om_t)$ converges to $\Om$ with respect to the Hausdorff distance  when $t$ goes to $0$ and we have
    \begin{eqnarray*}
d(\Om_t) &=& d\left(\frac{\Om+t B_1}{\sqrt{A(\Om+t B_1)}}\right)\\
& =& \frac{d(\Om)+2t}{\sqrt{A(\Om+t B_1)}}\\
& =& \frac{d(\Om)+2t}{\sqrt{A(\Om)+P(\Om)t+\pi t^2}}\\
& =& \frac{d(\Om)+2t}{\sqrt{1+P(\Om)t+{o}(t)}}\\
& =& d(\Om) +\frac{1}{2}\big(4-P(\Om)d(\Om)\big)t+{o}(t). 
\end{eqnarray*}
Since $\Om$ is not a disk, we have $4A(\Om)<P(\Om)d(\Om)$ (see \cite{inequalities_convex} and the references therein). Thus, $$4-P(\Om)d(\Om)<0.$$ This shows that the set $\Om$ is not a local minimizer of the diameter in $\K_1$.  

\item Let us consider $\Om\in \K_1$. Without loss of generality, we assume that a diameter of $\Om$ is colinear to the $x$-axis and contains the origin $(0,0)$. We consider $$a:=\inf\{y\ |\ \exists x\in \R,\ (x,y)\in \Om\}\ \text{and}\ b:=\sup\{y\ |\ \exists x\in \R,\ (x,y)\in \Om\},$$
and for every $t\in [0,1)$:
$$\Om_t := \frac{\Om\cap \{y\ge (1-t)a\} \cap \{y\leq (1-t)b\}}{\sqrt{A(\Om\cap \{y\ge (1-t)a\} \cap \{y\leq (1-t)b\})}}\in \K_1.$$
We have for every $t\in (0,1)$
$$d(\Om\cap \{y\ge (1-t)a\} \cap \{y\leq (1-t)b\})=d(\Om)\ \ \text{and}\ \ A(\Om\cap \{y\ge (1-t)a\} \cap \{y\leq (1-t)b\})<A(\Om)=1,$$
thus $$\forall t\in(0,1),\ \ \  d(\Om_t)>d(\Om).$$
Finally, since the sequence (of elements of $\K_1$) $(\Om_t)_{t\in(0,1)}$ converges with respect to the Hausdorff distance to $\Om$ when $t\rightarrow 0^+$, the domain $\Om$ is not a local maximizer of the diameter in $\K_1$. 
\item The last assertion is stated and proved in \cite[Lemma 3.5]{ftouh}. 
\end{enumerate}
\end{proof}

Now, that a perturbation lemma is proved, we will apply the same strategy in \cite{ftouh} by replacing the perimeter by the diameter. We note that as in the proof of \cite[Theorem 3.9]{ftouh}, the monotonicity of the function $f$ is a consequence of the third assertion of Lemma \ref{lem:perturbation}. 

\subsubsection{Elements of proof of Theorem \ref{th:main_2}}\label{ss:proof_D3}\color{black}

The description of the diagram is done in three steps: 
\begin{itemize}
    \item First, we note that the diagram $\D_3$ is closed. This follows directly from the Blaschke Selection Theorem (see for example \cite[Theorem 1.8.7]{schneider}) and the continuity of the involved functionals.
    \item first, we focus on the study of its boundary and prove that it is given by the union of the graphs of two continuous functions $f\leq g$. We prove, in particular, that $f$ is strictly increasing. The monotonicity of $g$ seems numerically to hold, nevertheless, we are not able to prove it as our strategy requires to prove that there is no local maximizer of the first Dirichlet eigenvalue among convex sets of given measure. 
    \item Then, we show that the diagram contains no holes and thus exactly corresponds to the set of points located between the two curves. 
\end{itemize}

\textit{\textbf{Closedness of the diagram:}} Using the Blaschke selection Theorem (see, for example \cite[Theorem 1.8.7]{schneider}) and the continuity of $d$, $\lambda_1$ and $A$ with respect to the Hausdorff distance, it is straightforward by reproducing the same steps as the proofs of \cite[Proposition 3.2 and Corollary 3.3]{ftouh} to prove that the set $\D_3$ is closed. 

\textit{\textbf{Continuity:}} We start by proving the continuity of $f$. Let $d_0\in [d(B),+\infty)$. Since the diagram $\D_3$ is closed, for every $d\in [d(B),+\infty)$,  there exists $\Om_d$ solution of the following minimization problem: $$\min \left\{ \lambda_1(\Omega)\ |\ \Omega\in \K_1\ \text{and}\ d(\Omega)= d\ \right\}.$$ 

\begin{itemize}
\item Let us show an \textit{inferior limit inequality}. Let $(d_n)_{n\ge 1}$ be a real sequence converging to $d_0$ such that $$\liminf\limits_{p\rightarrow d_0} \lambda_1(\Om_d)=\lim\limits_{n\rightarrow +\infty} \lambda_1(\Omega_{d_n}).$$
Up to translations, as the diameters of $(\Om_{d_{n}})_{n\in\N^*}$ are uniformly bounded, one may assume that the domains $(\Omega_{d_n})_{n\in \mathbb{N}^*}$ are included in a fixed disk: then by Blaschke's selection Theorem, $(\Omega_{d_n})$ converges to a convex set $\Om^*$ with respect to the Hausdorff distance, up to a subsequence that we also denote by $d_{n}$ for simplicity.

By the continuity of the diameter, the area and $\lambda_1$ with respect to the Hausdorff distance among convex sets, we have
$$
\left \{
\begin{array}{l}
    |\Omega^*| = \lim\limits_{n\rightarrow+\infty}  |\Omega_{d_n}| =1,\\[3mm]
    d(\Omega^*)=  \lim\limits_{n\rightarrow+\infty}  d(\Omega_{d_n}) = \lim\limits_{n\rightarrow+\infty}  d_n=d_0,\\[3mm]
    \lambda_1(\Omega^*)=\lim\limits_{n\rightarrow +\infty} \lambda_1(\Omega_{d_n})=\liminf\limits_{p\rightarrow d_0} \lambda_1(\Om_d).
\end{array}
\right.
    $$  
    
 Then by definition of $f$ (since $\Omega^*\in \K_1$ and $d(\Omega^*)=d_0$), we obtain

$$f(d_0)\leq\ \lambda_1(\Omega^*)=\lim\limits_{n\rightarrow +\infty} \lambda_1(\Omega_{d_n})=\liminf\limits_{d\rightarrow d_0} \lambda_1(\Om_d)=\liminf\limits_{d\rightarrow d_0} f(d).$$

\item It remains to prove \textbf{a superior limit inequality}. Let $(d_n)_{n\ge1}$ be a real sequence converging to $d_0$ such that
$$\limsup\limits_{d\rightarrow d_0} f(d) = \lim\limits_{n\rightarrow +\infty} f(d_n).$$
By Lemma \ref{lem:perturbation}, there exists a sequence $(K_n)_{n\ge1}$ of $\K_1$ converging to $\Omega_{d_0}$ for the Hausdorff distance, and such that $d(K_n)=d_n$ for every $n\in \mathbb{N}^*$.

Using the definition of $f$, one can write $$\forall n\in \mathbb{N}^*,\ \ \ \ f(d_n)\leq \lambda_1(K_n).$$

Passing to the limit, we get 
$$\limsup\limits_{p\rightarrow d_0} f(d) = \lim\limits_{n\rightarrow +\infty} f(d_n)\leq \lim\limits_{n\rightarrow +\infty} \lambda_1(K_n)=\lambda_1(\Omega_{d_0})=f(d_0).$$
\end{itemize}
As a consequence, we finally get $\lim\limits_{d\rightarrow d_0} f(d) = f(d_0)$, so $f$ is continuous on $[d(B),+\infty)$. 

The same method can be applied to prove the continuity of $g$.

\textit{\textbf{Monotonicity:}}  We now prove that \textit{$f$ is strictly increasing}. Let us assume by contradiction that this is not the case. Then, by the continuity of $f$ and the fact that $\lim\limits_{+\infty} f=+\infty$ (see \eqref{eq:equivalence-1}), we deduce the existence of a local minimum of $f$ at a point $d_{0}> d(B)$. Therefore, there exists $\Om^*\in \K_1$ and $\eps>0$ such that
$$d(\Om^*)=d_{0}\;\;\;\textrm{ and }\ \ \ \forall d\in(d_{0}-\eps,d_{0}+\eps), \;\;\lambda_{1}(\Om^*)=f(d_{0})\leq f(d),$$
which implies
$$\forall \Om\in\K_{1}\textrm{ such that }d(\Om)\in(d_{0}-\eps,d_{0}+\eps), \;\;\lambda_{1}(\Om^*)\leq \lambda_{1}(\Om).$$
Because of the continuity of the diameter in $\K_{1}$, this would imply that $\Om^*$ is a local minimum (for the Hausdorff distance) of $\lambda_{1}$ in $\K_{1}$, which, from the third assertion of Lemma \ref{lem:perturbation} implies that $\Om^*$ must be a disk, contradicting $d(\Om^*)>d(B)$.

\textit{\textbf{Simple connectedness of the diagram:}}
The simple connectedness of the diagram $\D_3$ follows the same steps as the proof of simple connectedness of $\D_2$ proved in \cite{ftouh}. As the proof is much involved but very similar to the case treated in \cite{ftouh}, we limit ourself to describing the main steps of the proof: 
\begin{itemize}
    \item First, we assume by contradiction that the diagram $\D_3$ is not simply connected. Therefore, there exists a point $Q(x_Q,y_Q)$ such that 
    $$x_Q>d(B)\ \ \ \ \ \text{and}\ \ \ \ \ f(x_Q)<y_Q<g(x_Q).$$
    Since $\D_3$ is closed, we deduce that there exists $r>0$ such that the open disk $B(x_Q,r)$ is not included in $\D_3$. 
    \item The second step is to construct and study the properties of relevant continuous paths inside $\D_3$ constructed via Minkowski sums of convex sets. More precisely, for every $\Om_1,\Om_2\in\K_1$ such that $d(\Om_1) = d(\Om_2)$ and $\lambda_1(\Om_1)<\lambda_1(\Om_2)$, we introduce the path $$\C_{\Om_1,\Om_2}:=\left\{\frac{1}{\sqrt{A(t\Om_1+(1-t)\Om_2)}}(t\Om_1+(1-t)\Om_2)\ \Big|\ t\in[0,1]\right\}\subset \D_3.$$
    By using Brunn--Minkowski inequality (see for example \cite[Theorem 7.1.1]{schneider}) for the area, we have 
    $$\forall t\in[0,1],\ \ \ A(t\Om_1+(1-t)\Om_2)^{1/2}\ge tA(\Om_1)^{1/2}+(1-t)A(\Om_2)^{1/2} = 1.$$
    On the other hand, since the involved shape functionals are invariant for rigid motions of the sets, we can assume without loss of generality that the diameters of the sets $\Om_1$ and $\Om_2$ are colinear, leading to 
    $$\forall t\in[0,1],\ \ \ d(t\Om_1+(1-t)\Om_2) = td(\Om_1)+(1-t)d(\Om_2).$$
    We then conclude that 
    $$\forall t\in[0,1],\ \ \ d\left(\frac{1}{\sqrt{A(t\Om_1+(1-t)\Om_2)}}(t\Om_1+(1-t)\Om_2)\right)\leq d(\Om_1) = d(\Om_2),$$
    showing that the arc (which belongs to $\D_3$) passes through the left side of the (vertical) segment $\mathcal{S}_{\Om_1,\Om_2}:= \{(d(\Om_1),y)\ |\ y\in [\lambda_1(\Om_1),\lambda_1(\Om_2)]\}$. 
    
    To be able to use the winding number tools, we are going to consider closed paths $\Gamma_{\Om_1,\Om_2}:=  \C_{\Om_1,\Om_2}\cup\mathcal{S}_{\Om_1,\Om_2}$ composed with two parts: one (real) part $\C_{\Om_1,\Om_2}$ included in the diagram $\D_3$ and one (artificial) part $\mathcal{S}_{\Om_1,\Om_2}$, that we do not claim (yet) to be included in $\D_3$, added to obtain a closed curve. 
    \item Now, that we have constructed closed curves $\Gamma_{\Om_1,\Om_2}$, we study their properties:
    \begin{itemize}
        \item We first prove that they are stable under perturbations of the domains $\Om_1$ and $\Om_2$ in the sense that if $(\Om_1^n)$ and $(\Om_2^n)$ are sequences of elements of $\K_1$ converging to $\Om_1$ and $\Om_2$ then we have that the sequence of paths $(\Gamma_{\Om_1^n,\Om_2^n})$ converges, with respect to the Hausdorff distance of sets, to $(\Gamma_{\Om_1,\Om_2})$.
        \item We then show some placement properties of the paths with respect to the point $Q$. We show that 
        \begin{itemize}
            \item There exists $d>d(B)$ and $\Om_1,\Om_2$ such that $d(\Om_1)=d(\Om_2) = d$ and the point $Q$ is "in the interior" of the closed curve $\Gamma_{\Om_1,\Om_2}$, where the notion of interior and exterior of a closed curve is defined via the winding number.
            \item There exists $d_0>d(B)$ such that for all $d\ge d_0$ and for all $\Om_1,\Om_2$ such that $d(\Om_1)=d(\Om_2) = d$, the point $Q$ is in the exterior of the closed curve $\Gamma_{\Om_1,\Om_2}$.
        \end{itemize}
    \end{itemize}
    \item The last step is to conclude by introducing the well defined quantity:
    $$d^* := \sup\{d\ |\ \exists \Om_1,\Om_2\in\K_1,\ \ \text{such that $d(\Om_1)=d(\Om_2)=d$ and $Q$ is in the interior of $\Gamma_{\Om_1,\Om_2}$}\},$$
    which provides a threshold between two different regimes: one where there exist paths having $Q$ in their interior and one where this is not the case. One then can find an absurdity by using the continuity of the index (i.e., winding number) of a closed curve in the same way as in the fourth step of the  proof of \cite[Theorem 3.14]{ftouh}. 
\end{itemize}
\color{black}

\subsubsection{A symmetry and regularity result}
In this section, we prove a symmetry and regularity result on the domains that fill the lower boundary of the diagram $\D_3$, i.e., those solving the problem 
\begin{equation}\label{prob:min__dla}
\min\{\lambda_1(\Om)\ |\ \text{$\Om\subset\R^2$ is convex},\  d(\Om)=d_0\ \text{and}\ A(\Om)=1\},
\end{equation}
where $d_0> \frac{2}{\sqrt{\pi}}$ (we recall that if $d_0=\frac{2}{\sqrt{\pi}}$, then, by the isodiametric inequality, the only solution of the problem is the disk). 

Our result is stated as follows:
\begin{theorem}\label{th:symmetry}
If $\Om^*$ is a solution of the problem \eqref{prob:min__dla}, with $d_0> \frac{2}{\sqrt{\pi}}$, then: 
\begin{enumerate}
    \item $\Om^*$ admits two orthogonal axes of symmetry. 
    \item Apart from two diametrically opposed points (i.e., points of the boundary of the domain that are on opposite ends of a diameter), the boundary of $\Om^*$ is (at least) $C^1$ at every point.  
\end{enumerate}
\end{theorem}
\begin{proof}
Let us prove each assertion:
\begin{enumerate}\color{black}
    \item Without loss of generality, we assume that there is a diameter of $\Om^*$ which is colinear to the $x$-axis and the origin $(0,0)$ is its midpoint. If we assume that the $x$-axis is not an axis of symmetry of $\Om^*$, then by performing Steiner symmetrization (see for example \cite[Section 2.2]{henrot}) on $\Om^*$ with respect to the $x$-axis, we obtain a set $\Om'$ which is symmetric with respect to the $x$-axis and such that 
    $$\lambda_1(\Om^*)<\lambda_1(\Om'),\ \ \  A(\Om^*)=A(\Om')\ \ \text{and}\ \ \ d(\Om^*)=d(\Om')=d_0,$$
    which contradicts the fact that $\Om^*$ is a solution of Problem \eqref{prob:min__dla}. 

    By repeating the same arguments for the $y$-axis, we deduce that the optimal set $\Om^*$ is also symmetric with respect to the $y$-axis. 
    
    \color{black}
   \item We recall that the lower boundary of the diagram $\D_3$ is given by the graph of the function $f$ defined in \eqref{eq:f} as follows:
$$
\begin{array}{cccl}
   f : &[d(B),+\infty)& \longrightarrow&\mathbb{R}\\
   & x &\longmapsto &\min\left\{\lambda_1(\Omega)\ |\ \Omega\in \mathcal{K}, A(\Om)=1\ \text{and}\ d(\Omega)= x\right\}
\end{array}
$$
    
    Let us assume, by contradiction, that $\Om^*$ has two distinct supporting lines at some point $x_0$ of its boundary, see Figure \ref{fig:cusp}, such that there exist two diametrically opposed points both different from $x_0$. By removing a small cup as it is done in \cite[Section 3.1]{henrot_ouder}, see Figure \ref{fig:cusp}, we construct a family of convex sets $(\Om_\eps)_{\eps>0}$ converging with respect to the Hausdorff distance to $\Om^*$ when $\eps$ goes to $0^+$. By reproducing the same steps of the proof of \cite[Lemma 3.3.2]{henrot}, we show that for sufficiently small values of $\eps>0$, we have:
    \begin{itemize}
        \item $|\Om_\eps|\lambda_1(\Om_\eps)<|\Om^*|\lambda_1(\Om^*)=\lambda_1(\Om^*)$.
        \item $|\Om_\eps|<|\Om^*|=1\ \ \ \ \ \text{(because of the strict inclusion $\Om_\eps\subset \Om^*$)}$.
        \item $d(\Om_\eps)=d(\Om^*)\ \ \ \ \ \text{(there exist two diametrically opposed points both different from $x_0$)}.$
    \end{itemize}

    \begin{figure}[!ht]
 \centering
    \includegraphics[scale=3]{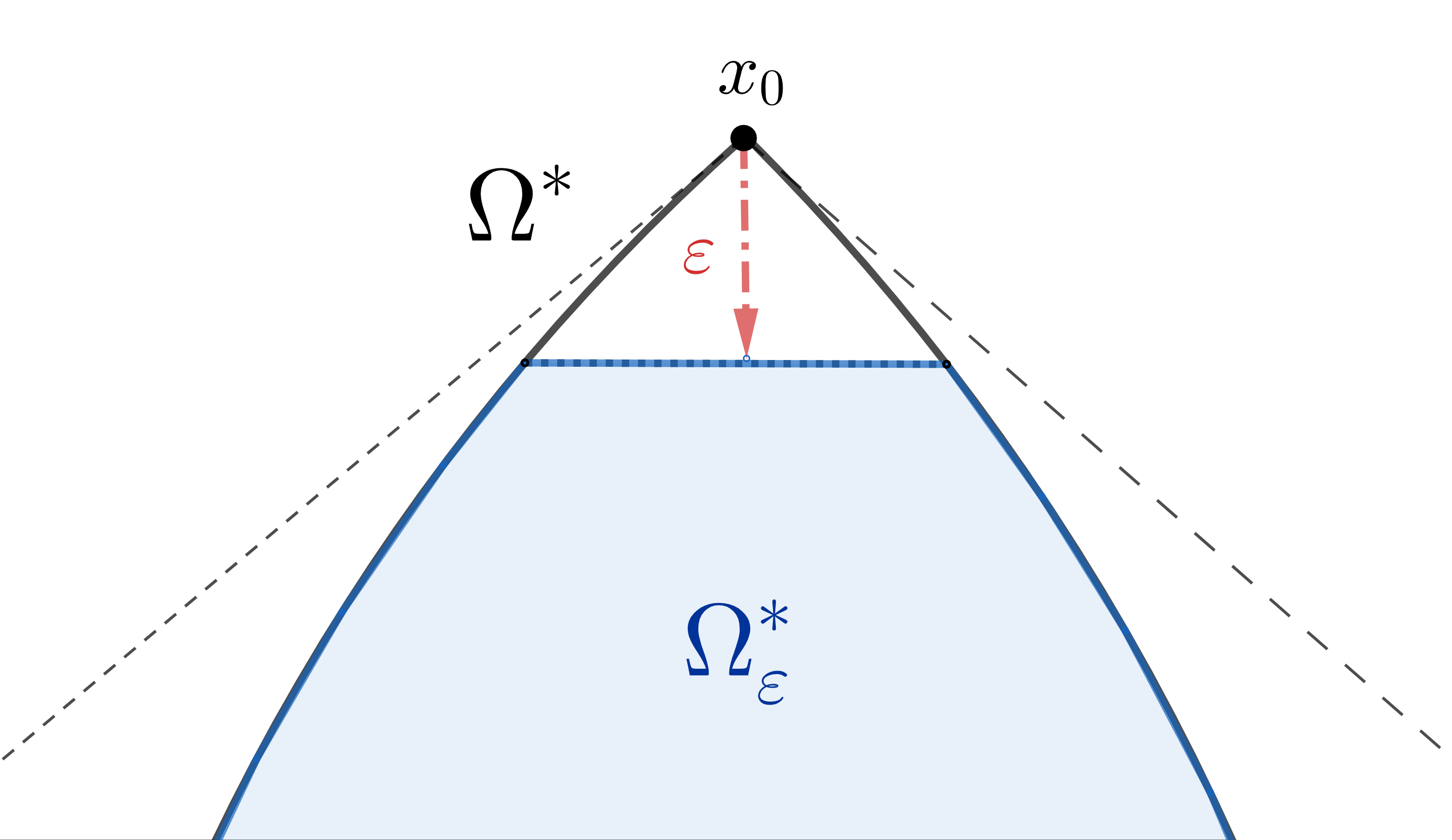}
    \caption{Removing a cap from $\Om^*$ to obtain $\Om^*_\eps$.}
    \label{fig:cusp}
\end{figure}
 
Thus, by considering the normalized family 
$\Om_\eps^*:=\frac{\Om_\eps}{\sqrt{|\Om_\eps|}}$, we  obtain a sequence of convex sets $(\Om_\eps^*)$ of unit area ($|\Om_\eps^*|=1$) converging to $\Om^*$ with respect to the Hausdorff distance such that 
$$d(\Om_\eps^*)=d\left(\frac{\Om_\eps}{\sqrt{|\Om_\eps|}}\right)=\frac{d(\Om_\eps)}{\sqrt{|\Om_\eps|}}>\frac{d(\Om^*)}{\sqrt{|\Om^*|}}=d(\Om^*),$$
and 
$$f(d(\Om_\eps^*))\leq \lambda_1(\Om_\eps^*)=\lambda_1\left(\frac{\Om_\eps}{\sqrt{|\Om_\eps|}}\right)=|\Om_\eps|\lambda_1(\Om_\eps)<\lambda_1(\Om^*) = f(d(\Om^*)).$$
Thus, we have 
$$
d(\Om_\eps^*)>d(\Om^*)\ \ \text{and}\ \  f(d(\Om_\eps^*))< f(d(\Om^*)), 
$$
which is a contradiction with the monotonicity of the function $f$ proved in Theorem \ref{th:main_2}. This ends the proof of the regularity assertion.
\end{enumerate}
\end{proof}

\subsubsection{Further results and comments}
We end this section by stating some interesting results and remarks: 
\begin{itemize}
    \item The result of monotonicity of the function $f$ is quite interesting and useful as it yields the (non-trivial) equivalence between four shape optimization problems, as stated in the following corollary:
    \begin{corollary}
 Let $d_0>d(B)=\frac{2}{\sqrt{\pi}}$. The following problems are equivalent:
 
\begin{minipage}{7cm}
\begin{enumerate}
    \item\label{item:1}  $ \min\{\lambda_1(\Omega)\ |\ \Omega \in \K_1\ \text{ and }\ d(\Omega)= d_0\}$
    \item\label{item:2}  $\min\{\lambda_1(\Omega)\ |\ \Omega \in \K_1\ \text{ and }\ d(\Omega)\ge d_0\}$
\end{enumerate}
\end{minipage}
\begin{minipage}{7cm}
\begin{enumerate}
  \setcounter{enumi}{2}
    \item\label{item:3} $\max\{d(\Omega)\ |\ \Omega \in \K_1\ \text{ and }\ \lambda_1(\Omega)=f(d_0)\}$
    \item\label{item:4} $\max\{d(\Omega)\ |\ \Omega \in \K_1\ \text{ and }\ \lambda_1(\Omega)\leq f(d_0)\}$,
\end{enumerate}
\end{minipage}\\[2mm]  
and, any solution satisfies the symmetry and regularity properties stated in Theorem \ref{th:symmetry}.
    \end{corollary}

The proof of this result is omitted as it is identical to the proof of \cite[Corollary 3.13]{ftouh}. 
\item We also note that the monotonicity of $f$ is also used in the proof of the regularity result of Theorem \ref{th:symmetry}. 

\item It is possible to obtain some sharp estimations on the functions $f$ and $g$ by using some classical inequalities. Indeed, we recall that by Makai's \cite{Makai} and Polya's \cite{Polya} inequalities, we have for every $\Om\in\K$
\begin{equation}\label{eq:inequalities_polya_makai}
    \frac{\pi^2}{16}\left(\frac{P(\Om)}{A(\Om)}\right)^2<\lambda_1(\Om)<\frac{\pi^2}{4}\left(\frac{P(\Om)}{A(\Om)}\right)^2.
\end{equation}
Both estimates are sharp as the lower one is asymptotically attained by any family of thin vanishing triangles and the second one is asymptotically attained by any family of thin vanishing rectangles. By using the following inequalities 
$$2 d(\Om)< P(\Om)< \frac{A(\Om)}{d(\Om)}+2 d(\Om),$$
that are both sharp as the equality is asymptotically attained by any family of thin vanishing domains, see \cite{inequalities_convex} and the references therein, we have 
\begin{equation}\label{eq:ilias_makai}
\frac{\pi^2}{4}\left(\frac{d(\Om)}{A(\Om)}\right)^2< \lambda_1(\Om)< \frac{\pi^2}{4}\left(\frac{4\sqrt{A(\Om)}}{d(\Om)}+\frac{2d(\Om)}{\sqrt{A(\Om)}}\right)^2,    
\end{equation}
where both estimates are sharp as for \eqref{eq:inequalities_polya_makai}. Finally, by using the definition of $f$ and $g$, we have the following inequalities
$$\forall x\ge d(B)=\frac{2}{\sqrt{\pi}},\ \ \ \frac{\pi^2}{4}x^2< f(x)\leq g(x) < \frac{\pi^2}{4}\left(\frac{4}{x}+2x\right)^2.$$
Moreover, by the sharpness of the estimates \ref{eq:ilias_makai}, we have the following asymptotics 
\begin{equation}\label{eq:equivalence-1}
f(x)\underset{x\rightarrow +\infty}{\sim}\frac{\pi^2}{4}x^2\ \ \ \text{and}\ \ \ g(x)\underset{x\rightarrow +\infty}{\sim}\pi^2x^2. 
\end{equation}

\end{itemize}

\subsection{The diagram \texorpdfstring{$\D_1$}{D1} }\label{ss:proof_pad}
\color{black}
\textbf{\textit{The continuity and the monotonicity:}} The proof of the continuity and the monotonicity of the functions $\phi$ and $\psi$ defined in \eqref{eq:psi} and \eqref{eq:phi} follows the same scheme presented in Section \ref{ss:proof_D3} for the functions $f$ and $g$ defined in \eqref{eq:f} and \eqref{eq:g}. Therefore, it only remains to prove the following perturbation lemma: 
\begin{lemma}
    We denote by $\K^{\text{diam}}_1$ the class of planar convex bodies of unit diameter endowed with the Hausdorff distance. We have 
\begin{enumerate}
    \item The sets of constant width are the only local maximizers of the perimeter in $\K^{\text{diam}}_1$.
    \item There is no local minimizer of the perimeter in $\K^{\text{diam}}_1$.
    \item Disks are the only local maximizers of the area in $\K^{\text{diam}}_1$.
    \item There is no local minimizer of the area in $\K^{\text{diam}}_1$.
\end{enumerate}
\end{lemma}
\begin{proof}
    \begin{enumerate}
        \item We have for every $\Om\in\K$, 
        $$\frac{P(\Om)}{d(\Om)} \leq \pi,$$
        with equality if and only if $\Om$ is a set of constant width. Let us now consider $\Om\in\K$, which is not a set of constant width, i.e., such that $\frac{P(\Om)}{d(\Om)} < \pi$ and $B_1$ the unit disk. We have for $t\ge 0$
        $$\frac{P(\Om+tB_1)}{d(\Om+tB_1)} = \frac{P(\Om)+2\pi t}{d(\Om)+2t} = \frac{P(\Om)}{d(\Om)}\cdot\left(1+\left(\frac{\pi}{P(\Om)}-\frac{1}{d(\Om)}\right)t +o(t)\right).$$
        Therefore, for sufficiently small $t>0$, we have 
        $$\frac{P(\Om+tB_1)}{d(\Om+tB_1)}>\frac{P(\Om)}{d(\Om)},$$
        which proves the first assertion.
        \item For the second assertion, we consider the same construction as in the second assertion of Lemma \ref{lem:perturbation}. 
        \item By the isodiametric inequality, disks are global maximizers of the area in $\K^{\text{diam}}_1$. If $\Om$ is not a disk then we can combine the isoperimetric inequality $4A(\Om)<\frac{P(\Om)^2}{\pi}$ with the inequality $P(\Om)\leq \pi d(\Om)$ and obtain the inequality $4 A(\Om)<P(\Om) d(\Om)$.

        On the other hand, we have by using Steiner's formula (c.f., Theorem  \ref{prop:steiner}) for $t\ge 0$:
        $$\frac{A(\Om+tB_1)}{d(\Om+t B_1)^2} = \frac{A(\Om)+P(\Om)t+\pi t^2}{(d(\Om)+2t)^2} = \frac{A(\Om)}{d(\Om)}\left(1+\left(\frac{P(\Om)}{A(\Om)}-\frac{4}{d(\Om)}\right)t +o(t)\right).$$
        Therefore, for sufficiently small values of $t>0$, we have 
        $$\frac{A(\Om+tB_1)}{d(\Om+t B_1)^2}>\frac{A(\Om)}{d(\Om)^2}.$$
        \item The fourth assertion is equivalent to the third assertion of Lemma \ref{lem:perturbation}. 
    \end{enumerate}
\end{proof}

\textbf{\textit{The simple connectedness:}}
By linearity properties of $P$ and $d$ with respect to the Minkowski sums and dilatation, the proof of the simple connectedness is quite standard. Let us consider two convex bodies $K_0$ and $K_1$ that are respectively solutions of the problems 
$$\min\left\{A(\Om)\ |\ \Omega\in \mathcal{K},\ d(\Om)=1\ \text{and}\ P(\Omega)= p_0\right\}$$
and 
$$\max\left\{A(\Om)\ |\ \Omega\in \mathcal{K},\ d(\Om)=1\ \text{and}\ P(\Omega)= p_0\right\}.$$

 For every $t\in [0,1]$, we define 
$$K_t = (1-t)K_0+t K_1.$$
Since the involved functionals are invariant by rotations, we assume without loss of generality that the diameters of $K_0$ and $K_1$ are colinear. We then have for every $t\in[0,1]$,
$$d(K_t) = (1-t)d(K_0)+t d(K_1)= (1-t)+t=1$$
and 
$$P(K_t) = (1-t)P(K_0)+t P(K_1)= (1-t)p_0+t p_0=p_0.$$

Then, since the perimeter and the area are continuous with respect to the Hausdorff distance in the class $\K$, the curve $t\in[0, 1]\longmapsto (P(K_t),A(K_t))$ is continuous in $\R^2$. Thus, the set $\{(P(K_t),A(K_t))\ |\ t\in [0,1]\}$ (which is included in the diagram $\D_1$) coincides with the vertical segment connecting the points $(p_0,A(K_0))$ and $(p_0,A(K_1))$, which proves that the diagram is exactly given by the set of points located between the graphs of the functions $\phi$ and $\psi$. 
\color{black}

\section{Parametrizations of convex sets and numerical setting}\label{s:numerical_shape_optimization_methods}

Before describing the parametrizations used in the present paper, let us recall the definition of (directional) first order shape derivative that is very important in numerical shape optimization. 
\begin{definition}
Let us take a shape depending functional $J:\Om\subset \R^2\longrightarrow \R$, and let $V:\R^2\rightarrow \R^2$ be a perturbation vector field. For $\Om\subset \R^2$, we take $\Om_t:=(I+tV)(\Om)$ where $I:x\in\R^n\longmapsto x$ is the identity map and $t$ a positive number. We say that the functional $J$ admits a directional shape derivative at $\Om$ in the direction $V$ if the following limit $\lim\limits_{t\rightarrow 0^+}\frac{J(\Om_t)-J(\Om)}{t}$ exists. In this case, we write
$$J'(\Om,V):= \lim\limits_{t\rightarrow 0^+}\frac{J(\Om_t)-J(\Om)}{t}.$$
\end{definition}

We recall that when $\Om$ is convex (or sufficiently smooth), the shape derivative of the area $A$ in the direction $V:\R^2\longrightarrow \R^2$ is given by the formula 
\begin{equation}\label{eq:shape_derivative_area}
    A'(\Om,V) = \int_{\partial \Om} \langle V,n\rangle d\mathcal{H}^1,
\end{equation}
and for the first Dirichlet eigenvalue $\lambda_1$, the shape derivative is given by
\begin{equation}\label{eq:shape_derivative_dirichlet}
\lambda_1'(\Om,V)=-\int_{\partial \Om} |\nabla u_1|^2\langle V,n\rangle d\mathcal{H}^1,
\end{equation}
where $n$ stands for the exterior unit normal vector to $\partial \Om$ and $u_1\in H^1_0(\Om)$ corresponds to a normalized eigenfunction (i.e., $\|u_1\|_2=1$) corresponding to the first eigenvalue $\lambda_1(\Om)$. \vspace{2mm}

Let us now present the four parametrizations considered in the paper and show how the convexity and the other constraints are implemented.  
\subsection{Parametrization via the support function }\label{ss:support_function}
\subsubsection{Definitions and main properties}
The support function is a useful tool to parametrize a convex set by a function defined on the unit sphere. It allows to formulate geometric problems into analytical ones. Thus, one can use tools from calculus of variations to tackle geometric questions. 


Let us now recall the definition of the support function: 
\begin{definition}
Let $\Om \in \K$  be a planar convex body. The support function $h_\Om$ is defined on $\R^n$ by 
$$\forall x \in \R^2,\ \ \  h_\Om(x) := \sup_{y\in \Om} \langle x,y \rangle.$$
The support function is positively $1$-homogeneous, so we can equivalently consider the restriction of $h_\Om$ to the interval $[0,2\pi)$. The support function of a planar set $\Om \in \K$ is  then defined by 
$$\forall \theta \in [0,2\pi),\ \ \ h_\Om(\theta)= \sup_{x\in \Om} \left\langle x,\binom{\cos \theta}{\sin \theta} \right\rangle= \sup_{(x_1,x_2)\in \Om} (x_1 \cos{\theta}+ x_2 \sin{\theta}).$$
\end{definition}



The support function has various interesting properties as linearity for Minkowski sums, characterizing a convex set and quite practical formulations of different geometrical quantities such as the perimeter, diameter, area and width. There are many other properties that enhance the popularity of this parametrization; we refer to \cite{support_properties_bayen} for a complete survey and detailed proofs. 

Let us state the main properties of the support function used in the present paper. \begin{Proposition}
Let $\Om_1,\Om_2\in \K$ and $h_{\Om_1},h_{\Om_2}$ the corresponding support functions, we have the following properties:
\begin{enumerate}
    \item $\Om_1\subset \Om_2\Leftrightarrow h_{\Om_1}\leq h_{\Om_2}$. 
    \item $h_{\lambda_1\Om_1+\lambda_2\Om_2}= \lambda_1 h_{\Om_1}+ \lambda_2 h_{\Om_2}$, where $\lambda_1,\lambda_2>0$. 
    \item $d^H(\Om_1,\Om_2) = \|h_{\Om_1}-h_{\Om_2}\|_{\infty}:= \sup\limits_{\theta\in[0,2\pi)} |h_{\Om_1}(\theta)-h_{\Om_2}(\theta)|$. 
\end{enumerate}
\end{Proposition}
\begin{proof}
For the first assertion we refer to \cite[Theorem 3.3.1]{convex_geometry}. For the second, we refer to \cite[Theorem 3.3.2]{convex_geometry}. As for the last one, we refer to \cite[Theorem 3.3.6]{convex_geometry} or \cite[Lemma 1.8.14]{schneider}.
\end{proof}

It is now natural to wonder how can the support function describe a convex shape (or more precisely its boundary). The following proposition provides an efficient parametrization of strictly convex planar domains, which are considered in numerical simulations to approximate the optimal shapes.  
\begin{Proposition}[]\label{prop:boundary_support}
Let $\Om\in \K$. The support function $h_\Om$ of the convex $\Om$ is of class $C^1$ on $\R$ if and only if $\Om$ is strictly convex, in which case its boundary $\partial \Om$ will be parametrized as follows: 
    \begin{equation}\label{eq:change_of_variables}
\left\{
\begin{array}{l}
  x_\theta = h_\Om(\theta) \cos \theta - h'_\Om(\theta) \sin{\theta},\vspace{2mm}\\
  y_\theta = h_\Om(\theta)\sin{\theta} + h'_\Om(\theta) \cos{\theta},\\
\end{array}
\right.
\end{equation}
where $\theta\in [0,2\pi)$. 
\end{Proposition}
\begin{proof}
    See Propositions 2.1 and 2.2 of \cite{support_properties_bayen}.
\end{proof}

Now that we know that for any convex set one can associate a support function, it is natural to seek for conditions that a function should satisfy in order to be the support function of a convex set. The answer is tightly related to the fact that the convexity of a set is equivalent to the positivity of the radius of curvature at any point of its boundary. 

\begin{Proposition}
Let $\Om$ a strictly convex planar set, we assume that its support function $h_\Om$ is $C^{1,1}$, then the geometric radius of curvature of $\partial \Om$ is given by $R_\Om = h_\Om''+h_\Om$ and we have 
\begin{equation}\label{eq:convexite_support}
    \forall \theta\in [0,2\pi),\ \ \ \ \ R_\Om(\theta) = h_\Om''(\theta)+h_\Om(\theta)\ge 0.
\end{equation}
Reciprocally, if $h$ is a $C^{1,1}$, $2\pi$ periodic function satisfying \eqref{eq:convexite_support}, then there exists a convex set $\Om\in \K$ such that $h=h_\Om$.
\end{Proposition}
\begin{proof}
    See Proposition 2.3 and Corollary 2.1 of \cite{support_properties_bayen}.
\end{proof}
\begin{remark}
The results above are stated for strictly convex sets with smooth support functions (which is enough for the numerical simulations, see \ref{ss:numerical_setting}). Nevertheless, let us give some remarks on the singular cases:
\begin{itemize}
    \item When $h_\Om$ is just $C^1$, the condition $R_\Om:=h_\Om''+h_\Om\ge 0$ can be understood in the sense of distributions that is to say that $R_\Om:=h_\Om''+h_\Om$ is a positive Radon measure (i.e., for all $C^\infty$ positive function $\phi$ of compact support in $[0,2\pi]$, one has: $\int_0^{2\pi}R_\Om\phi\ge 0$). 
    \item When $\Om$ is not strictly convex, the support function $h_\Om$ is not $C^1$ and the measure corresponding to the radius of curvature $R_\Om$ may involve Dirac measures. This is for example the case for polygons where $R_\Om$ will be given by a finite sum of Dirac measures, see \cite{schneider} for example.
\end{itemize}
\end{remark}

In addition to providing a quite simple description to the convexity constraint, see \eqref{eq:convexite_support}, the support function provides elegant expressions for some relevant geometric functionals. 
\begin{Proposition}
Let $\Om\in \K$ and $h_\Om$ be its support function, we have the following formulas: 
\begin{enumerate}
    \item for the perimeter $$P(\Om) = \int_{0}^{2\pi} h_\Om(\theta)d \theta,$$
    \item for the area $$A(\Om) = \frac{1}{2}\int_0^{2\pi} \big(h_\Om^2(\theta)-h_\Om'^2(\theta)\big)d\theta,$$
    \item for the minimal width 
    $$w(\Om) = \min\limits_{\theta\in [0,2\pi)} \big(h_\Om(\theta)+h_\Om (\pi+\theta)\big),$$
    \item for the diameter $$d(\Om) = \max\limits_{\theta\in [0,2\pi)} \big( h_\Om(\theta)+h_\Om (\pi+\theta)\big).$$
\end{enumerate}
\end{Proposition}
\begin{proof}
The presented formulas are classical in convex geometry. We refer to \cite[Theorem 2.1]{support_properties_bayen} for the perimeter, to \cite[Theorem 2.2]{support_properties_bayen} for the area and for \cite[Section 3.1]{support_properties_bayen} for the width and the diameter. 
\end{proof}

\subsubsection{Numerical setting}\label{ss:numerical_setting}
Let us take $\Om\in \K$. Since $h_\Om$ is an $H^1$, $2\pi$-periodic function, it admits a decomposition in Fourier series: 
$$h_\Om(\theta) = a_0 + \sum_{n=1}^\infty (a_n \cos{n\theta}+ b_n\sin{n\theta}),$$
where $(a_n)_n$ and $(b_n)_n$ denote the Fourier coefficients defined by
$$a_0 = \frac{1}{2\pi}\int_0^{2\pi}h_\Om(\theta)d\theta$$
and 
$$\forall n\in \N^*,\ \ \ a_n = \frac{1}{\pi}\int_0^{2\pi} h_\Om(\theta)\cos{(n\theta)}d\theta,\ \ \ b_n = \frac{1}{\pi}\int_0^{2\pi} h_\Om(\theta)\sin{(n\theta)}d\theta.$$

We can then express the area and perimeter via the Fourier coefficient as follows 
$$P(\Om) = 2\pi a_0\ \ \ \text{and}\ \ \ A(\Om)= \pi a_0^2 + \frac{\pi}{2}\sum_{k=1}^\infty (1-k^2)(a_k^2+b_k^2).$$

To retrieve a finite dimensional setting, the idea is to parametrize sets via Fourier coefficients of their support functions truncated at a certain order $N\ge 1$. Thus, we will look for solutions in the set: 
$$\mathcal{H}_N:= \left\{\theta\longmapsto a_0 + \sum_{k=1}^N \big(a_k \cos{(k\theta)}+ b_k\sin{(k\theta)}\big) \ \big|\ a_0,\dots,a_N,b_1,\dots,b_N\in \R\right\}.$$
This approach is justified by the following approximation proposition: 
\begin{Proposition}(\cite[Section 3.4]{schneider})\\
Let $\Om\in \K$ and $\eps>0$. Then there exists $N_\eps$ and $\Om_\eps$ with support function $h_{\Om_\eps}\in \mathcal{H}_{N_\eps}$ such that $d^H(\Om,\Om_\eps)<\eps$. 
\end{Proposition}
For more convergence results and applications of this method for different problems, we refer to \cite{antunes2019parametric}. \vspace{2mm}

Let $N\ge 1$, we summarize the parametrizations of functionals and constraints: 
\begin{itemize}
    \item The set $\Om$ is parametrized via the Fourier coefficients of its support function $a_0,\ldots,a_N,b_1,\ldots,b_N$.
    \item The convexity constraint is given by the condition $h''_\Om+h_\Om\ge 0$ on $[0,2\pi)$. In \cite{support_premier} the authors provide an exact characterization of this condition in terms of the Fourier coefficients, involving concepts from semidefinite programming. In \cite{MR3627042} the author provides a discrete alternative of the convexity inequality which has the advantage of being linear in terms of the Fourier coefficients. We choose $\theta_m = 2\pi m/M$ where $m\in \llbracket 1,M \rrbracket$ for some positive integer $M$ and we impose the inequalities $h''_\Om(\theta_m)+h_\Om(\theta_m)\ge 0$ for $m\in \llbracket 1,M \rrbracket$. As already shown in \cite{MR3627042} we obtain the following system of linear inequalities: $$\begin{pmatrix}
    1 & \alpha_{1,1} & \cdots & \alpha_{1,N} & \cdots & \beta_{1,1} & \cdots & \beta_{1,N}\\
     \vdots & \vdots & \ddots & \vdots & \cdots & \vdots & \ddots & \vdots\\
    1 & \alpha_{N,1} & \cdots & \alpha_{N,N} & \cdots & \beta_{N,1} & \cdots & \beta_{N,N}
    \end{pmatrix} 
    \begin{pmatrix}
    a_0\\
    a_1\\
     \vdots \\
    a_N\\ 
    b_1\\
    \vdots\\ 
    b_N
    \end{pmatrix} \ge \begin{pmatrix}
    0\\
     \vdots \\
    0
    \end{pmatrix}$$
    where $\alpha_{m,k}=(1-k^2)\cos{(k\theta_m)}$ and $\beta_{m,k}=(1-k^2)\sin{(k\theta_m)}$ for $(m,k)\in \llbracket 1,M\rrbracket\times \llbracket 1,N\rrbracket$. 
    \item The perimeter constraint $P(\Om)=p_0$ is given by $$2\pi a_0 = p_0.$$
    \item The area constraint $A(\Om)=A_0$ is given by $$\pi a_0^2 + \frac{\pi}{2}\sum_{n=1}^N (1-k^2)(a_k^2+b_k^2)=A_0.$$
    \item The diameter constraint $d(\Om) = d_0$ is equivalent to 
    \begin{equation*}
    \left\{
    \begin{array}{l}
    \forall \theta\in [0,2\pi),\ \ \ \ h_\Om(\theta)+h_\Om(\pi+\theta)\leq d_0,\vspace{2mm}\\
    \exists \theta_0\in [0,2\pi),\ \ \ h_\Om(\theta_0)+h_\Om(\pi+\theta_0)= d_0.\\
    \end{array}
    \right.
    \end{equation*}
    Again, as for the convexity constraint, we choose $\theta_m = 2\pi m/M$ where $m\in \llbracket 1,M' \rrbracket$ for some positive integer $M'$ and we impose the inequalities $h_\Om(\theta_m)+h_\Om(\pi+\theta_m)\leq d_0$ for $m\in \llbracket 1,M' \rrbracket$, we also assume without loss of generality that $h_\Om(0)+h_\Om(\pi)= d_0$ (because all functionals are invariant by rotations). All theses conditions can be written in terms of $(a_k)$ and $(b_k)$ as the following linear constraints: 
     \begin{equation*}
    \left\{
    \begin{array}{l}
    \forall m\in \llbracket 1,M' \rrbracket,\ \ \ \ 2 a_0 + \sum\limits_{k=1}^N \big((1+(-1)^k) \cos{(k\theta_m)}\cdot a_k+ (1+(-1)^k)\sin{(k\theta_m)}\cdot b_k\big)\leq d_0,\vspace{2mm}\\
    \ \ \ 2 a_0 + \sum\limits_{k=1}^N (1+(-1)^k)\cdot a_k = d_0.\\
    \end{array}
    \right.
    \end{equation*}
\end{itemize}

\subsubsection{Computation of the gradients}
 In order to have an efficient optimization algorithm, we compute the derivatives of the eigenvalue and the area in terms of the Fourier coefficients of the support function (while for the convexity, the diameter and the perimeter constraints no gradient computation is needed in this setting since these constraints are linear). To this aim, we first consider two types of perturbations, a cosine term and a sine term, namely two families of deformations $(V_{a_k})$ and $(V_{b_k})$ that respectively correspond to the perturbation of the coefficients $(a_k)$ and $(b_k)$ in the Fourier decomposition of the support function. As stated in Proposition \ref{prop:boundary_support}, when $\Om$ is strictly convex, the support function provides a parametrization of its boundary $\partial \Om = \{(x_\theta,y_\theta)|\ \theta\in [0,2\pi]\}$; then  the perturbation fields $(V_{a_k})$ and $(V_{b_k})$ are explicitly given on the boundary of $\Om$ as follows:
 \begin{equation*}
    \left\{
    \begin{array}{l}
    V_{a_k}(x_\theta,y_\theta) = \big(\cos{(k\theta)}\cos{\theta}+k \sin(k\theta)\sin{\theta},\cos(k\theta) \sin\theta - k \sin(k\theta) \cos{\theta}\big),\ \ \ \ \text{where $k\in \llbracket 0,N\rrbracket$}
    \vspace{2mm}\\
    V_{b_k}(x_\theta,y_\theta) = \big(\sin{(k\theta)}\cos{\theta}+k \cos(k\theta)\sin{\theta},\sin(k\theta) \sin\theta - k \cos(k\theta) \cos{\theta}\big),\ \ \ \ \text{where $k\in \llbracket 1,N\rrbracket$}
    \end{array}
    \right.
\end{equation*}

By applying \eqref{eq:shape_derivative_area}, we have the following formulas for the shape derivatives of the area functional $A$ in the directions $(V_{a_k})$ and $(V_{b_k})$:
\begin{equation*}
    \left\{
    \begin{array}{l}
    A'(\Om,V_{a_0})= 2\pi a_0
    \vspace{2mm}\\
    A'(\Om,V_{a_k})=\pi(1-k^2)a_k,\ \ \ \ \text{where $k\in \llbracket 1,N\rrbracket$} \vspace{2mm}\\
    A'(\Om,V_{b_k})=\pi(1-k^2)b_k,\ \ \ \ \text{where $k\in \llbracket 1,N\rrbracket$}  
    \end{array}
    \right.
\end{equation*}
We note that the derivative vanishes for the perturbations $V_{a_1}$ and $V_{b_1}$ since they correspond to translations. 

Similarly to \cite[Section 5.1.1]{antunes2019parametric}, by using \eqref{eq:shape_derivative_dirichlet} and the change of variables in \eqref{eq:change_of_variables}, we obtain formulas for the shape derivatives of the Dirichlet eigenvalue as follows
\begin{equation*}
    \left\{
    \begin{array}{l}
    \lambda_1'(\Om,V_{a_k})=-\int_0^{2\pi} |\nabla u_1|^2(x_\theta,y_\theta)\cos{(k\theta)}\big(h''_\Om(\theta)+h_\Om(\theta)\big)d\theta,\ \ \ \ \text{where $k\in \llbracket 0,N\rrbracket$}, \vspace{3mm}\\
    \lambda_1'(\Om,V_{b_k})=-\int_0^{2\pi} |\nabla u_1|^2(x_\theta,y_\theta)\sin{(k\theta)}\big(h''_\Om(\theta)+h_\Om(\theta)\big)d\theta,\ \ \ \ \text{where $k\in \llbracket 1,N\rrbracket$}.
    \end{array}
    \right.
\end{equation*}
The computation of the integrals is done by using an order $1$ trapezoidal quadrature and the computation of the eigenfunction $u_1$ is done by using Matlab's toolbox \texttt{PDEtool}.

\subsection{Parametrization via the Gauge function }\label{ss:gauge_function}
\subsubsection{Definition and main properties}
A classical way to parametrize starshaped open sets (in particular convex ones) is by using the so-called \textit{gauge function}, see for example \cite[Section 1.7]{schneider}. 
\begin{definition}
Let $\Om$ a bounded, open subset of $\R^2$, starshaped with respect to the origin. The gauge function $g_\Om$ is defined on $\R^2$ by
$$\forall x\in \R^2,\ \ \ g_\Om(x)=\inf\{t>0\ |\ \frac{1}{t}x\in \Om\}.$$
The gauge function is positively $1$-homogeneous, so one can equivalently consider the restriction of $g_\Om$ to the interval $[0,2\pi)$. 
\end{definition}
We use polar coordinates representation $(r,\theta)$ for the domains, we then define the gauge function on $\R$ as follows: $$\forall \theta\in [0,2\pi),\ \ \ g_\Om(\theta) = \inf\left\{t>0\ \Big|\ \frac{1}{t}\  \binom{\cos{\theta}}{\sin{\theta}}\in\Om \right\}.$$

The open set $\Om$ is then given by 
$$\Om=\left\{(r,\theta)\in [0,+\infty)\times [0,2\pi)\ \Big|\ r<\frac{1}{g_\Om(\theta)}\right\}.$$

The curvature of the boundary of $\Om$ is given by 
$$\kappa_\Om = \frac{g''_\Om+g_\Om}{\left(1+\left(\frac{g'_\Om}{g_\Om}\right)^2\right)^{\frac{3}{2}}},$$
where the second order derivative is to be understood in the sense of distributions. Thus, as for the support function, the starshaped set $\Om\subset\R^2$ is convex if and only if
$$ g''_\Om+g_\Om\ge 0.$$
Moreover, straight lines in $\partial \Om$ are parameterized by the set $\{g''_\Om +g_\Om = 0\}$, and corners in the boundary are seen as Dirac masses in the measure $g''_\Om +g_\Om$. 

Both the perimeter and area can be expressed via the gauge function as follows:
\begin{Proposition}\label{prop:integral_gauge}
Let $\Om$ a planar set star-shaped with respect to the origin. The perimeter and the area of $\Om$ are given by:
$$P(\Om) = \int_0^{2\pi} \frac{\sqrt{g^2_\Om+g'^2_\Om}}{g^2_\Om}d\theta\ \ \ \text{and}\ \ \ A(\Om) = \frac{1}{2} \int_0^{2\pi} \frac{d\theta}{g_\Om^2}.$$
\end{Proposition}
\color{black}
\begin{proof}
    The gauge function $g_\Om$ is equal to the inverse of the radial function $\rho_\Om$. We have 
    $$A(\Om) = \int_{\theta=0}^{2\pi} \int_{r=0}^{\rho_\Om(\theta)}rdrd\theta = \int_{\theta=0}^{2\pi} \int_{r=0}^{1/g_\Om(\theta)}rdrd\theta = \frac{1}{2}\int_{0}^{2\pi} \frac{1}{g_\Om^2}d\theta$$
    and 
    $$P(\Om) =\int_{\theta = 0}^{2\pi} \sqrt{\left(\frac{d\rho_\Om}{d\theta}\right)^2+\rho_\Om(\theta)^2}d\theta = \int_{\theta = 0}^{2\pi} \sqrt{\left(-\frac{g'_\Om}{g_\Om^2}\right)^2+\frac{1}{g_\Om^2}}d\theta =\int_0^{2\pi} \frac{\sqrt{g^2_\Om+g'^2_\Om}}{g^2_\Om}d\theta.$$
\end{proof}
\color{black}

\subsubsection{Numerical setting}\label{para:num_gauge}
Similarly to the case of support function, in the planar case, we can decompose its gauge function as a Fourier series: 
$$g_\Om(\theta) = a_0 + \sum_{k=1}^\infty (a_k \cos{k\theta}+ b_k\sin{k\theta}),$$
where $(a_n)_n$ and $(b_n)_n$ denote the Fourier coefficients defined by
$$a_0 = \frac{1}{2\pi}\int_0^{2\pi}g_\Om(\theta)d\theta$$
and 
$$\forall k\in \N^*,\ \ \ a_k = \frac{1}{\pi}\int_0^{2\pi} g_\Om(\theta)\cos{(k\theta)}d\theta,\ \ \ b_k = \frac{1}{\pi}\int_0^{2\pi} g_\Om(\theta)\sin{(k\theta)}d\theta.$$

Here also, we look for solutions among truncated functions given in the following space:
$$\mathcal{H}_N:= \left\{\theta\longmapsto a_0 + \sum_{k=1}^N \big(a_k \cos{(k\theta)}+ b_k\sin{(k\theta)}\big) \ \big|\ a_0,\dots,a_N,b_1,\dots,b_N\in \R\right\}.$$

In practice, the computation of the perimeter and the area is done by considering a uniform discretization\\ 
$\left\{\theta_k:=\frac{2k\pi}{M}\ |\ k\in \llbracket 0,M-1\rrbracket\right\}$ of the interval $[0,2\pi)$, with $M$ a positive integer (we take it equal to $200$ for the applications). We then approach the domain $\Om$ by the polygon $\Om_M$ of vertices $A_k\left(\frac{\cos{\theta_k}}{g_\Om(\theta_k)},\frac{\sin{\theta_k}}{g_\Om(\theta_k)}\right)$, where $k\in \llbracket 0,M-1\rrbracket$. The functionals perimeter and area (given as integrals in Proposition \ref{prop:integral_gauge}) are then computed in terms of $(a_k)$ and $(b_k)$ by using an order 1 trapezoidal quadrature: 
\begin{equation*}
\left\{
\begin{array}{l}
  P(\Om) \approx \frac{1}{M}\sum\limits_{k=0}^{M-1} \frac{\sqrt{u^2_\Om(\theta_k)+u'^2_\Om(\theta_k)}}{u^2_\Om(\theta_k)}= \frac{1}{M} \sum\limits_{k=0}^{M-1}  \frac{\sqrt{\left(a_0 + \sum\limits_{p=1}^N \big(a_p \cos{(p\theta_k)}+b_p\sin{(p\theta_k)}\big)\right)^2+\left(\sum\limits_{p=1}^N \big(-p a_p \sin{(p\theta_k)}+ p b_p\cos{(p\theta_k)}\big)\right)^2}}{\left(a_0 + \sum\limits_{p=1}^N \big(a_p \cos{(p\theta_k)}+b_p\sin{(p\theta_k)}\big)\right)^2}, \vspace{4mm}\\
  A(\Om) \approx \frac{1}{M}\sum\limits_{k=0}^{M-1} \frac{1}{u^2_\Om(\theta_k)}= \frac{1}{M}\sum\limits_{k=0}^{M-1} \frac{1}{\left(a_0 + \sum\limits_{p=1}^N \big(a_p \cos{(p\theta_k)}+b_p\sin{(p\theta_k)}\big)\right)^2} .
\end{array}
\right.
\end{equation*}

Here also the convexity is parametrized as in the previous section by linear inequalities involving the coefficients $(a_k)$ and $(b_k)$. 

\subsubsection{Computation of the gradients}
The shape gradients of the area and the perimeter are computed by differentiating the explicit formulae above with respect to the Fourier coefficients. As for the Dirichlet eigenvalue, one has to use the Hadamard formula \eqref{eq:shape_derivative_dirichlet}
with $V$ is a perturbation field corresponding to the perturbation of a Fourier coefficient. Let us investigate the values of such perturbations on the boundary of $\Om$. Let $\phi:\theta\in \R\longmapsto v(\theta)$ be
 a Lipschitz $2\pi$-periodic function, for sufficiently small values of $t>0$, we write: 
 $$\frac{1}{g_\Om+t\phi}=\frac{1}{g_\Om}\left(1+t\frac{\phi}{g_\Om}\right)^{-1}=\frac{1}{g_\Om}\left(1-\frac{\phi}{g_\Om}\cdot t+{o}(t)\right)= \frac{1}{g_\Om}-\frac{\phi}{g_\Om^2}\cdot t+{o}(t).$$
 \textcolor{black}{We recall that for a perturbation $V$, we have $\Om_t = (I+tV)(\Om)$.} Thus, perturbing the gauge function in a direction $\phi$ corresponds to a perturbation field defined on the boundary $\partial\Om=\left\{\left(\frac{\cos{\theta}}{g_\Om(\theta)},\frac{\sin{\theta}}{g_\Om(\theta)}\right)\ |\ \theta\in [0,2\pi]\right\}$ by
 $$ \left(\frac{\cos{\theta}}{g_\Om(\theta)},\frac{\sin{\theta}}{g_\Om(\theta)}\right) \in \partial \Om \longmapsto -\frac{\phi(\theta)}{g_\Om^3(\theta)}\binom{\cos{(\theta)}}{\sin{(\theta)}}\in \R^2.$$
 
 We then deduce that the perturbation fields corresponding to the perturbations of the coefficients $(a_k)$ and $(b_k)$ are given by 
 \begin{equation*}
    \left\{
    \begin{array}{l}
    V_{a_k}\left(\frac{\cos{\theta}}{g_\Om(\theta)},\frac{\sin{\theta}}{g_\Om(\theta)}\right)= -\frac{\cos{(k\theta)}}{g_\Om^3(\theta)}\binom{\cos{(\theta)}}{\sin{(\theta)}},\ \ \ \ \text{where $k\in \llbracket 0,N\rrbracket$}, \vspace{3mm}\\
    V_{b_k}\left(\frac{\cos{\theta}}{g_\Om(\theta)},\frac{\sin{\theta}}{g_\Om(\theta)}\right)= -\frac{\sin{(k\theta)}}{g_\Om^3(\theta)}\binom{\cos{(\theta)}}{\sin{(\theta)}},\ \ \ \ \text{where $k\in \llbracket 1,N\rrbracket$},
    \end{array}
    \right.
\end{equation*}
where $\theta \in [0,2\pi]$.\vspace{2mm}

Once the perturbation fields are known, we use the polygonal approximation $\Om_M$ (introduced above in Paragraph \ref{para:num_gauge}) of the domain $\Om$ to provide a numerical approximation of the shape gradient as follows:

$$\lambda_1'(\Om,V)\approx -\sum\limits_{k=0}^{M-1} |\nabla u_1|^2(x_{I_k},y_{I_k}) \langle V(x_{I_k},y_{I_k}),n_k\rangle d\sigma_k,$$
with the convention $A_M := A_0$ and: 
\begin{itemize}
    \item $d\sigma_k=\sqrt{(x_{A_k}-x_{A_{k+1}})^2+(y_{A_k}-y_{A_{k+1}})^2}$,
    \item $I_k$ is the middle of the segment $[A_k A_{k+1}]$.  
    \item $n_k:= \frac{1}{d\sigma_k}\binom{-(y_{A_k}-y_{A_{k+1}})}{x_{A_k}-x_{A_{k+1}}}$ is the exterior unit vector normal to the segment $[A_k A_{k+1}]$. 
\end{itemize}

\subsection{Polygonal approximation and parametrization via the  vertices}\label{ss:vertices}
In this section, we propose to parametrize a convex set via the coordinates $(x_k,y_k)_{k\in \llbracket 0,M-1\rrbracket}$ of the vertices $A_k$ of a corresponding polygonal approximation denoted by $\Om_M$ (with $M\ge 3$). We assume that the points $(A_k)_{k\in \llbracket 0,M-1\rrbracket}$ form in this order a simple polygon (that is a polygon that does not intersect itself and has no holes) and recall the conventions $A_M:=A_0$ and $A_{-1}:=A_{M-1}$.  

As for the previous cases, we have formulas for the involved geometrical quantities: 
 $$
\left \{
\begin{array}{c @{=} c}
    P(\Om_M)\ \ \ &\ \ \ \sum\limits_{k=0}^{M-1} \sqrt{(x_{k+1}-x_k)^2+(y_{k+1}-y_k)^2}, \vspace{3mm}\\ 
    |\Om_M|\ \ \ &\  \frac{1}{2}\left| \sum\limits_{k=0}^{M-1} x_ky_{k+1}-x_{k+1}y_k \right|\ \ \ \ \ \ \ \ \ \ \ \ \ \ \ \ \ \ \ \vspace{3mm}\\
    d(\Om_M)\ \ \ &\  \max\limits_{i,j} \sqrt{(x_i-x_j)^2+(y_i-y_j)^2}\ \ \ \ \ \ \ \ \ \ \ \ \ \ \ \ \ \ \ 
\end{array}
\right.
$$   

It is easily seen that $\Om_M$ is convex if and only if all the interior angles are less than or equal to $\pi$. By using the cross product, this, in turn, is equivalent to the following quadratic constraints:
$$(x_{k-1}-x_k)(y_{k+1}-y_k)-(y_{k-1}-y_k)(x_{k+1}-x_k)\leq 0,$$
for $k\in \llbracket 0,M-1\rrbracket$, where we used the conventions $x_0 := x_M$, $y_0 := y_M$, $x_{M+1} := x_1$ and $y_{M+1} := y_1$. This characterization of convexity is quite natural and has already been considered in literature, see \cite{MR4089508} for example. 

The gradients of the perimeter, area and convexity constraints (corresponding to the variables $(x_k)$ and $(y_k)$) are directly obtained by differentiating the explicit formulae given above. On the other hand, the gradient of the eigenvalue is computed (as in the previous section) by using Hadamard formula formula \eqref{eq:shape_derivative_dirichlet}.
As for the diameter functional, the shape derivative is computed by using the result of Theorem \ref{th:diametre}.

\begin{center}
 \begin{figure}[ht]
  \centering
    \includegraphics[scale=0.5]{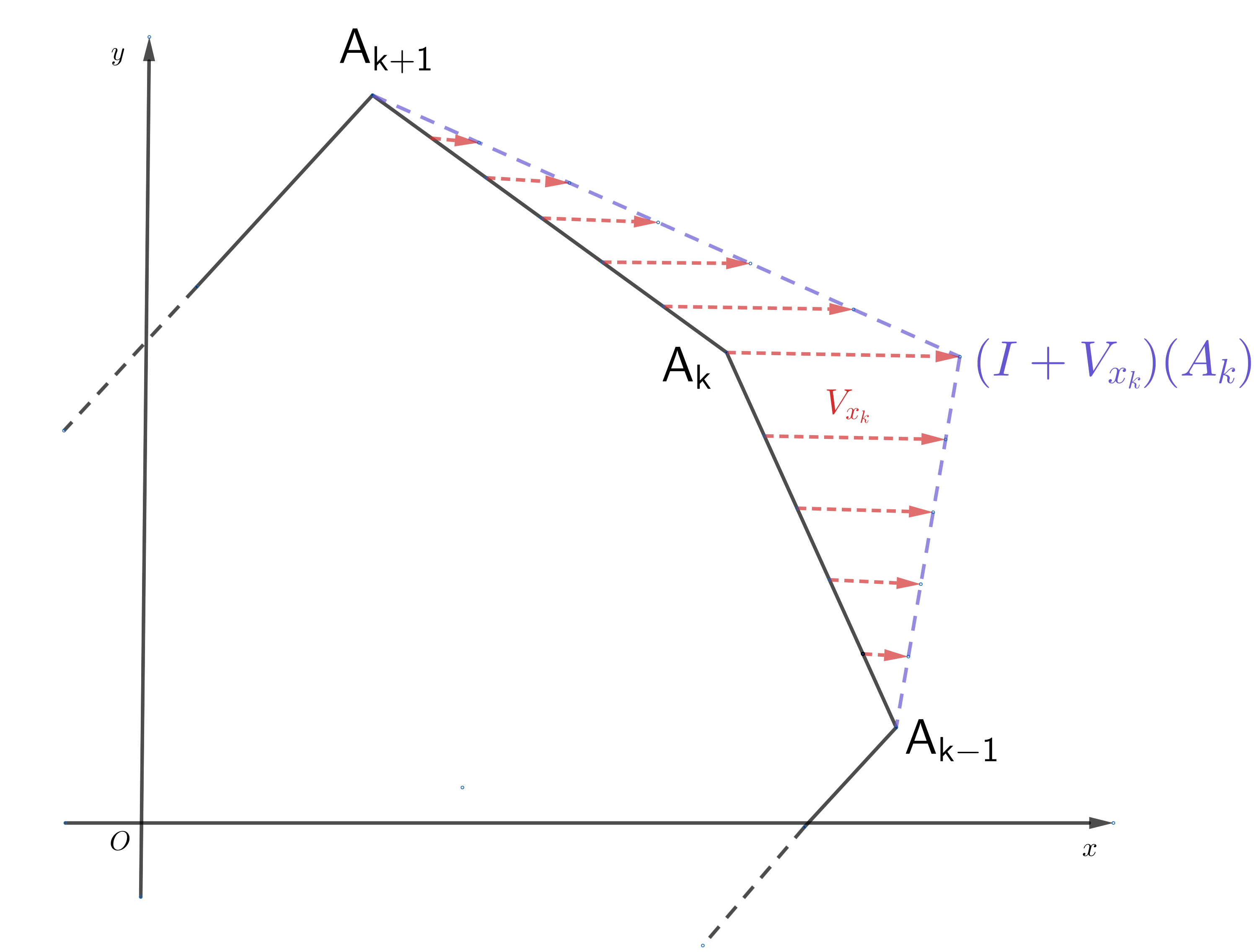}
    \caption{Perturbation field $V_{x_k}$ associated to the perturbation of the parameter $x_k$.}
    \label{fig:per}
\end{figure}
\end{center}

\subsection{Parametrization via the radial function }\label{ss:radial_function}
\subsubsection{Definition and main properties} 
It is common to parametrize star-shaped domains via their \textit{radial function}. In this section, we present this parametrization 
\begin{definition}
Let $\Om\subset \R^2$ be a domain star-shaped with respect to the origin. The radial function $\rho_\Om$ is defined on $\R^2$ by 
$$\forall x\in \R^2,\ \ \ \rho_\Om(x) = \sup\{t>0\ |\ t x\in \Om\}.$$
The radial function is positively $1$-homogeneous, so one can equivalently consider the restriction of $\rho_\Om$ to the interval $[0,2\pi)$. The radial function can then be defined on $[0,2\pi)$ as follows 
$$\forall \theta\in [0,2\pi),\ \ \ \rho_\Om(\theta) = \sup\left\{t>0\ |\ t \binom{\cos{\theta}}{\sin{\theta}}\in \Om\right\}.$$
\end{definition}

If $\Om$ is open, it can be given in polar coordinates as follows 
$$\Om = \left\{(r,\theta)\in [0;+\infty)\times \R\ |\ r<\rho_\Om(\theta)\right\}.$$

We note that the radial function is simply the inverse of the gauge function introduced in Section \ref{ss:gauge_function}. 

\subsubsection{Numerical setting}

Unfortunately, in contrary to the previous cases, convexity cannot be given by linear constraints on the Fourier coefficients of the periodic function $\rho_\Om$. We propose to approximate a set via polygons of vertices $\rho_\Om(\theta_k)\binom{\cos{\frac{2k\pi}{M}}}{\sin{\frac{2k\pi}{M}}}\in \R^2$, where $k\in \llbracket 0, M-1\rrbracket$ and $M$ a sufficiently large integer (in practice we take $M=200$). 

Thus, a star-shaped set $\Om$ will be parametrized via $M$ positive distances $(\rho_k)_{k\in \llbracket 1,M\rrbracket}$ that describe a polygonal approximation of $\Om$ given by vertices $A_k=\rho_k\binom{\cos{\frac{2k\pi}{M}}}{\sin{\frac{2k\pi}{M}}}$. We always consider the convention $A_M:=A_0$ and $A_{-1}:=A_{M-1}$ (in particular $\rho_M:=\rho_0$ and $\rho_{-1}:=\rho_{M-1}$).

This setting allows to give good approximations of the involved geometrical functionals (perimeter, area and diameter). we have 
\begin{enumerate}
    \item for the area: 
    $$A(\Om)=\frac{1}{2}\sin{\frac{2\pi}{M}}\cdot \sum_{k=0}^{M-1}\rho_k\rho_{k+1},$$
    \item for the perimeter: 
    $$P(\Om) = \sum_{k=0}^{M-1} \sqrt{\rho_k^2+\rho_{k+1}^2-2\rho_k\rho_{k+1} \cos{\left(\frac{2k\pi}{M}\right)}},$$
    \item and the diameter:
    $$d(\Om) = \max_{i\ne j}\sqrt{\left[\rho_i\cos{\frac{2i\pi}{M}}-\rho_j\cos{\frac{2j\pi}{M}}\right]^2+\left[\rho_i\sin{\frac{2i\pi}{M}}-\rho_j\sin{\frac{2j\pi}{M}}\right]^2},$$
    this formula provides the diameter in $\O(M^2)$ complexity. 
\end{enumerate}

It remains to describe the convexity constraint via the parameters $(\rho_k)_{k\in \llbracket 0,M-1\rrbracket}$: we remark that the polygon (which contains the origin $O$) whose vertices are given by $A_k:=\left(\rho_k\cos{\frac{2k\pi}{M}},\rho_k\sin{\frac{2k\pi}{M}}\right)$ is convex if and only if the sum of the areas of the triangles $OA_kA_{k+1}$ and $OA_kA_{k-1}$ is greater or equal than the area of $OA_{k-1}A_{k+1}$, see Figure \ref{fig:convex_area}. 
\begin{center}
 \begin{figure}[ht]
  \centering
    \includegraphics[scale=0.4]{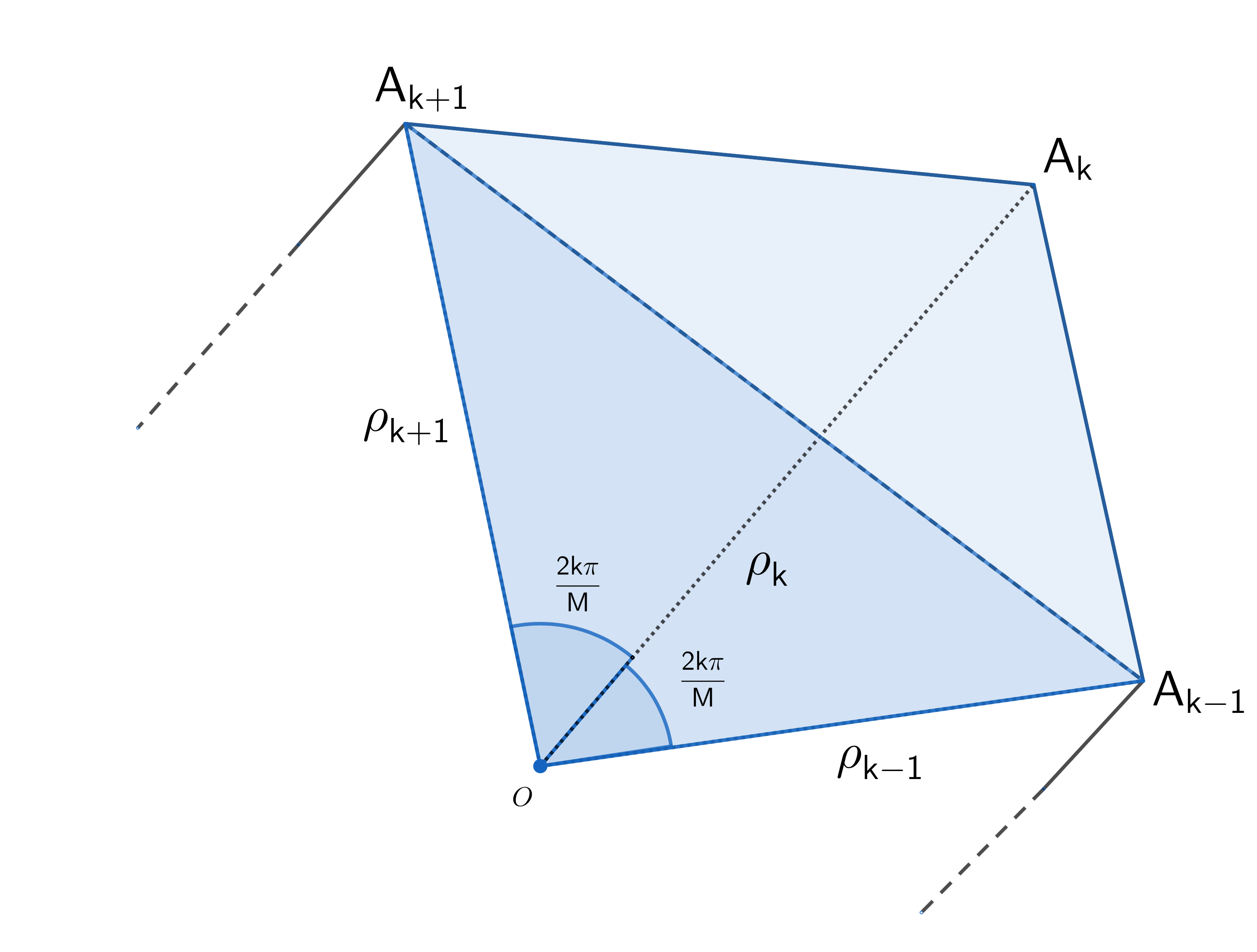}
    \caption{Convexity constraint via areas of the triangles.}
    \label{fig:convex_area}
\end{figure}
   
\end{center}

We have
\begin{equation*}
\left\{
\begin{array}{l}
  \mathcal{S}_{OA_{k-1}A_k} = \frac{1}{2}\rho_{k-1}\rho_k\sin{\frac{2k\pi}{M}}\vspace{2mm}\\
  \mathcal{S}_{OA_kA_{k+1}} = \frac{1}{2}\rho_{k}\rho_{k+1}\sin{\frac{2k\pi}{M}}\vspace{2mm}\\
  \mathcal{S}_{OA_{k-1}A_{k+1}} = \frac{1}{2}\rho_{k-1}\rho_{k+1}\sin{\frac{4k\pi}{M}}= \rho_{k-1}\rho_{k+1}\sin{\frac{2k\pi}{M}}\cos{\frac{2k\pi}{M}}.
\end{array}
\right.
\end{equation*}

Thus, the convexity constraint given by $\mathcal{S}_{OA_{k-1}A_k}+\mathcal{S}_{OA_kA_{k+1}}\ge \mathcal{S}_{OA_{k-1}A_{k+1}} $ is equivalent to the following quadratic constraint: 
\begin{equation}\label{ineq:conv_radial}
C_k:=2\cos{\left(\frac{2\pi}{M}\right)}\rho_{k-1}\rho_{k+1}-\rho_k(\rho_{k-1}+\rho_{k+1})\leq 0,    
\end{equation}
where $k\in \llbracket 0,M-1\rrbracket$. 

\subsubsection{Computation of the gradients}
Now that we brought the shape optimization problem to a finite dimensional optimization one, it remains to compute the gradients of the involved functionals and constraints. 

Let us take $\Om\subset \R^2$ a domain starshaped with respect to the origin $O$ and parametrized by $(\rho_k)_{k\in \llbracket 0,M-1\rrbracket}$.
For every $k\in \llbracket 0,M-1\rrbracket$, we denote by $V_{\rho_k}$ the perturbation field corresponding to the perturbation of the variable $\rho_k$. It is null on the whole boundary except on the sides $[A_{k-1}A_k]$ and $[A_k A_{k+1}]$, see Figure \ref{fig:perturbation_radial}.

\begin{center}
 \begin{figure}[ht]
  \centering
    \includegraphics[scale=0.4]{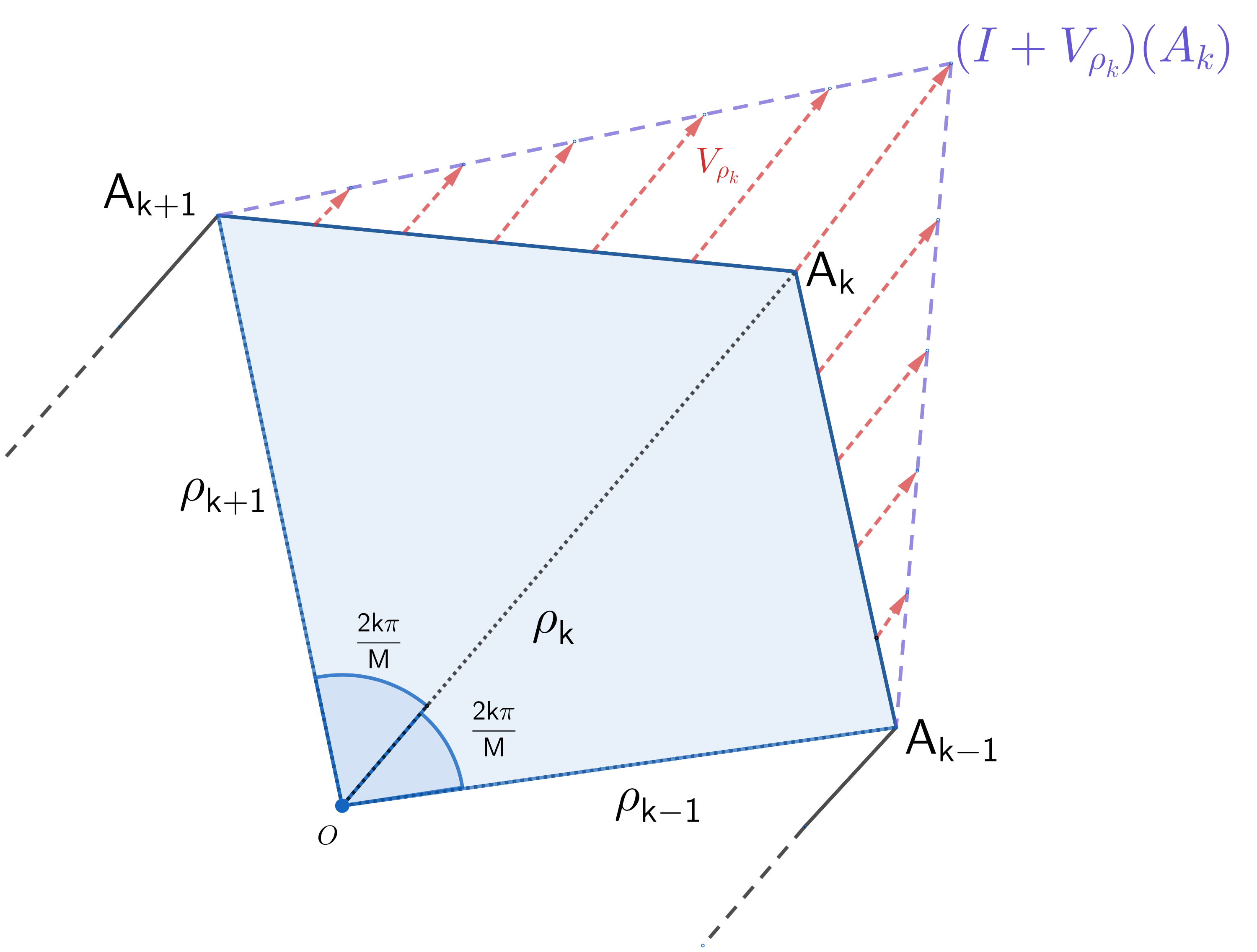}
    \caption{Perturbation field $V_{\rho_k}$.}
    \label{fig:perturbation_radial}
\end{figure}
\end{center}

Since we dispose of explicit formulae for the perimeter and the area, we can directly compute the corresponding shape gradients. We have for every $k\in \llbracket 0,M-1\rrbracket$,
$$A'(\Om,V_{\rho_k})= \frac{\sin{\left(\frac{2\pi}{M}\right)}}{2}\cdot (\rho_{k-1}+\rho_{k+1}),$$
and
$$P'(\Om,V_{\rho_k})= \frac{\rho_k-\rho_{k-1}\cos{\left(\frac{2\pi}{M}\right)}}{\sqrt{\rho_{k-1}^2+\rho_k^2-2\rho_{k-1}\rho_k\cos{\left(\frac{2\pi}{M}\right)}}}+\frac{\rho_k-\rho_{k+1}\cos{\left(\frac{2\pi}{M}\right)}}{\sqrt{\rho_{k+1}^2+\rho_k^2-2\rho_{k+1}\rho_k\cos{\left(\frac{2\pi}{M}\right)}}}.$$

For the eigenvalue, we use as before Hadamard's formula $$\lambda_1'(\Om,V)= -\int_{\partial \Om} |\nabla u_1|^2\langle V,n\rangle d\sigma,$$
where $u_1\in H^1_0(\Om)$ is a normalized eigenfunction (i.e. $\|u_1\|_2=1$) corresponding to $\lambda_1(\Om)$ and $V$ is a perturbation field. 

For every $k\in \llbracket 0,M-1\rrbracket$, we discretize the side $[A_k A_{k+1}]$ in $\ell$ small segments of length $\frac{A_k A_{k+1}}{\ell}$ centered in some points $B_k^i\in [A_k A_{k+1}]$. We then compute approximations of the gradients as follows
$$\lambda_1'(\Om,V_{\rho_k})\approx -\sum\limits_{i=1}^\ell \left(|\nabla u_1|^2(x_{B_k^i},y_{B_k^i}) \langle V_{\rho_k}(x_{B_k^i},y_{B_k^i}),n_{k}\rangle \frac{d\sigma_k}{\ell}-|\nabla u_1|^2(x_{B_{k-1}^i},y_{B_{k-1}^i}) \langle V_{\rho_k}(x_{B_{k-1}^i},y_{B_{k-1}^i}),n_{k-1}\rangle \frac{d\sigma_{k-1}}{\ell}\right),$$
with
\begin{itemize}
    \item the conventions $A_M:=A_0$ and $A_{-1}:=A_{M-1}$ (in particular $\rho_M:=\rho_0$ and $\rho_{-1}:=\rho_{M-1}$),
    \item $d\sigma_k=\sqrt{(x_{A_k}-x_{A_{k+1}})^2+(y_{A_k}-y_{A_{k+1}})^2}$,
    \item the points $(B_k^i)_{i\in \llbracket 1,M\rrbracket }$ given by $$\forall i\in \llbracket 1,\ell \rrbracket,\ \ \ \  B_k^i:= \left(1-\frac{i}{2\ell}\right)A_k+ \frac{i}{2\ell} A_{k+1},$$  
    \item $n_k:= \frac{1}{d\sigma_k}\binom{-(y_{A_k}-y_{A_{k+1}})}{x_{A_k}-x_{A_{k+1}}}$ is the exterior unit vector normal to the segment $[A_k A_{k+1}]$. 
\end{itemize}

Finally, for the diameter, we use the following shape derivative formula obtained in Theorem \ref{th:diametre}: 
$$d'(\Om,V)=\max\left\{\left\langle\frac{x-y}{|x-y|},V(x)-V(y)\right\rangle\ \Big|\ x,y \in \Omega,\ \text{such that}\  |x-y|=d(\Om)\right\}.$$

 \color{black}
\begin{remark}
    Although the radial and the gauge functions are related via the simple relation $g_\Om=1/\rho_\Om$, the method presented in this section is different from the one based on the Fourier coefficients of the gauge function presented in \ref{ss:gauge_function}. Indeed, in the present section, we propose to optimize the values $\rho_i=\rho_\Om(\theta_i)$, where $(\theta_i)_{i\in\llbracket 1 , N \rrbracket}$ is a uniform discretization of the interval $[0,2\pi)$. In contrast with the setting of Section \ref{ss:gauge_function}, where the convexity is characterized by linear constraints on the Fourier coefficients, in the present section is characterized by quadratic constraints \eqref{ineq:conv_radial}. 
\end{remark}
\color{black}
 
\subsection{Computations of the functionals and numerical optimization}\label{ss:numerical_computations}
Let us give few words on the numerical computation of the functionals. In all the parametrizations above we dispose of analytical formulas that provide good approximations of the area and the perimeter. Let us give some elements on the computations of the remaining functionals involved in the paper.
\begin{itemize}
    \item The first Dirichlet eigenvalue is computed by the "Partial Differential Equation Toolbox" of Matlab that is based on finite elements methods. 
    \item As explained above, the computation of the diameter depends on the choice of the parametrization: indeed, when parameterizing a convex $\Om$ via its support function $h_\Om$, it is given by $d(\Om)=\max\limits_{\theta\in [0,2\pi]} \big(h_\Om(\theta)+h_\Om(\pi+\theta)\big)$, meanwhile, when using a polygonal approximation, we compute the diameter of the convex hull via a fast method of computation (with complexity $\O(M)$, where $M$ is the number of vertices), which consists of finding all antipodal pairs of points and looking for the diametrical between them. This is classically known as Shamos algorithm \cite{MR805539}.
\end{itemize}

As for the optimization, we used Matlab's \texttt{fmincon} function with the \texttt{interior-point} and/or \texttt{sqp} algorithms. 



\section{Improved numerical description of Blaschke--Santal\'o diagrams}\label{s:bs_diagrams} 
In the present section, we show how the application of the different methods described in Section \ref{s:numerical_shape_optimization_methods} can be used to obtain improved descriptions of the boundaries of the diagrams $\D_1$, $\D_2$ and $\D_3$ introduced in Section \ref{s:introduction}. 

\subsection{The diagram \texorpdfstring{$\D_1$}{D1}  of the triplet \texorpdfstring{$(P,A,d)$}{(P,A,d)}}
\subsubsection{Naive approach and classical results}
We are interested in studying the diagram
$$\D_1:= \{(P(\Om),A(\Om))\ |\ \Om\in \K\ \text{and}\ d(\Om)=1\}.$$  
This diagram is (as far as we know) one of the unsolved diagrams introduced by Santal\'o in \cite{santalo}, but we note that there are quite advanced results on the characterization of its boundary: 
\begin{itemize}
    \item in \cite{MR2410988}, the authors solve the problem corresponding to the upper boundary, namely they prove that the problem $$\max \{A(\Om)\ |\ \Om\in \K,\ P(\Om) = d_0\ \text{and}\ d(\Om)=1\},$$
where $d_0\in (2,\pi]$, is solved by symmetric lenses (that are given by the intersection of two disks with the same radius) of diameter $1$ and perimeter $d_0$.
    \item In \cite{MR1544679}, the author manages to describe the lower boundary of the diagram that corresponds to perimeters $d_0\in (2,3]$, he shows that the optimal domains are given by \textit{subequilateral triangles} (i.e., isosceles triangles whose smaller inner angle is less than $\frac{\pi}{3}$). 
    \item At last, there is the famous Blaschke–Lebesgue's Theorem, named after W. Blaschke and H. Lebesgue, which states that the Reuleaux triangle (see Figure \ref{fig:adp1}) has the least area of all domains of given constant width. It is classical that sets of constant width have the same perimeter, thus in the diagram, those sets fill the vertical line $\left\{\pi\right\}\times [\frac{1}{2}(\pi-\sqrt{3}),\frac{\pi}{4}]$, see Figure \ref{fig:adp1}.
\end{itemize}


In Figure \ref{fig:adp1}, we plot the points corresponding to the extremal sets described above and a cloud of dots obtained by randomly generating $10^5$ polygons whose numbers of sides are in $\llbracket 3,30\rrbracket$. 
\begin{center}
 \begin{figure}[ht]
  \centering
    \includegraphics[scale=0.6]{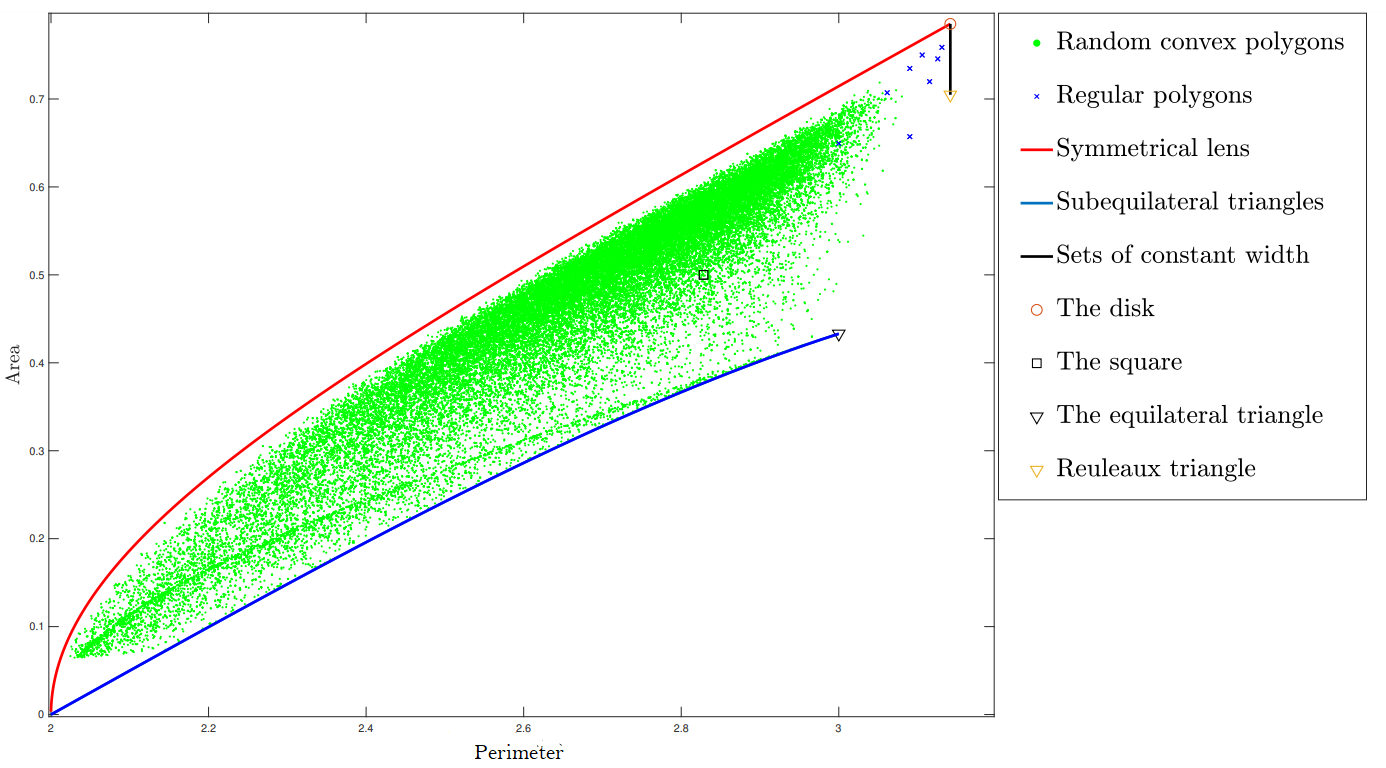}
    \caption{Approximation of the diagram $\D_1$ via random convex sets and some relevant shapes.}
    \label{fig:adp1}
\end{figure}
\end{center}

\subsubsection{On the numerical approximation of the extremal shapes}
We use the methods of Section \ref{s:numerical_shape_optimization_methods}, in order to obtain a numerical approximation of the missing boundary (which should be connecting the equilateral and Reuleaux triangles in Figure \ref{fig:adp1}). 

We numerically solve the following shape optimization problems: 
\begin{equation}\label{prob:minADP}
\min\backslash \max\{A(\Om)\ |\ \Om\in \K,\ P(\Om) = d_0\ \text{and}\ d(\Om) = 1\},    
\end{equation}
where $d_0\in (3,\pi/4)$. 

The parametrization via Fourier coefficients of the support function (Section \ref{ss:support_function}) provides quite satisfying results. Indeed, we obtain symmetrical lenses (see Figure \ref{fig:lens}) as optimal shapes (which is in concordance with the result proved in \cite{MR2410988}). 

\begin{center}
 \begin{figure}[!ht]
  \centering
    \includegraphics[scale=0.35]{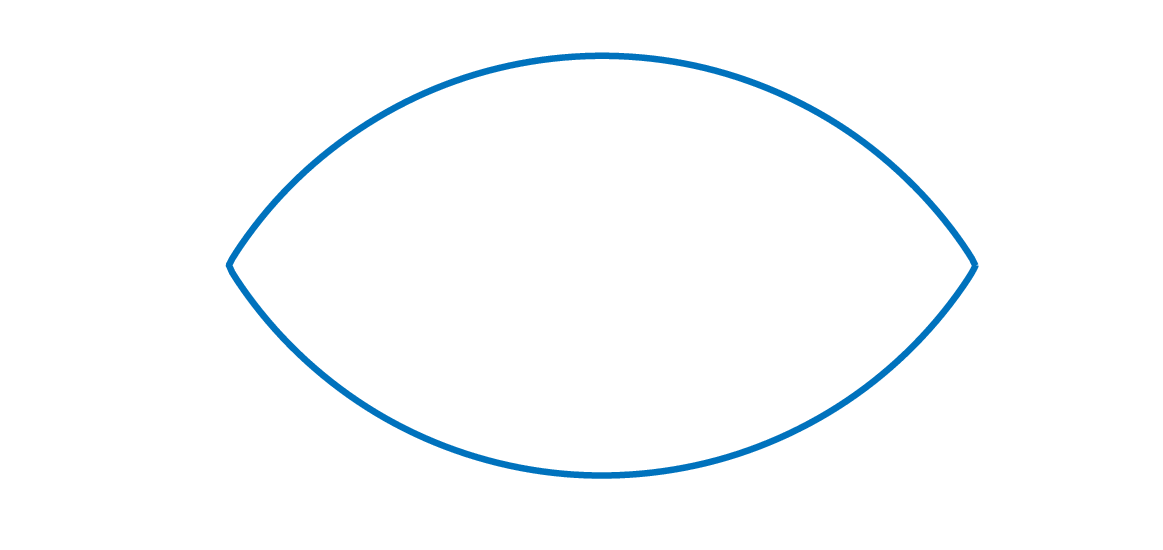}
    \caption{A symmetrical lens obtained as a solution of the problem $ \max\{A(\Om)\ |\ \Om\in \K,\ P(\Om) = 2.4\ \text{and}\ d(\Om) = 1\}$.}
    \label{fig:lens}
\end{figure}
\end{center}

As for the lower boundary, to obtain good approximations, we combine the two methods of sections \ref{ss:support_function} and \ref{ss:radial_function}. We first use the parametrization via Fourier coefficients of the support function truncated at a certain order $N$ to find first approximations of the extremal sets that will be used as initial shapes for the parametrization using radial function. We note that by this process we are able to obtain quite accurate description of the lower boundary, see Figure \ref{fig:improved_apd}. 

In a first time, as explained in section \ref{ss:support_function}, by using the parametrization via the Fourier coefficients of the support function, the Problem  \eqref{prob:minADP} is reduced to the following finite dimensional minimization problem: 
$$\min\limits_{(a_0,...,b_N)\in \R^{2N+1}} \ \ \left(\pi a_0^2 + \frac{\pi}{2}\sum_{k=1}^N (1-k^2)(a_k^2+b_k^2)\right),$$
with linear constraints on the Fourier coefficients: 
\begin{itemize}
    \item a perimeter constraint: $2\pi a_0 = d_0$,
    \item and the convexity constraint: $$\begin{pmatrix}
    1 & \alpha_{1,1} & \cdots & \alpha_{1,N} & \cdots & \beta_{1,1} & \cdots & \beta_{1,N}\\
     \vdots & \vdots & \ddots & \vdots & \cdots & \vdots & \ddots & \vdots\\
    1 & \alpha_{N,1} & \cdots & \alpha_{N,N} & \cdots & \beta_{N,1} & \cdots & \beta_{N,N}
    \end{pmatrix} 
    \begin{pmatrix}
    a_0\\
    a_1\\
     \vdots \\
    a_N\\ 
    b_1\\
    \vdots\\ 
    b_N
    \end{pmatrix} \ge \begin{pmatrix}
    0\\
     \vdots \\
    0
    \end{pmatrix},$$
    where $\alpha_{m,k}=(1-k^2)\cos{k\theta_m}$ and $\beta_{m,k}=(1-k^2)\sin{k\theta_m}$ for $(m,k)\in \llbracket 1,M\rrbracket\times \llbracket 1,N\rrbracket$, with $M$ taken to be equal to 1000. 
\end{itemize}

Before showing the obtained results, let us first analyze the accuracy of the present method (based the support function): we solve the latter optimization problem for different values of the parameter $N$ in the case $d_0=3$ for which we know that the optimal shape is given by the equilateral triangle.  

In Figure \ref{fig:triangles}, we present the optimal shapes obtained for the choices of $N\in \{20,40,100,140\}$:

 \begin{figure}[ht]
 \centering
    \includegraphics[scale=0.2]{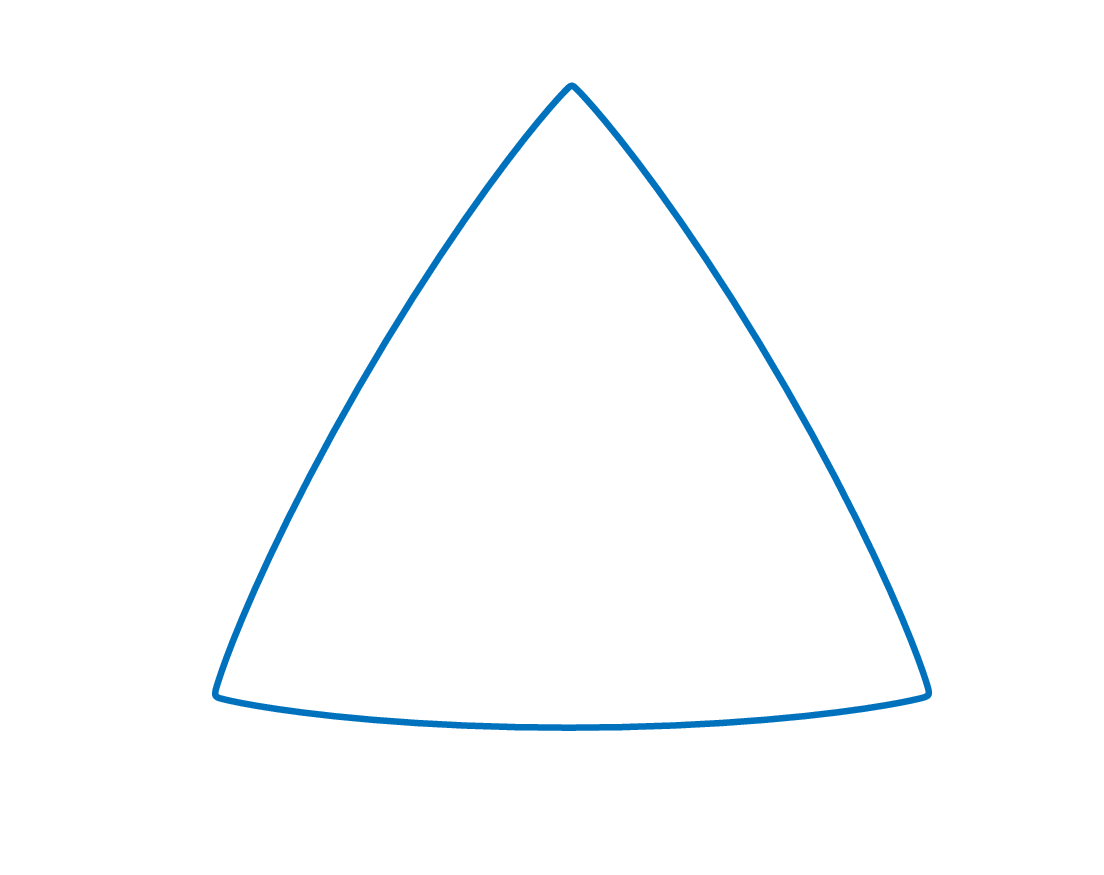} \includegraphics[scale=0.2]{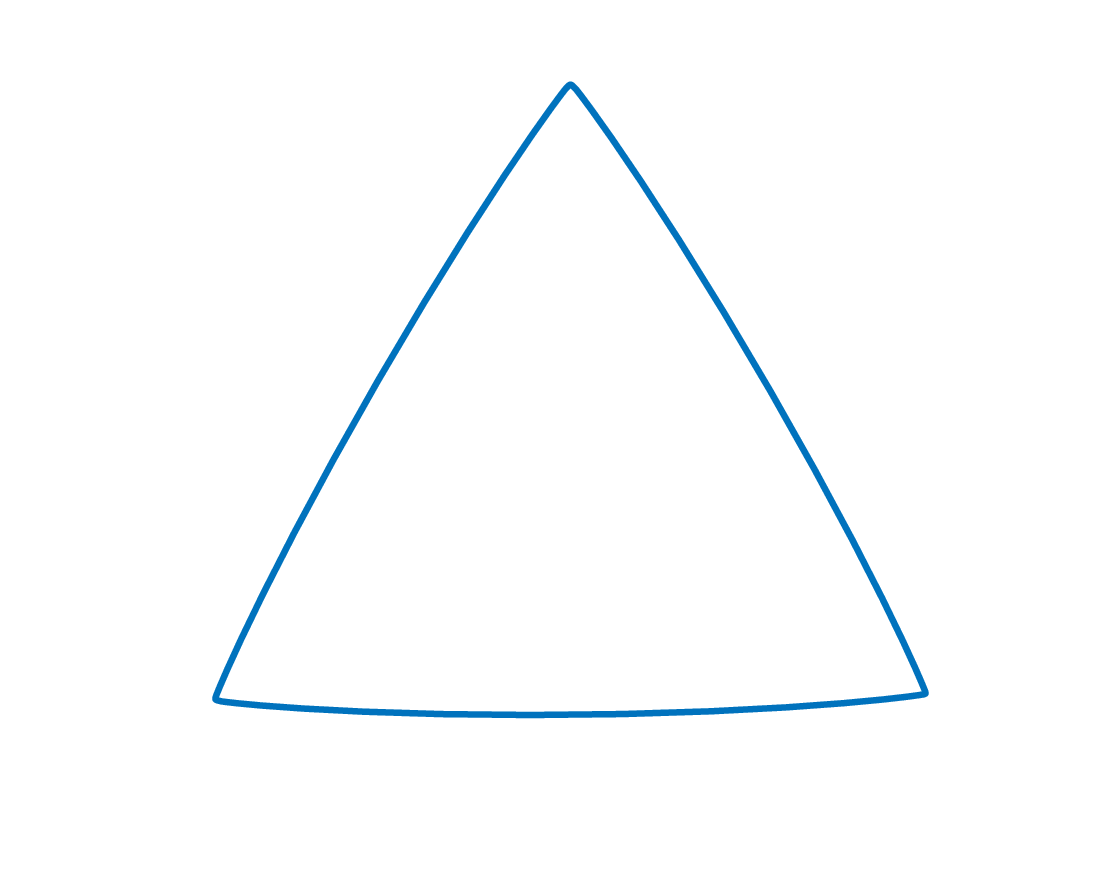}
    \includegraphics[scale=0.2]{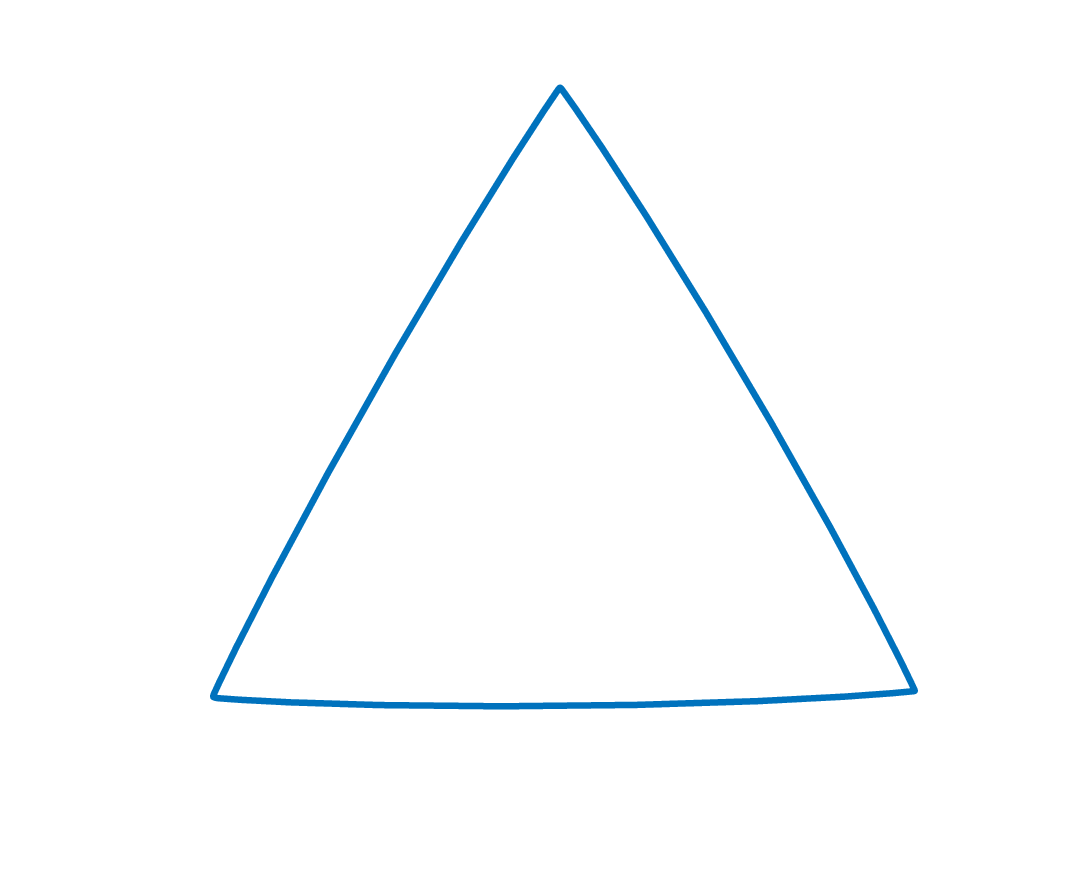}
    \includegraphics[scale=0.2]{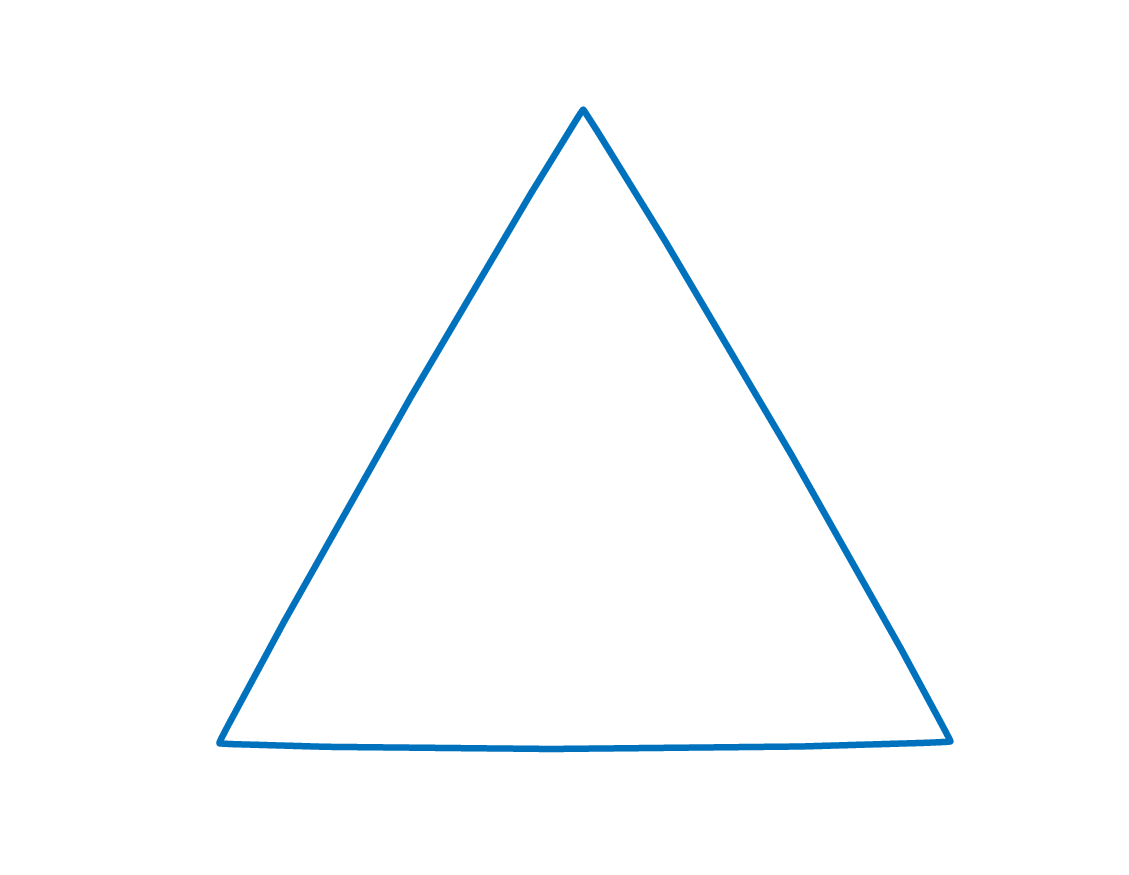}
    \caption{Obtained solutions for $d_0=3$ and $N\in \{20,40,100,140\}$ (approximation of an equilateral triangle).}
    \label{fig:triangles}
\end{figure}

In the Figure \ref{fig:error}, we plot the relative errors  in function of the order of truncation $N$. It shows that the method based on the support function is not very relevant when the optimal shape is polygonal (which can frequently be the case when imposing convexity constraint). 
 \begin{figure}[ht]
 \centering
    \includegraphics[scale=0.3]{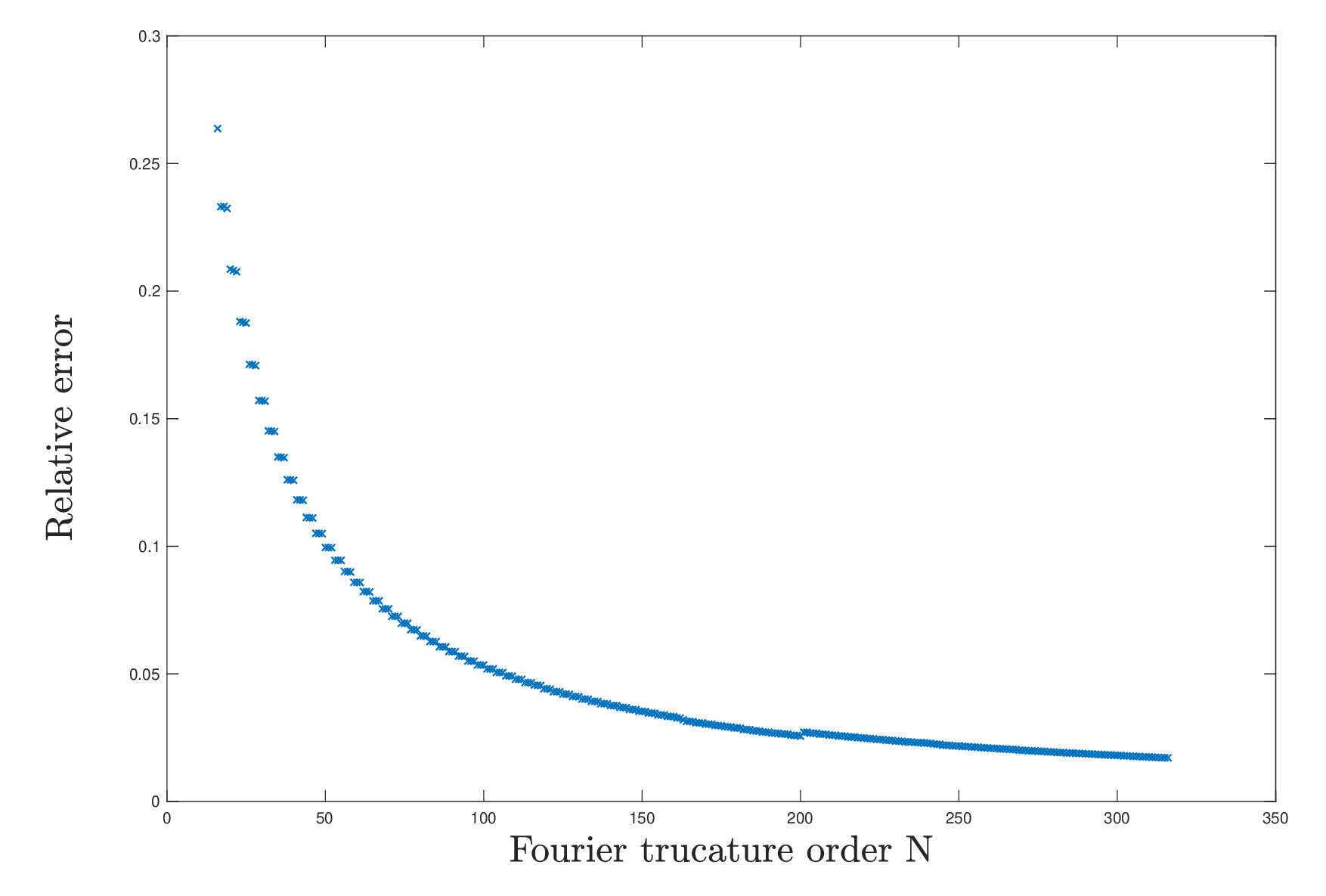}
    \caption{Relative errors in function of the truncation order $N$ in the case $d_0=3$.}
    \label{fig:error}
\end{figure}

We then obtain (see Figure \ref{fig:approx_support}) an approximation of the missing lower boundary corresponding to domains obtained by considering $401$ Fourier coefficients ($N=200$). 
 \begin{figure}[!ht]
 \centering
    \includegraphics[scale=0.5]{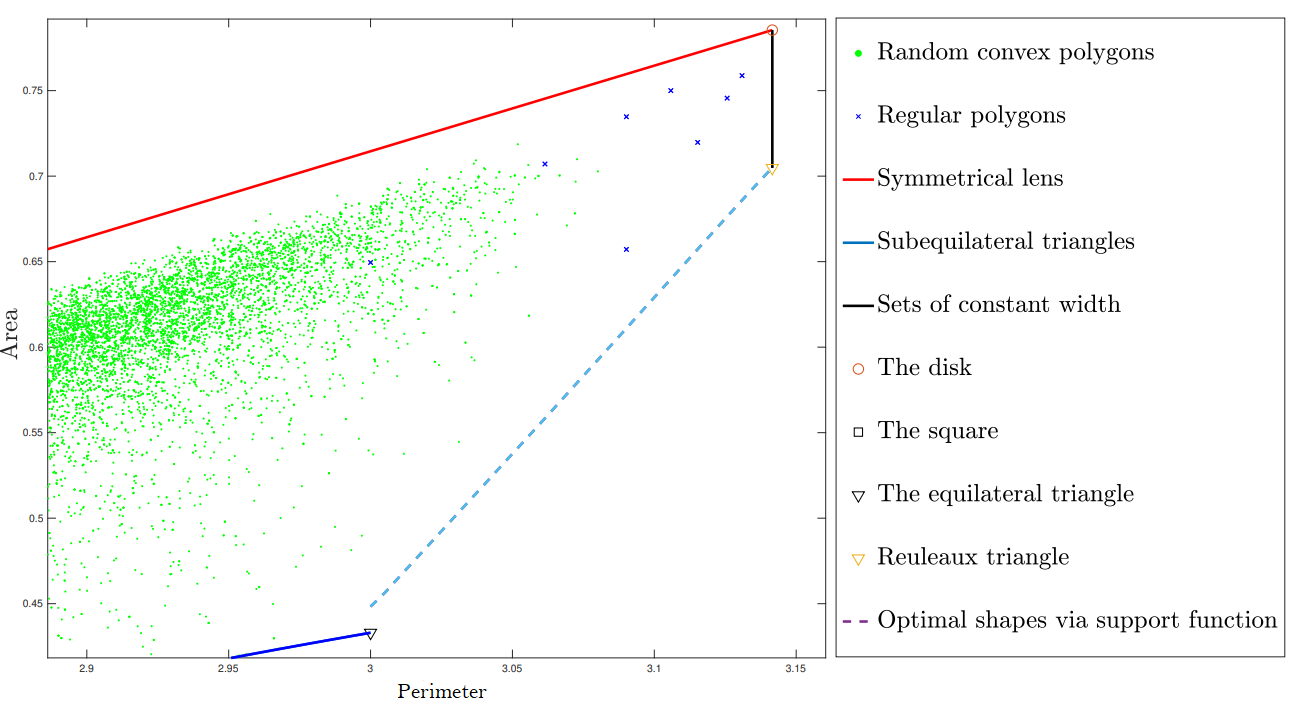}
    \caption{Approximation of the missing part of the lower boundary by optimizing the Fourier coefficients of the support function.}
    \label{fig:approx_support}
\end{figure}

\pagebreak
We then use the domains obtained as initial shapes for the radial function-based method (see Section \ref{ss:radial_function}) and find better shapes that improve the description of the missing part in the lower boundary of the diagram $\D_1$ (see Figures \ref{fig:comparison} and \ref{fig:improved_apd}).

\begin{figure}[!ht]
    \centering
\begin{tabular}{ | m{3cm} | m{4cm}| m{4cm} | } 
  \hline
  Method & Support function & Radial function  \\ 
  \hline
 Obtained shape for $d_0=3.07$ in the minimization problem \eqref{prob:minADP} & \includegraphics[scale=0.22]{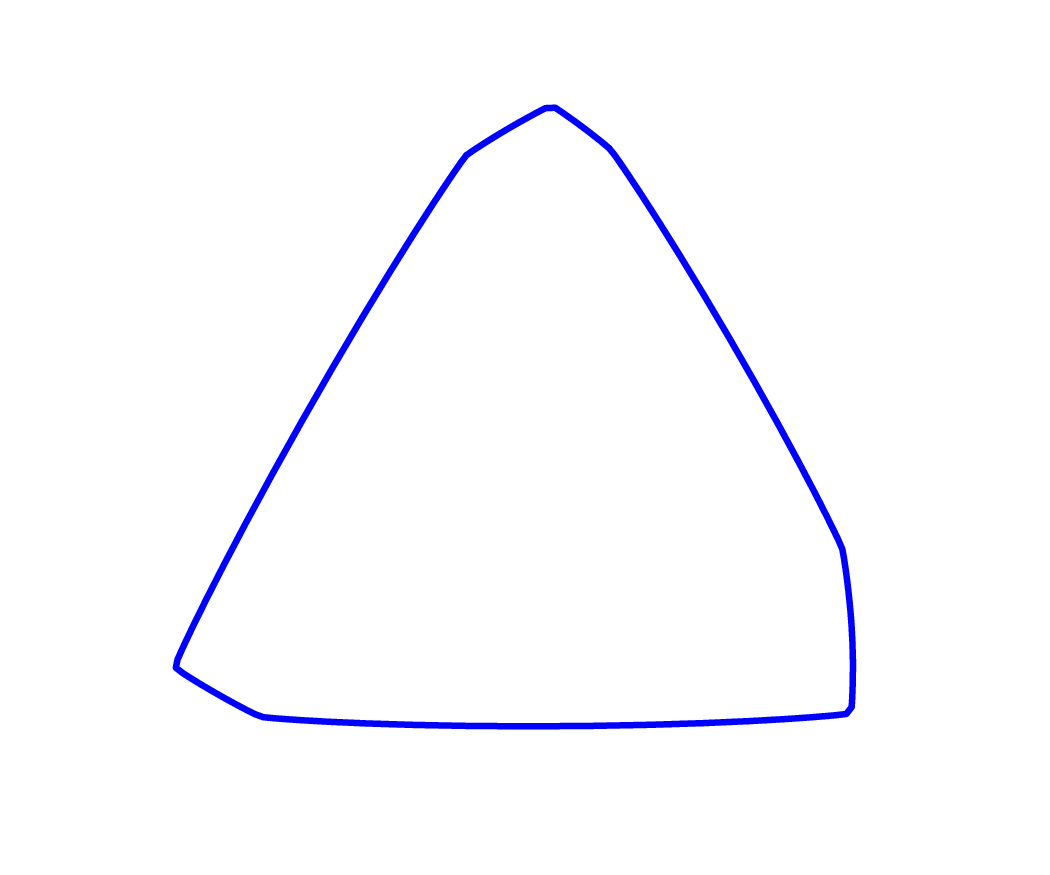} & \includegraphics[scale=0.23]{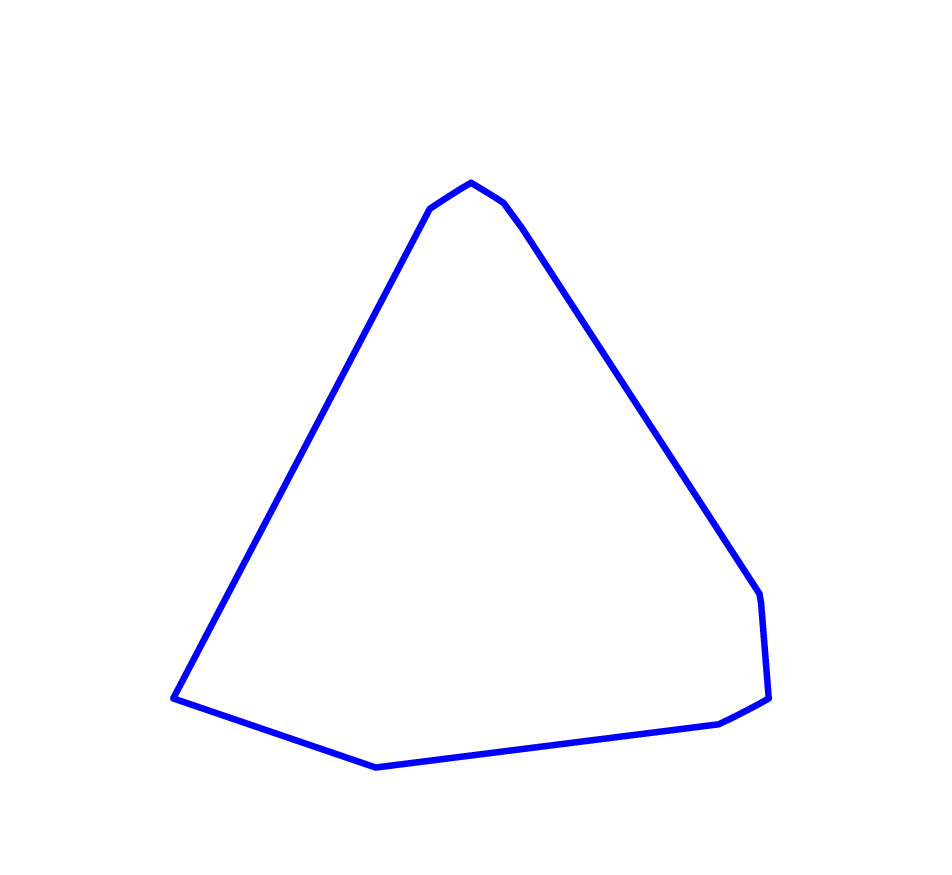} \\ 
  \hline
  Corresponding area &  0.5881 &  0.5687 \\ 
  \hline
\end{tabular}
    \caption{The radial function parametrization allows to improve the result of the support function method.}
    \label{fig:comparison}
\end{figure}

 \begin{figure}[!ht]
 \centering
    \includegraphics[scale=0.27]{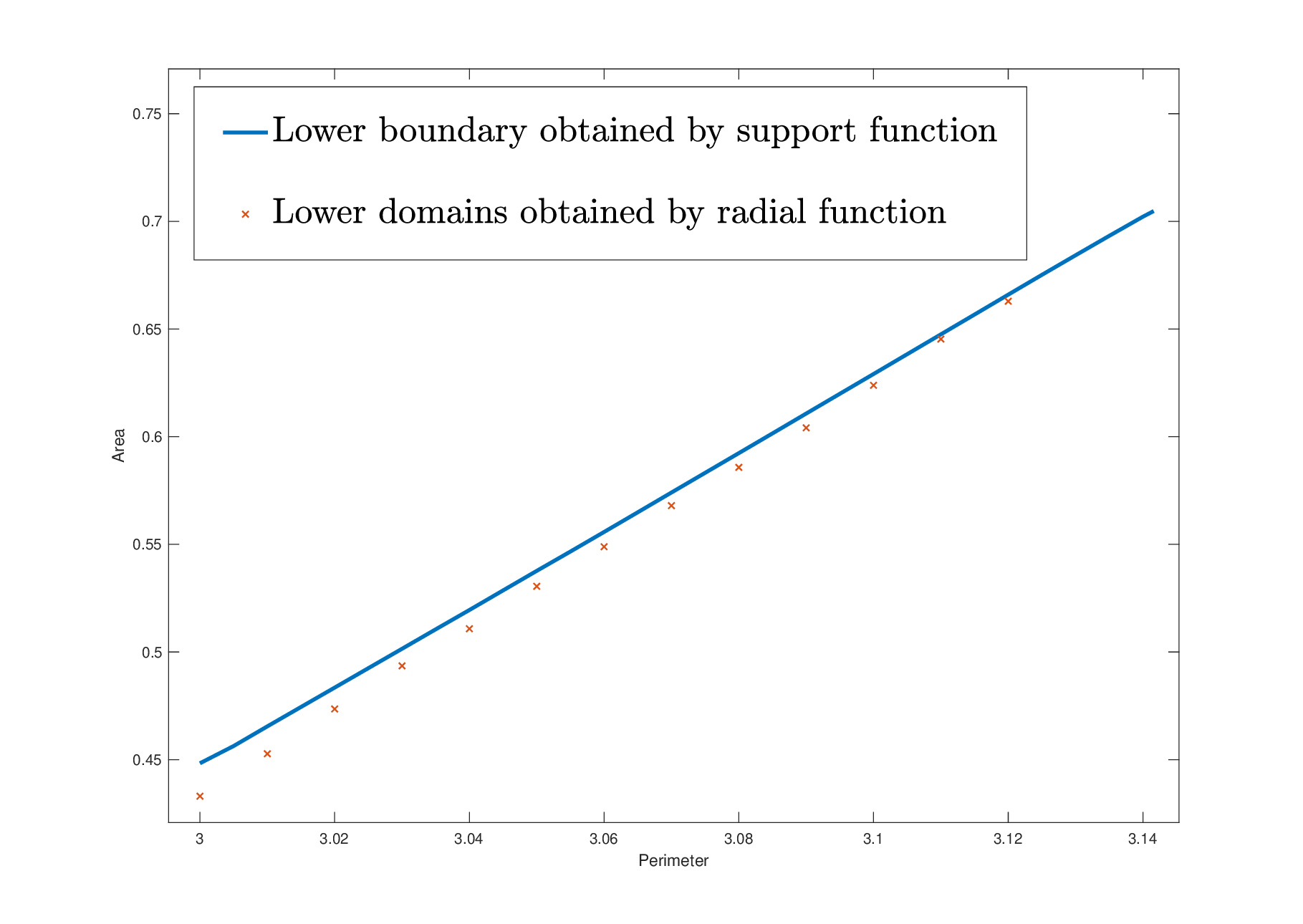}
    \caption{Improved description of the lower boundary by combining the two methods.}
    \label{fig:improved_apd}
\end{figure}

\begin{remark}
We refer to the recent work \cite{bogosel:hal-03607776}, where the author proposes an improvement of the methods based on the Fourier coefficients of the support and gauge functions of convex sets (described in sections \ref{ss:gauge_function} and \ref{ss:support_function}), in such a way to tackle the issue related to the appearance of segments or corners in the optimal shape. 
\end{remark}

\subsubsection{Extremal shapes and improved description of the diagram}

At last, we provide some extremal shapes obtained for relevant values of $d_0$ in Figure \ref{fig:optimal_adp} and improved description of the diagram $\D_1$ in Figure \ref{fig:improved_ADP}. 
\begin{figure}[!ht]
    \centering
\begin{tabular}{|c|c|c|c|c|c|}
\hline
 Problem & $d_0=\pi$ & $d_0=3.07$ & $d_0=3$& $d_0=2.4$\\
\hline
Upper boundary & \includegraphics[scale=0.18]{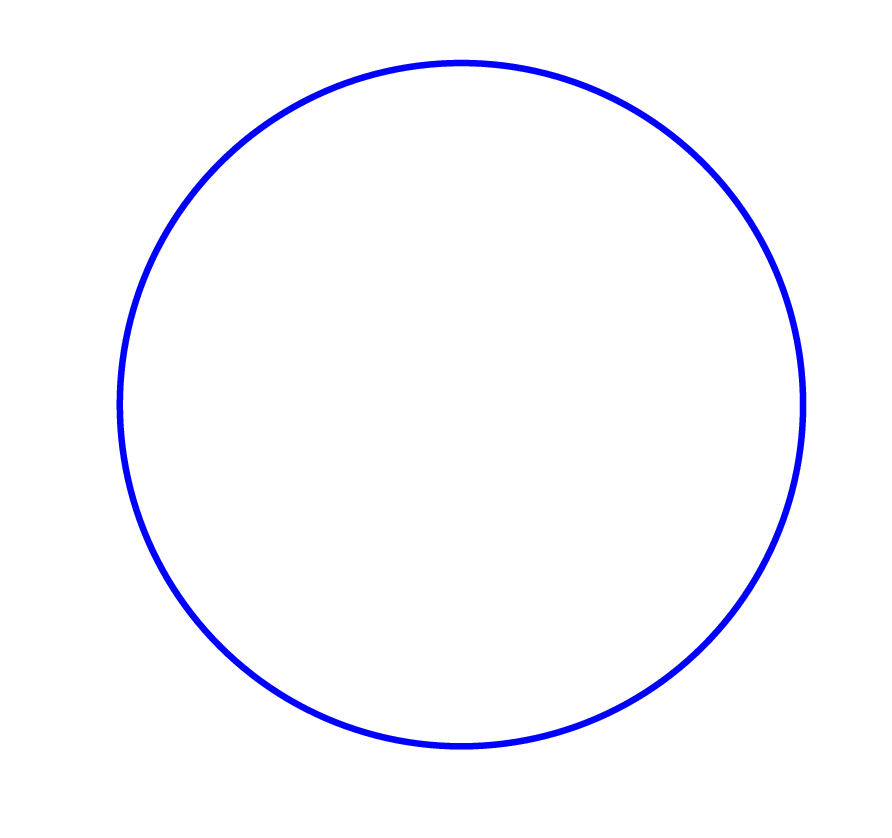} & \includegraphics[scale=0.17]{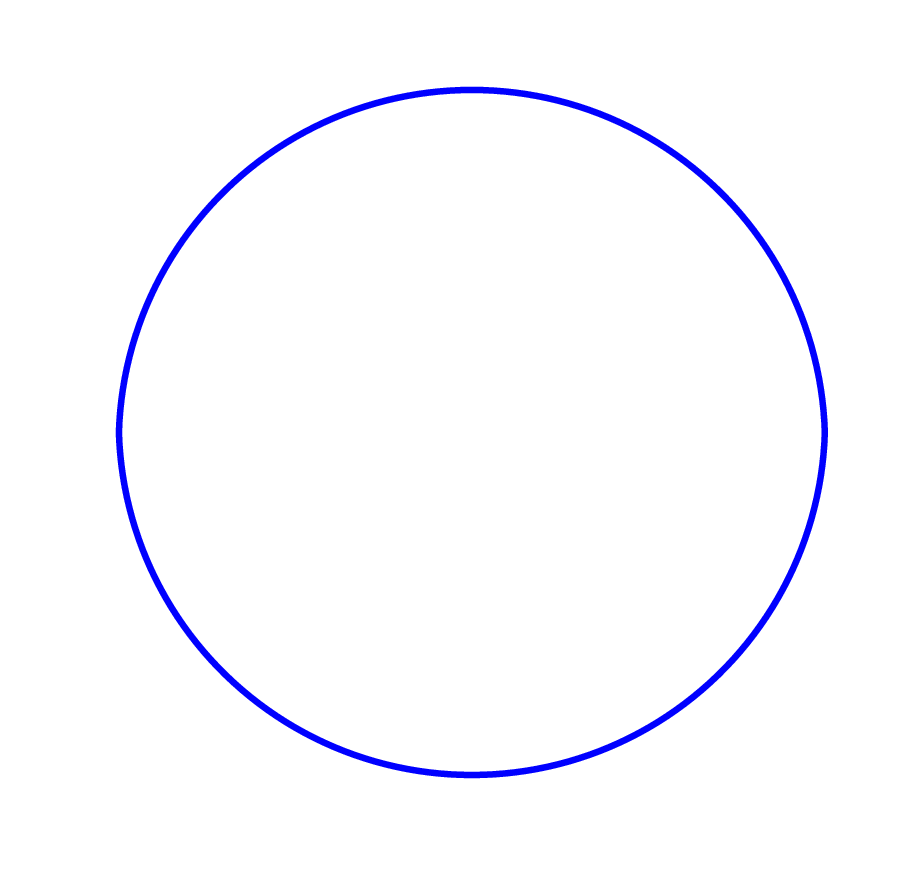}&  \includegraphics[scale=0.17]{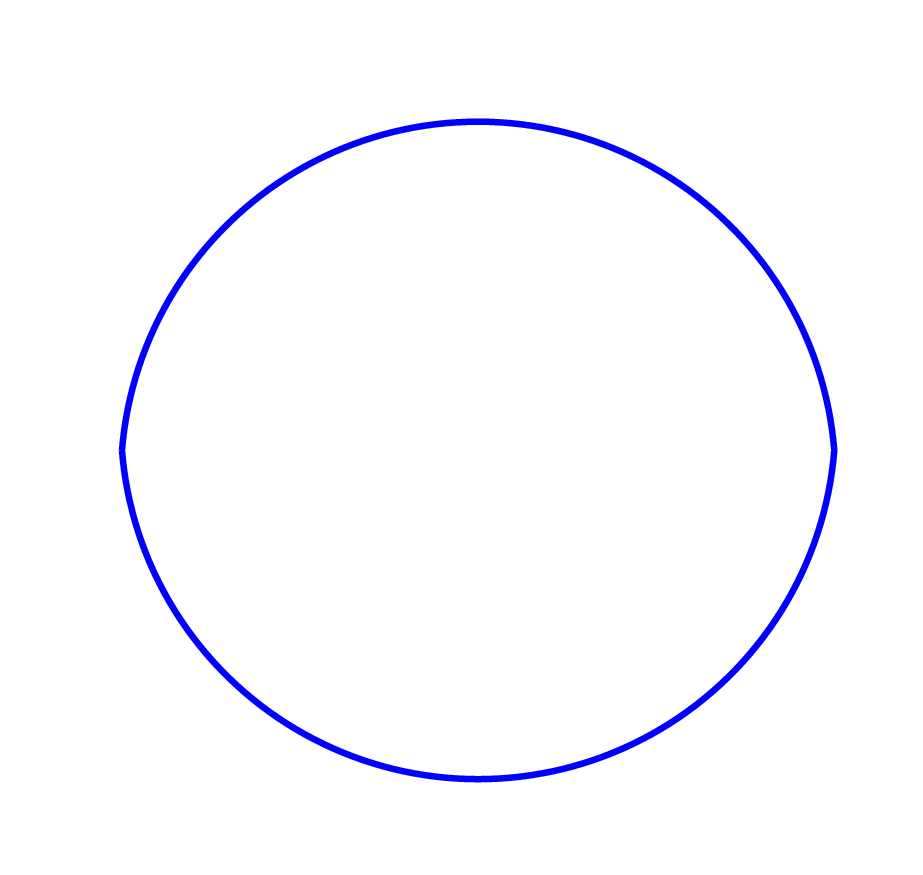}& \includegraphics[scale=0.17]{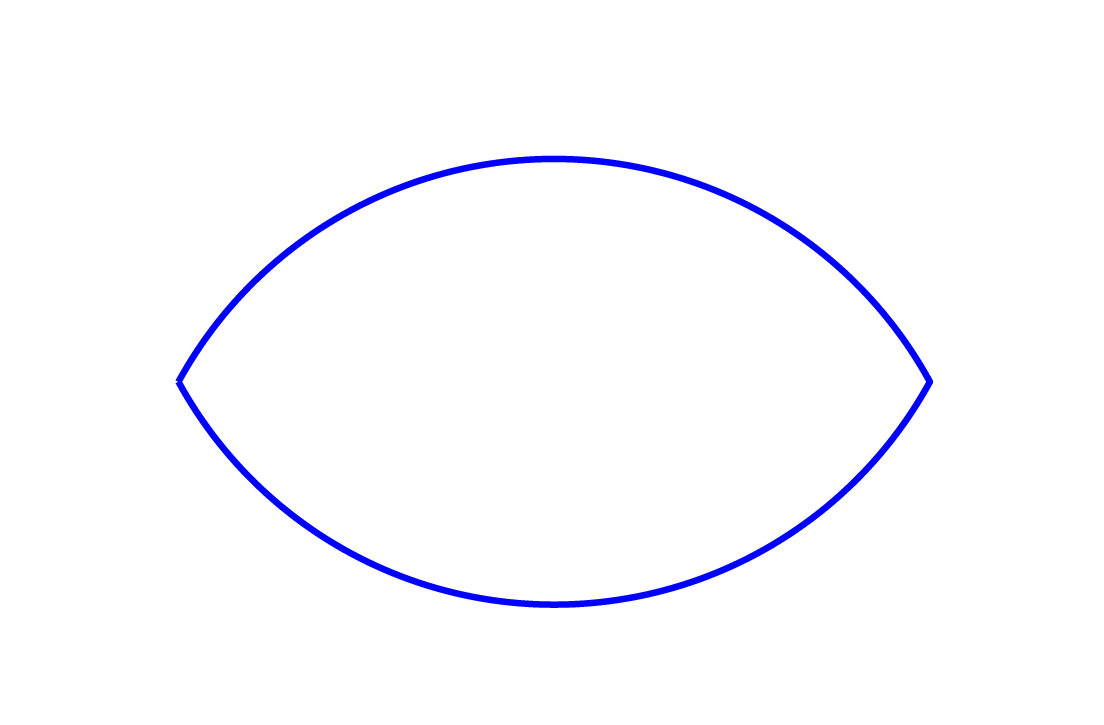} \\
\hline
Lower boundary &\includegraphics[scale=0.2]{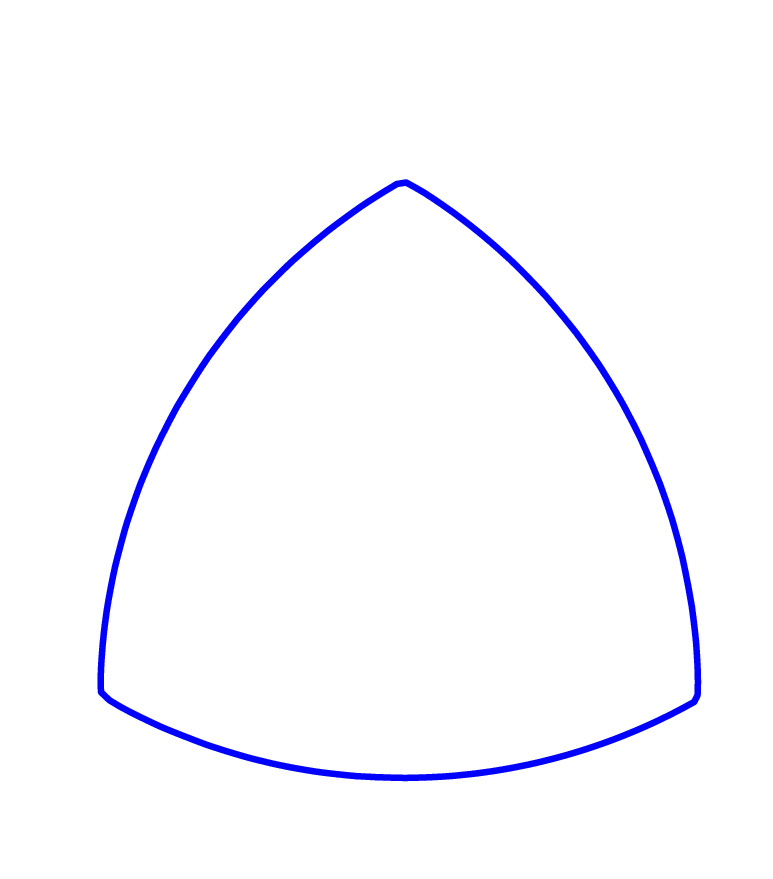}& \includegraphics[scale=0.2]{radial_vs.eps}& \includegraphics[scale=0.135]{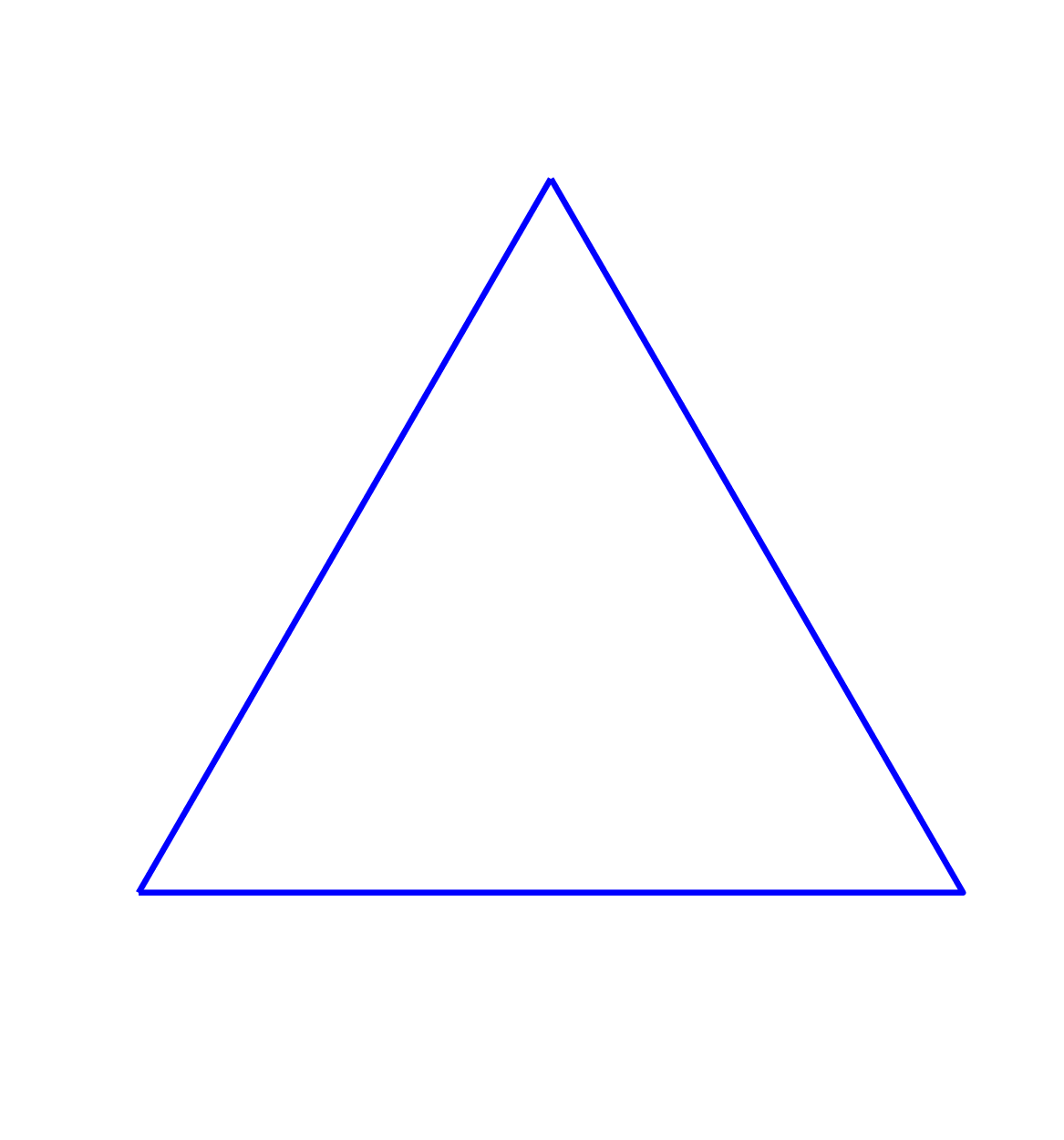}& \includegraphics[scale=0.17]{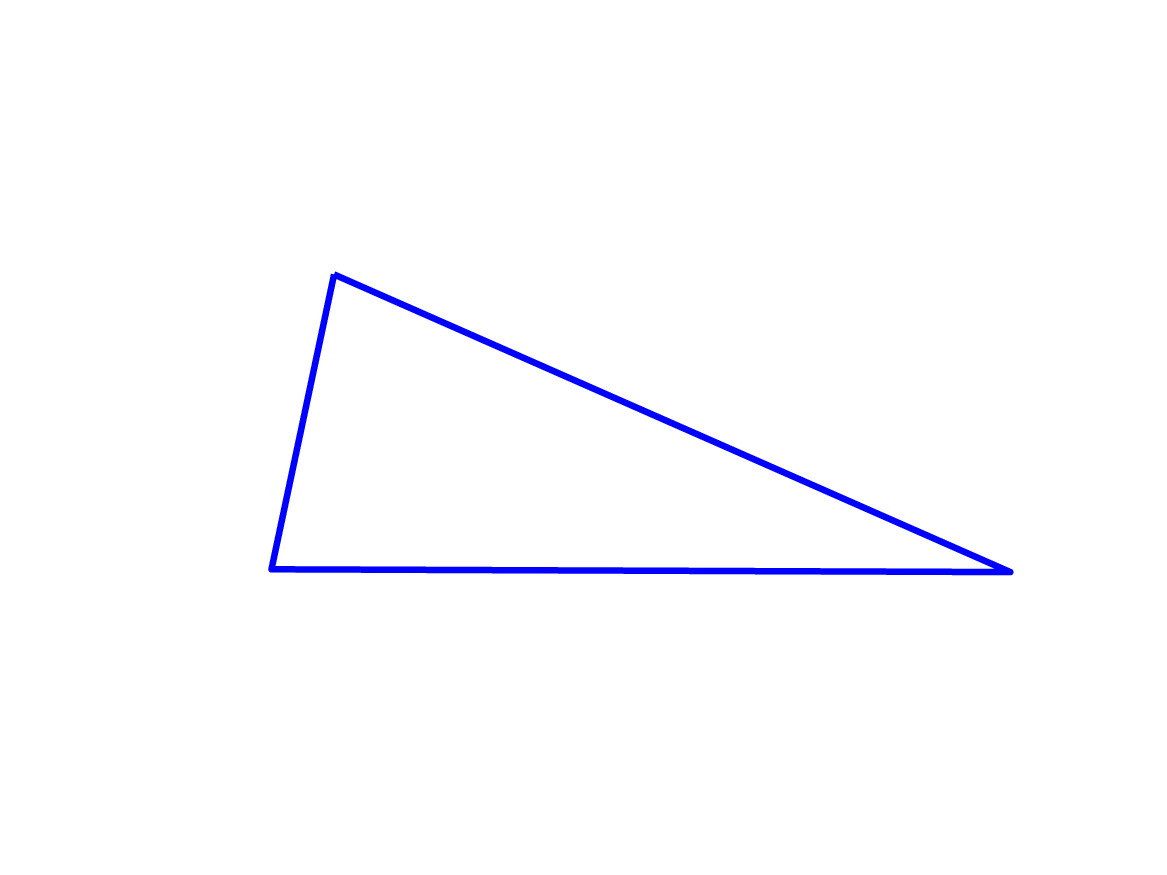}\\
\hline
\end{tabular}
    \caption{Extremal shapes corresponding to different values of $d_0$.}
    \label{fig:optimal_adp}
\end{figure}
 \begin{figure}[!ht]
 \centering
    \includegraphics[scale=0.5]{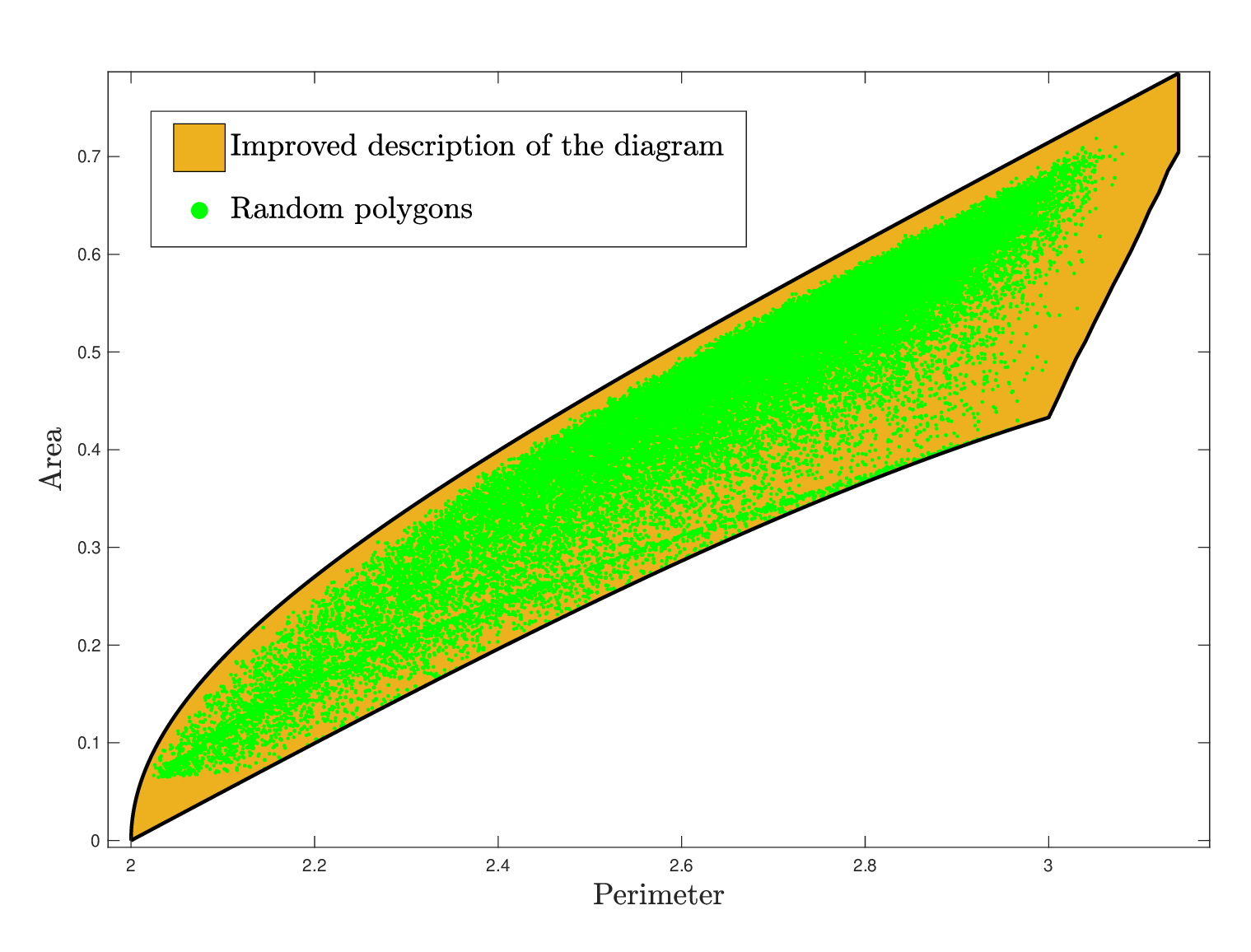}
    \caption{Improved description of the diagram of the triplet $(P,A,d)$.}
    \label{fig:improved_ADP}
\end{figure}

\newpage
\subsection{The diagram \texorpdfstring{$\D_2$}{D2}  of the triplet \texorpdfstring{$(P,\lambda_1,A)$}{(P,lambda,A)}}\label{ss:pla}

The diagram of the triplet $(P,\lambda_1,A)$ is given by
$$\D_2:=\{\big(P(\Om),\lambda_1(\Om)\big)\ |\ \Om \in \K\ \text{and}\ A(\Om)=1\}.$$

This diagram has been first introduced by P. Antunes and P. Freitas in \cite{freitas_antunes} and deeply investigated in \cite{ftouh}.

In a first time, we give an approximation of the diagram by generating $10^5$ random convex polygons (as it was done before in \cite{freitas_antunes}). We obtain the following Figure \ref{fig:random_PLA}: 
 \begin{figure}[!ht]
 \centering
    \includegraphics[scale=0.45]{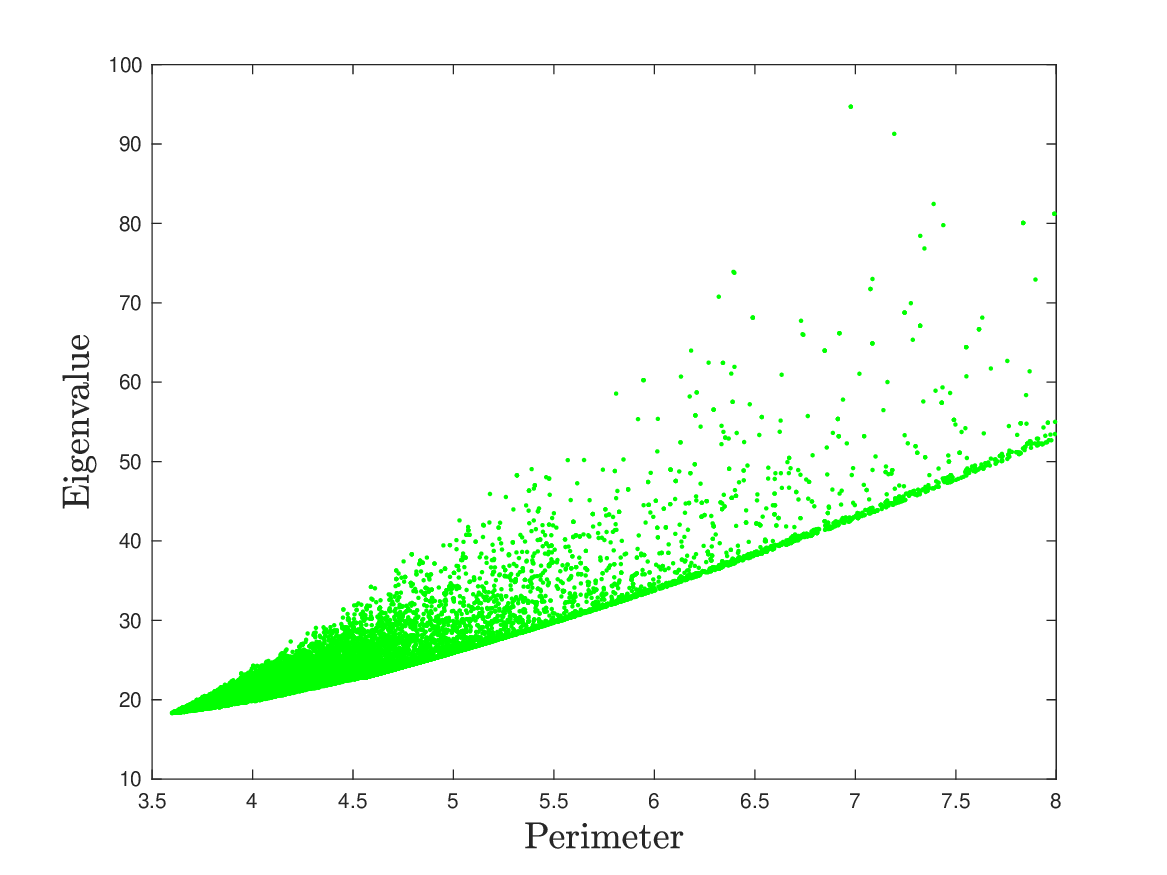}
    \caption{Approximation of the diagram via random convex polygons.}
    \label{fig:random_PLA}
\end{figure}

In order to give a more satisfying approximation of the diagram, we want to find the upper and lower domains and thus have a more accurate description of the boundary of the diagram. We are then led to (numerically) solve the following shape optimization problems:
\begin{equation}\label{prob:pla}
\max\backslash\min\{\lambda_1(\Om)\ |\ \Om\in \K,\ P(\Om) = p_0\ \text{and}\ A(\Om) = 1\},    
\end{equation}

It is shown in \cite[Theorem 1.2]{ftouh} that apart from the disk the domains that correspond to points located on the lower boundary are polygonal, while, the ones corresponding to points of the upper boundary are smooth $C^{1,1}$. We then apply the method based on the coordinates of the vertices described in Section \ref{ss:vertices} for the lower boundary and the other methods for the upper one and obtain quite satisfying results. In Figure \ref{fig:optimal_PLA}, we provide the obtained optimal shapes corresponding to solutions of the problems \eqref{prob:pla} for some relevant values of $p_0$. 

\begin{figure}[!ht]
    \centering
\begin{tabular}{|c|c|c|c|c|c|}
\hline
 Problem & $p_0=P(B)=2\sqrt{\pi}$ & $p_0=3.8$ & $p_0=4$& $p_0=4.2$\\
\hline
Upper boundary & \includegraphics[scale=0.16]{disk.eps} & \includegraphics[scale=0.21]{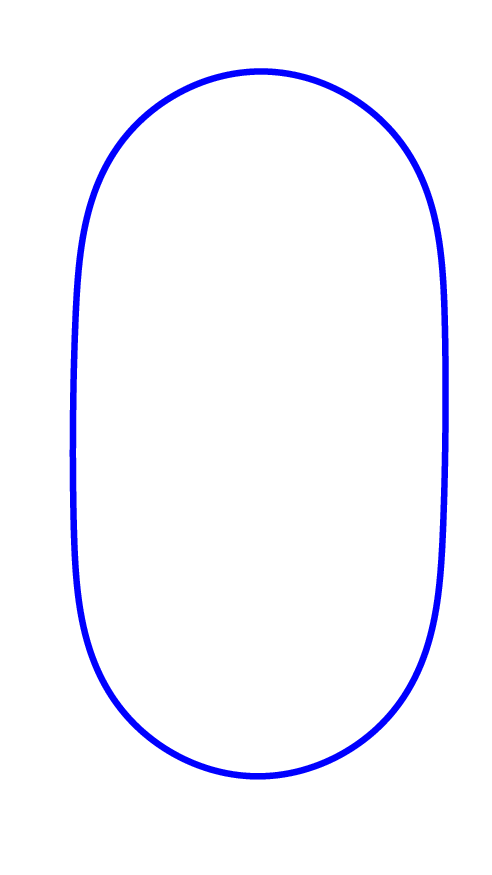}&  \includegraphics[scale=0.24]{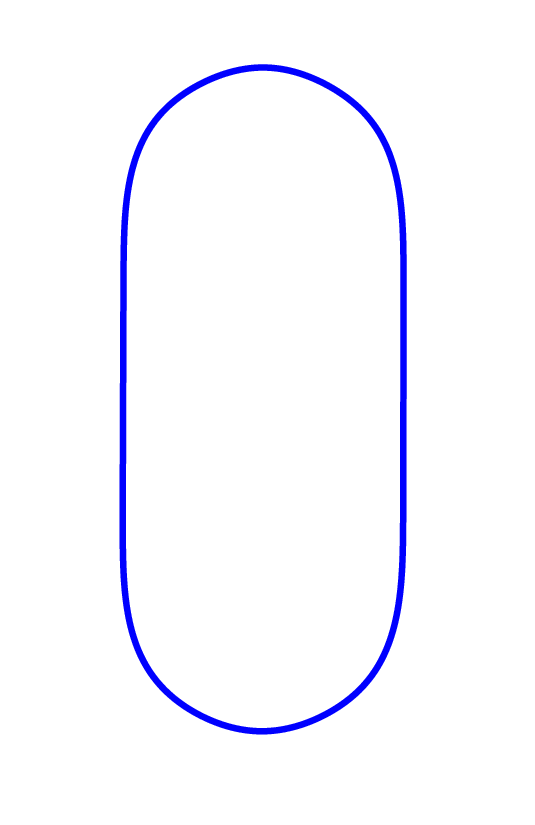}& \includegraphics[scale=0.27]{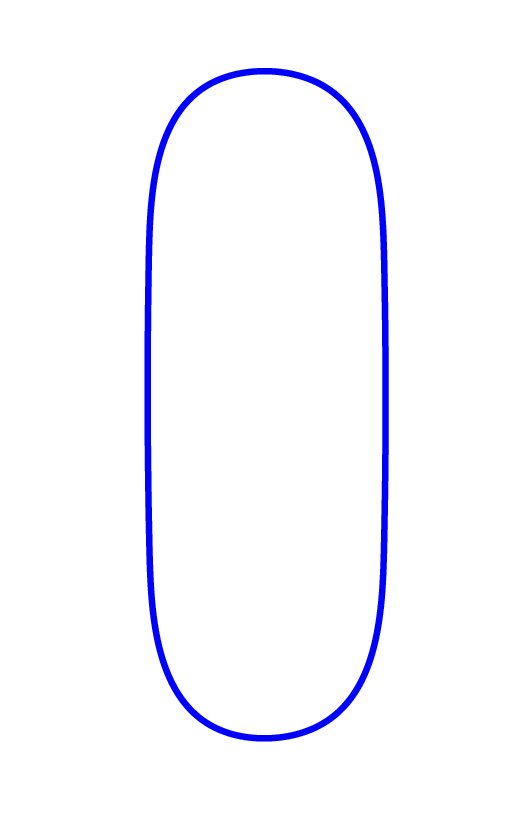} \\
\hline
Lower boundary &\includegraphics[scale=0.16]{disk.eps}& \includegraphics[scale=0.16]{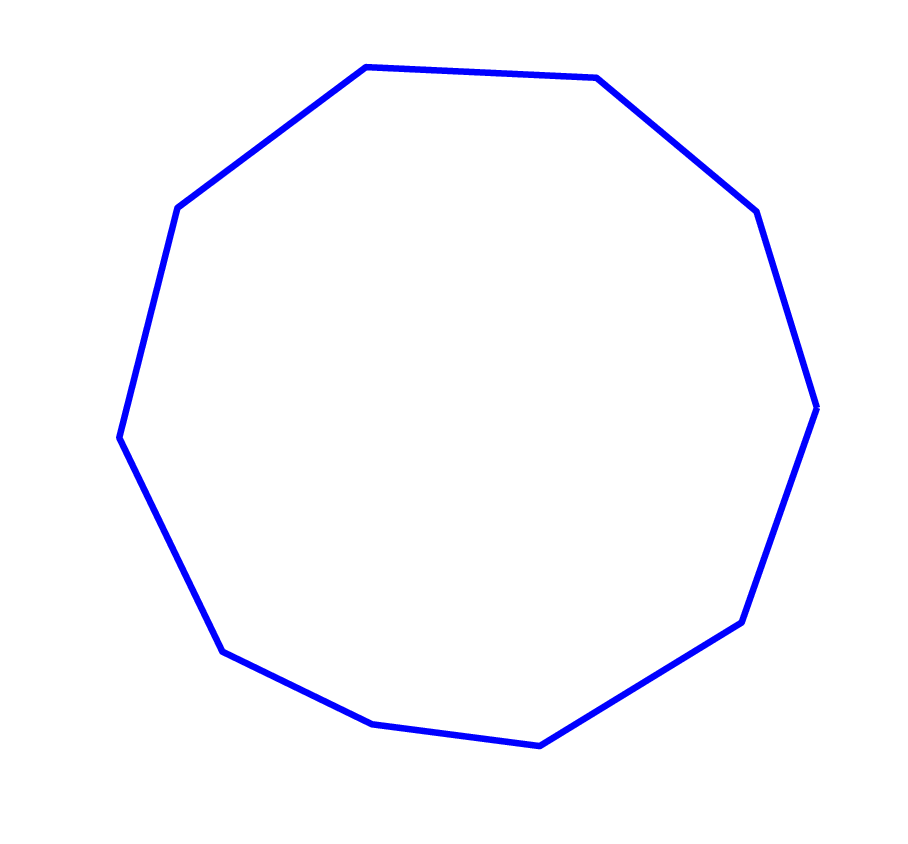}& \includegraphics[scale=0.135]{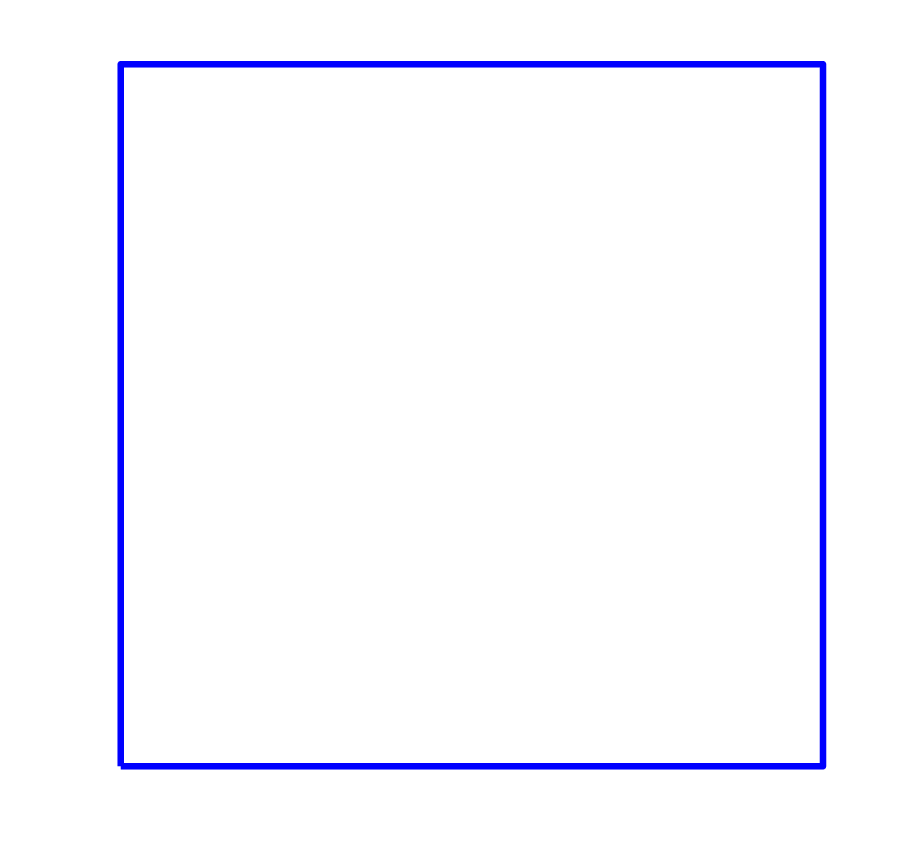}& \includegraphics[scale=0.17]{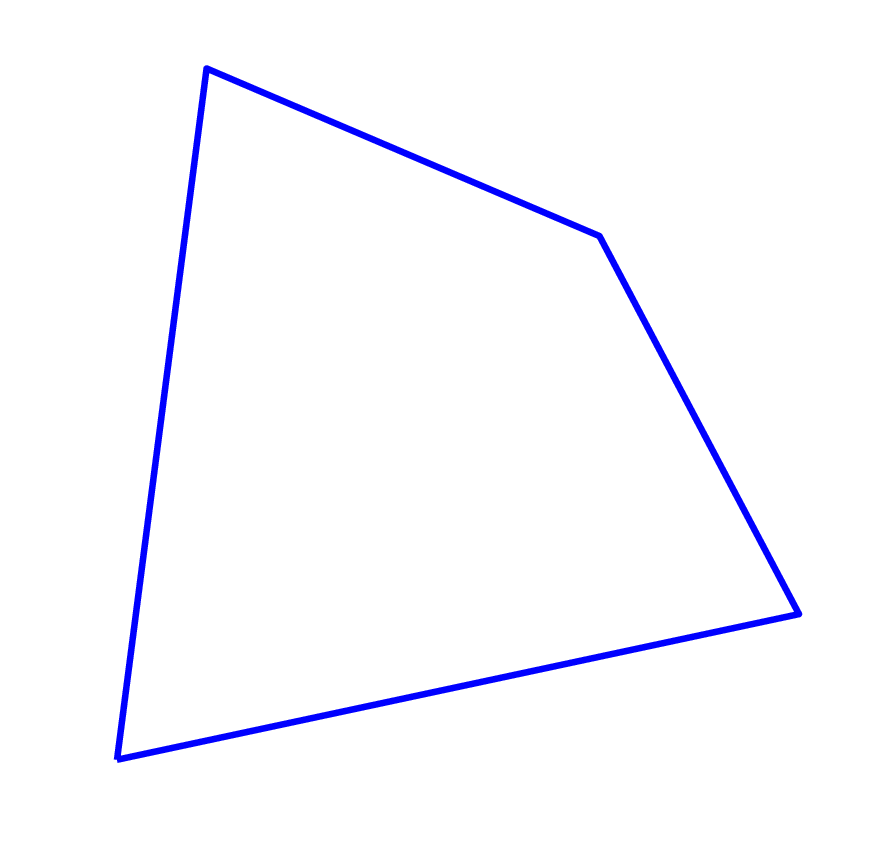}\\
\hline
\end{tabular}
    \caption{Numerically obtained optimal shapes corresponding to different values of $p_0$.}
    \label{fig:optimal_PLA}
\end{figure}

Finally, once the boundary is known, we use the vertical convexity of the diagram $\D_2$ (see \cite{ftouh}) to provide an improved numerical description of the diagram, see Figure \ref{fig:improved_text}.

 \begin{figure}[ht]
 \centering
    \includegraphics[scale=0.45]{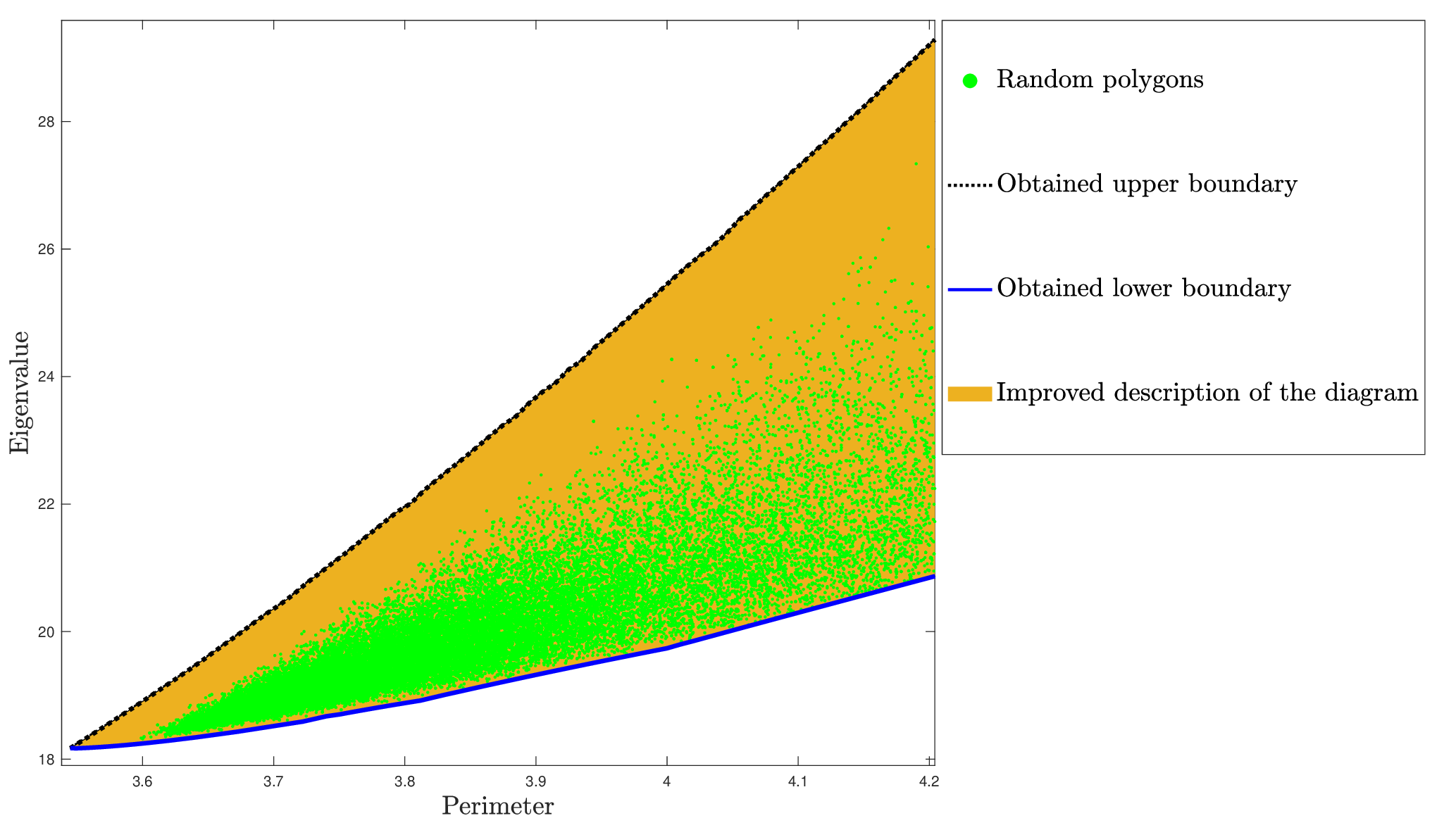}
    \caption{improved description of the diagram $(P,\lambda_1,A)$.}
    \label{fig:improved_text}
\end{figure}

\pagebreak
\subsection{The diagram \texorpdfstring{$\D_3$}{D3} of the triplet \texorpdfstring{$(d,\lambda_1,A)$}{(,lambda,A)}}\label{ss:dla}

Let us now consider the diagram $\D_3$ relating the diameter, the first Dirichlet eigenvalue and the area:
$$\D_3:= \{\big(d(\Om),\lambda_1(\Om)\big)\ |\ \Om \in \K\ \text{and}\ A(\Om)=1\}.$$

As for the previous diagrams, in a first time, we give an approximation of $\D_3$ by generating $10^5$ random convex polygons, see Figure \ref{fig:random_DLA}. 
 \begin{figure}[!ht]
 \centering
    \includegraphics[scale=0.5]{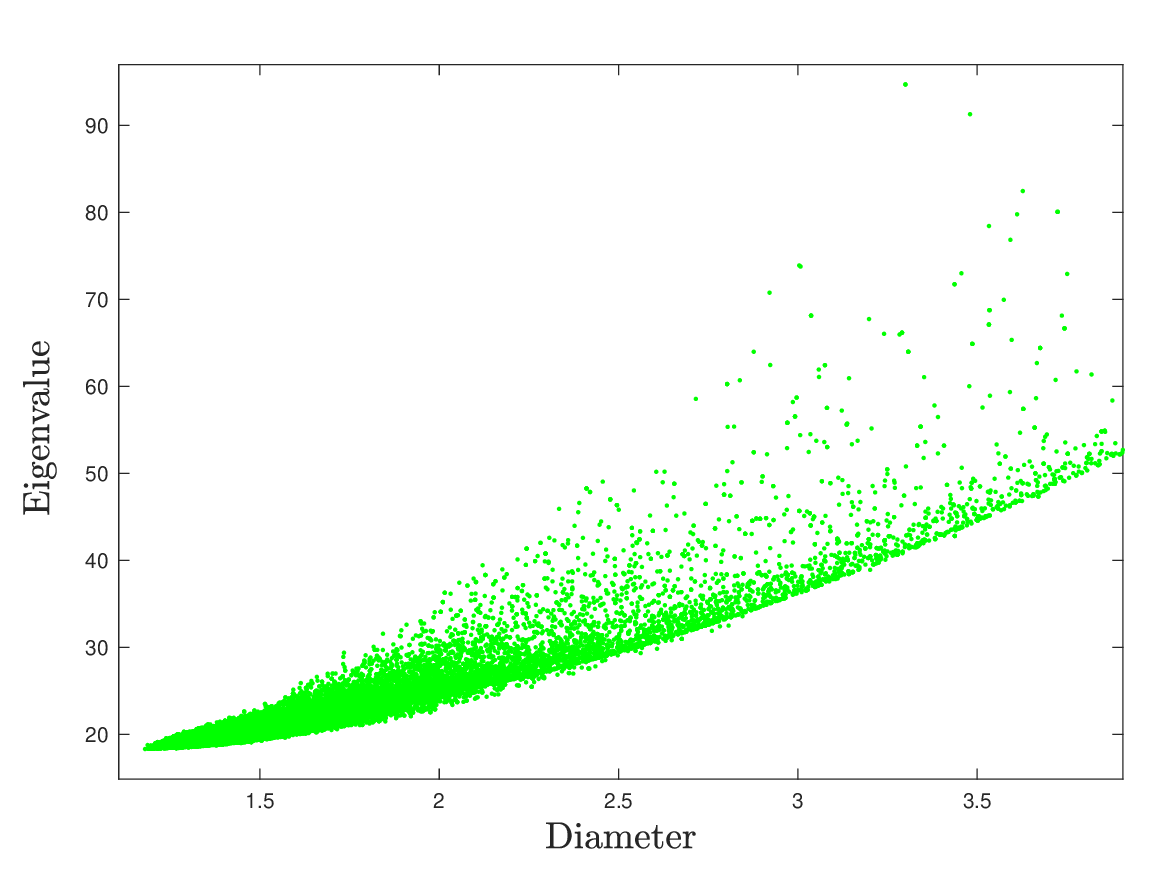}
    \caption{Approximation of the diagram $\D_3$ via $10^5$ random convex polygons.}
    \label{fig:random_DLA}
\end{figure}

Here also, we are aiming at describing the upper and lower boundaries of the diagram $\D_3$, which means to solve the following shape optimization problems: 
\begin{equation}\label{prob:dla}
\max\backslash\min\{\lambda_1(\Om)\ |\ \Om\in \K,\ d(\Om) = d_0\ \text{and}\ A(\Om) = 1\}, 
\end{equation}
where $d_0\in [2/\sqrt{\pi},+\infty)$ (because of the isodiametric inequality $\frac{d(\Om)}{\sqrt{A(\Om)}}\ge \frac{2}{\sqrt{\pi}}$).  

For \textbf{the lower boundary}, both the methods of parametrization via the Fourier coefficients of the support function (see Section \ref{ss:support_function}) and via the discretized radial function (see Section \ref{ss:radial_function}) provide satisfying results and suggest that the optimal sets are symmetrical 2-cap bodies, that are given by the the convex hulls of a disk and two points symmetric with respect its center (see Figure \ref{fig:two_cap}). 

 \begin{figure}[!ht]
        \centering
        \includegraphics[scale=0.3]{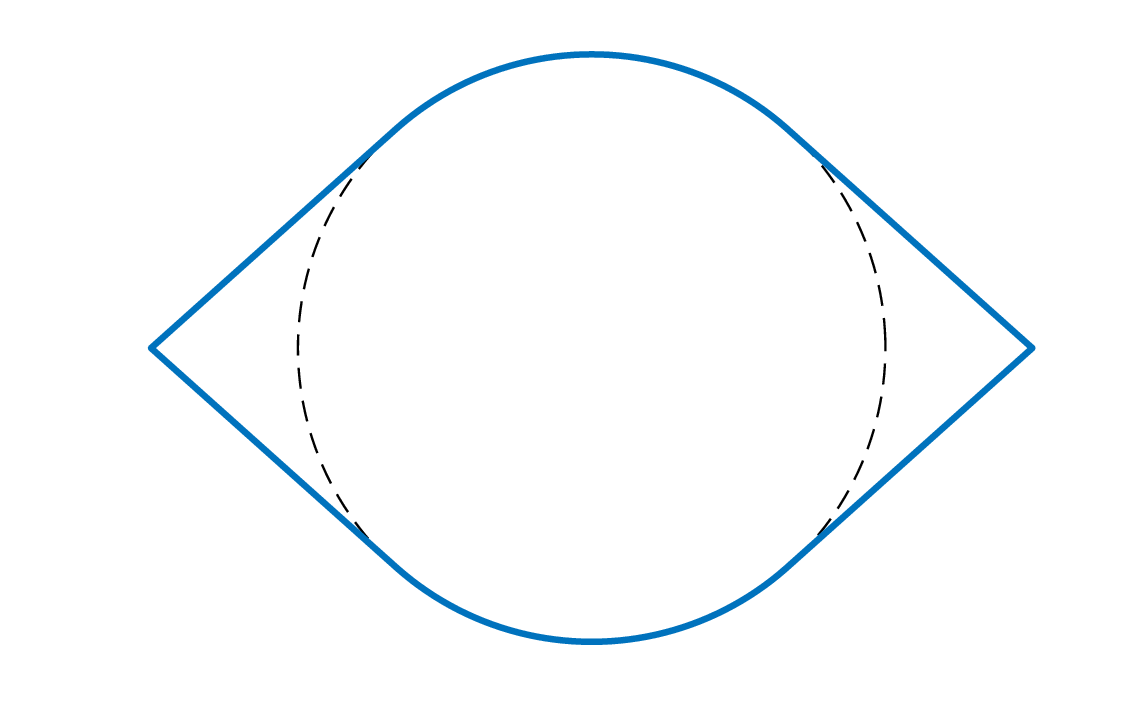} 
        \caption{Obtained symmetrical 2-cap body.}
        \label{fig:two_cap}
\end{figure}

On the other hand, \textbf{the upper domains} are quite surprising, since we (numerically) observe the existence of a threshold $d^*$, such that the solutions for $d> d^*$ seem to be given by symmetric spherical slices, that are domains defined as the intersection of a disk with a strip of width smaller that the disk's radius and centered at its center, see Figure \ref{fig:delyon_corp}, meanwhile, when $d< d^*$, the optimal domains seem to be given by some kind of smoothed regular nonagons, see Figure \ref{fig:delyon_corp}. 

 \begin{figure}[!ht]
        \centering
        \begin{tabular}{cc} \includegraphics[scale=0.35]{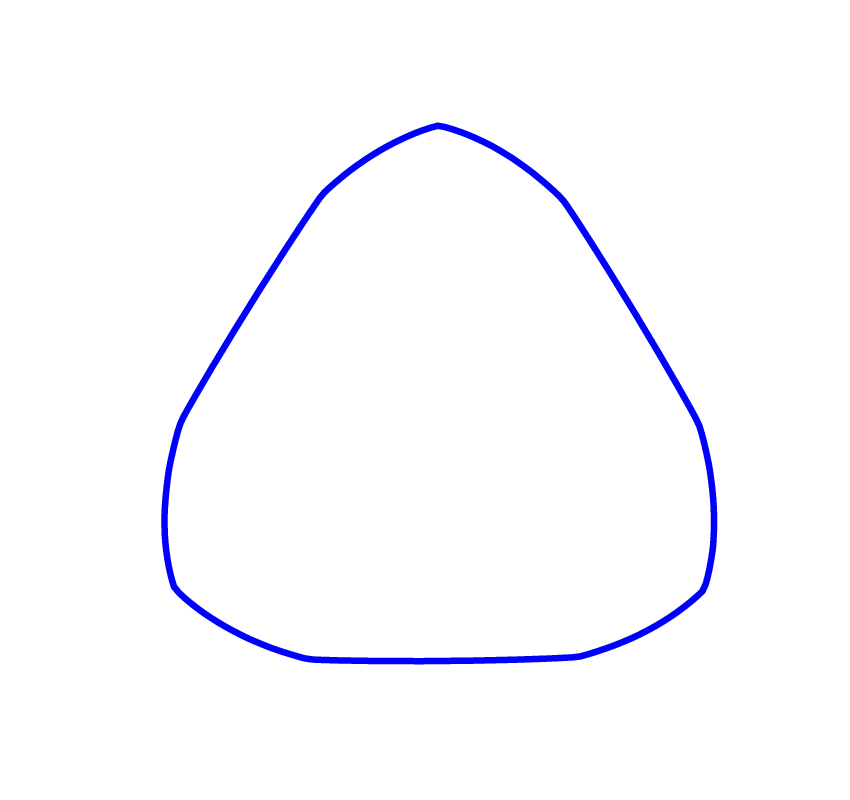} & \includegraphics[scale=0.35]{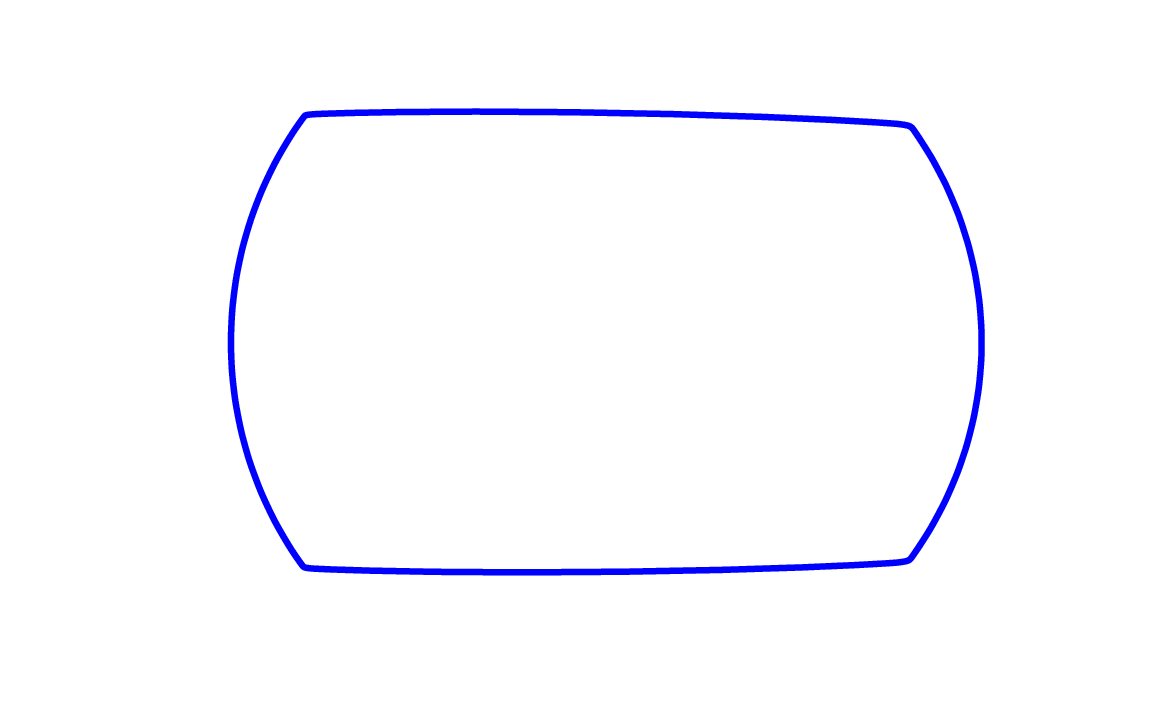} \\
        \end{tabular}
        \caption{Obtained upper shapes corresponding to $d_0=1.18$ for the smoothed nonagon and to $d_0=1.33$ for the symmetric slice, we used the parametrization via the Fourier coefficients of the support function with 161 coefficient ($N=80$).}
        \label{fig:delyon_corp}
\end{figure}

At a first sight, it may be surprising that the optimal shapes do not "continuously" vary in terms of the involved parameters, but, we should note that this phenomenon has recently been observed in \cite{delyon2}, where the authors provide the complete description of the diagram involving the diameter $d$, the inradius $r$ and the area $A$; they prove that one of the boundaries is filled by smoothed nonagons and symmetrical slices meanwhile the other one is filled by symmetrical 2-cap bodies. This leads us to investigate these families of shapes that also seem to be extremal shapes for the diagram $\D_3$ (see Figure \ref{fig:adl_avant}). These similarities may be explained by the fact that the functional $1/r$ corresponds to the first eigenvalue of the infinity-Laplace operator $\Delta_\infty$ which may be defined as the limit when $p\rightarrow+\infty$ of the $p$-Laplace operator (see \cite{MR2431408} and references therein). Meanwhile, $\lambda_1$ corresponds to the first eigenvalue of the $2$-Laplace operator, see \cite{MR2062882} for more details.  

We then compute the values taken by the involved functionals on the symmetrical slices and the smoothed nonagons and obtain Figure \ref{fig:adl_avant}. 
 \begin{figure}[!ht]
 \centering
    \includegraphics[scale=0.54]{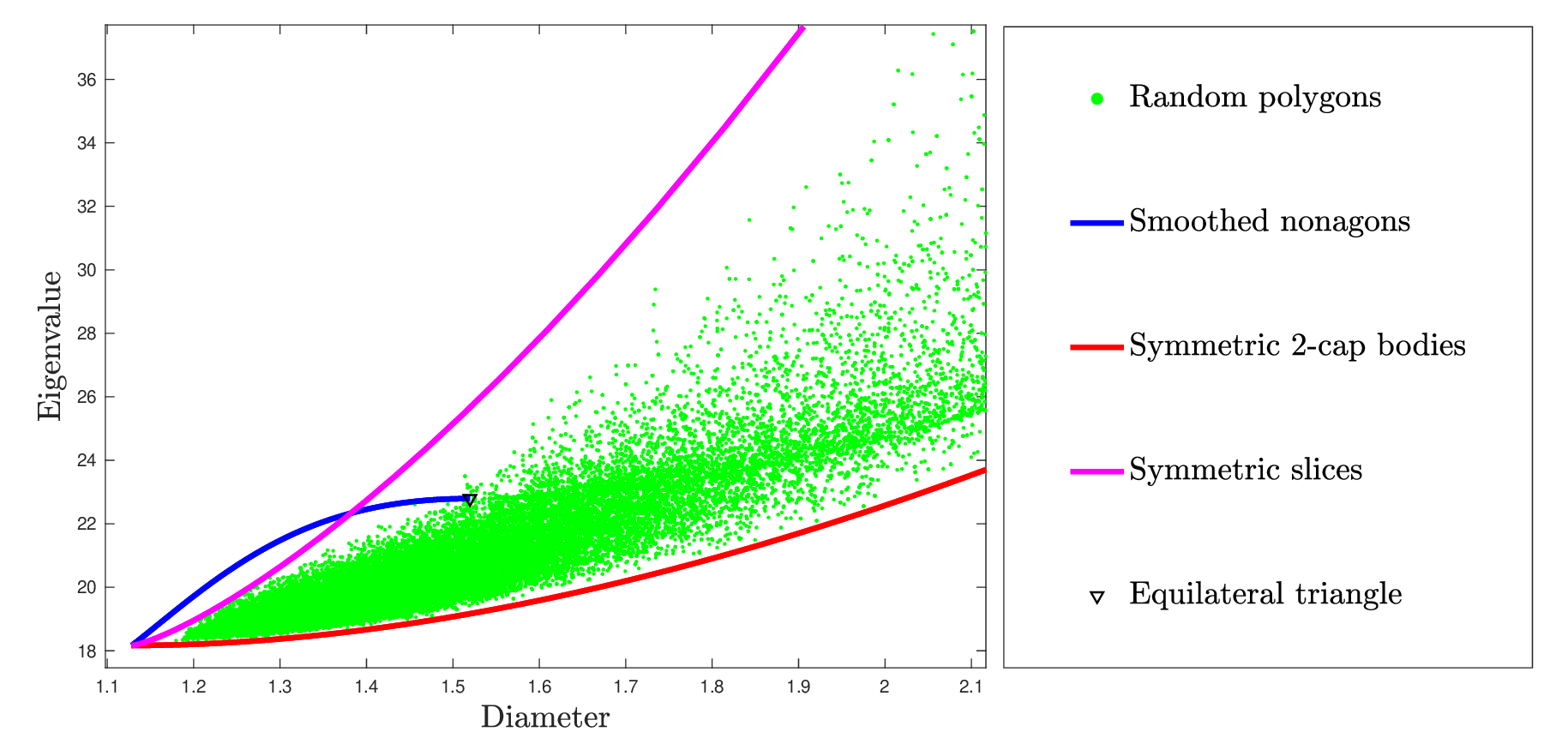}
    \caption{Approximation of the diagram  $\D_3$ with expected extremal sets.}
    \label{fig:adl_avant}
\end{figure}

\pagebreak
The expected optimal shapes and the best we managed to find are the ones given in the following table: 
\begin{figure}[!ht]
    \centering
\begin{tabular}{|c|c|c|c|c|c|}
\hline
 Problem & $d_0=d(B)=2/\sqrt{\pi}$ & $d_0=1.2$ & $d_0=d^*\approx 1.38\dots$& $d_0=1.9$\\
\hline
Upper boundary & \includegraphics[scale=0.16]{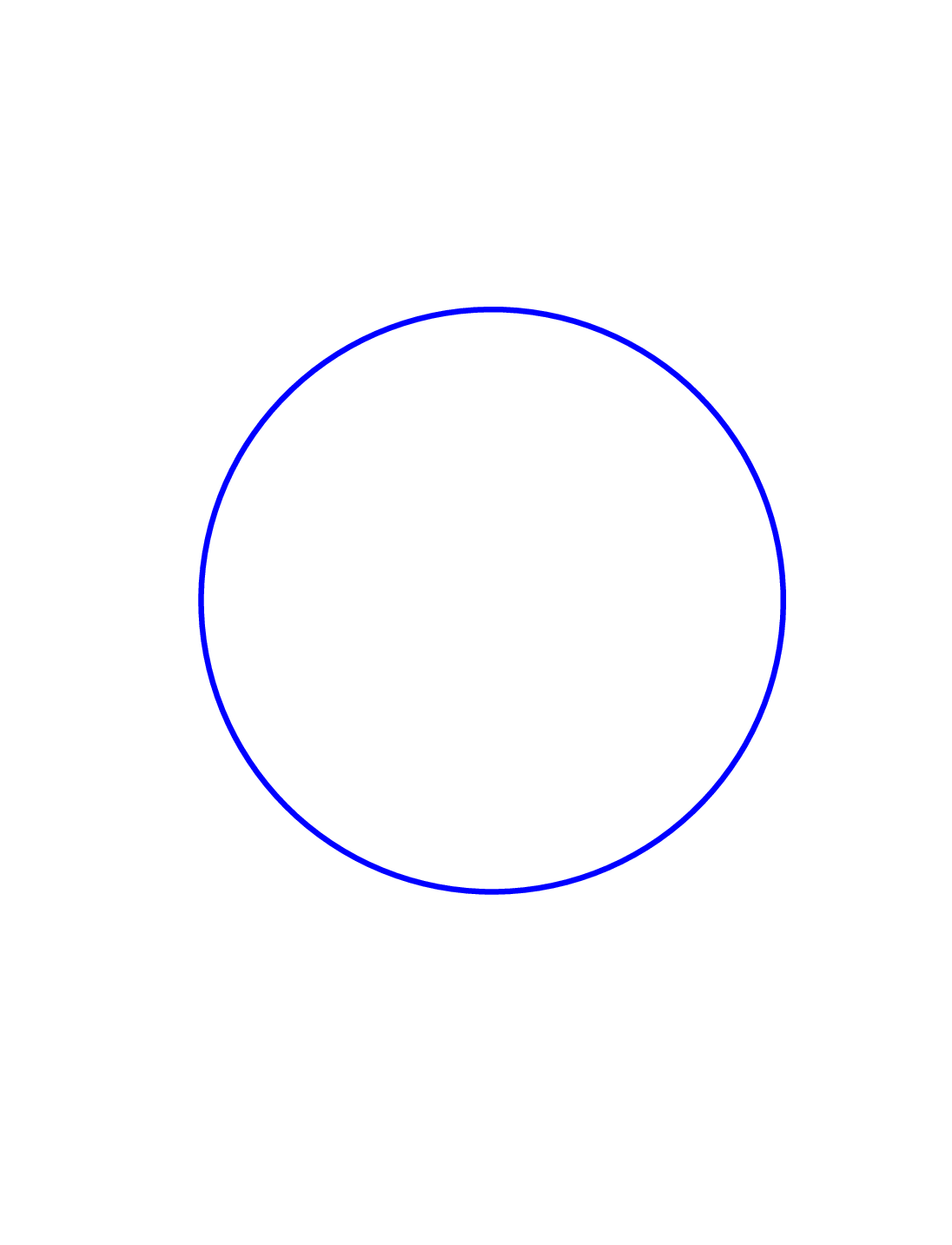} & \includegraphics[scale=0.15]{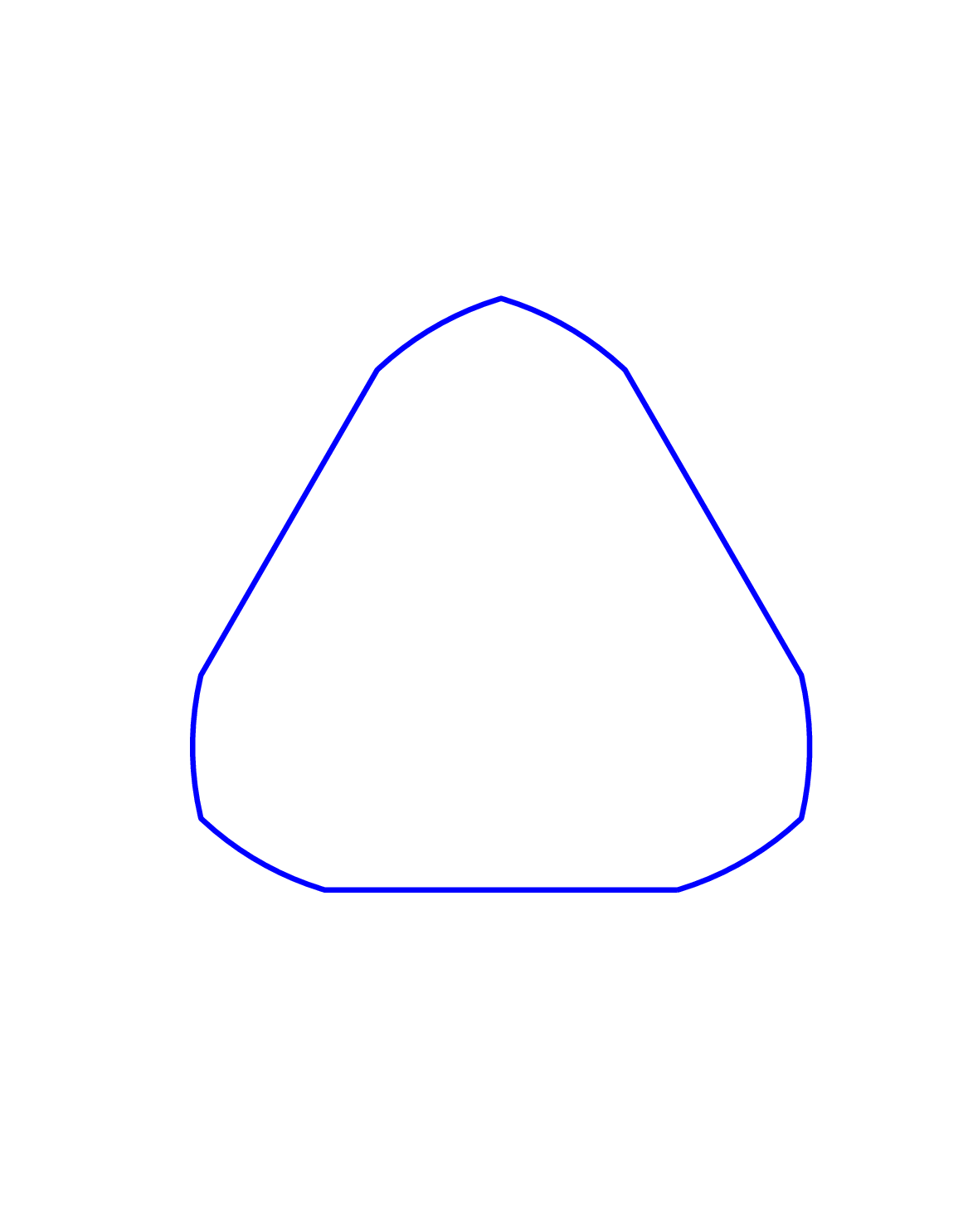}&  \includegraphics[scale=0.16]{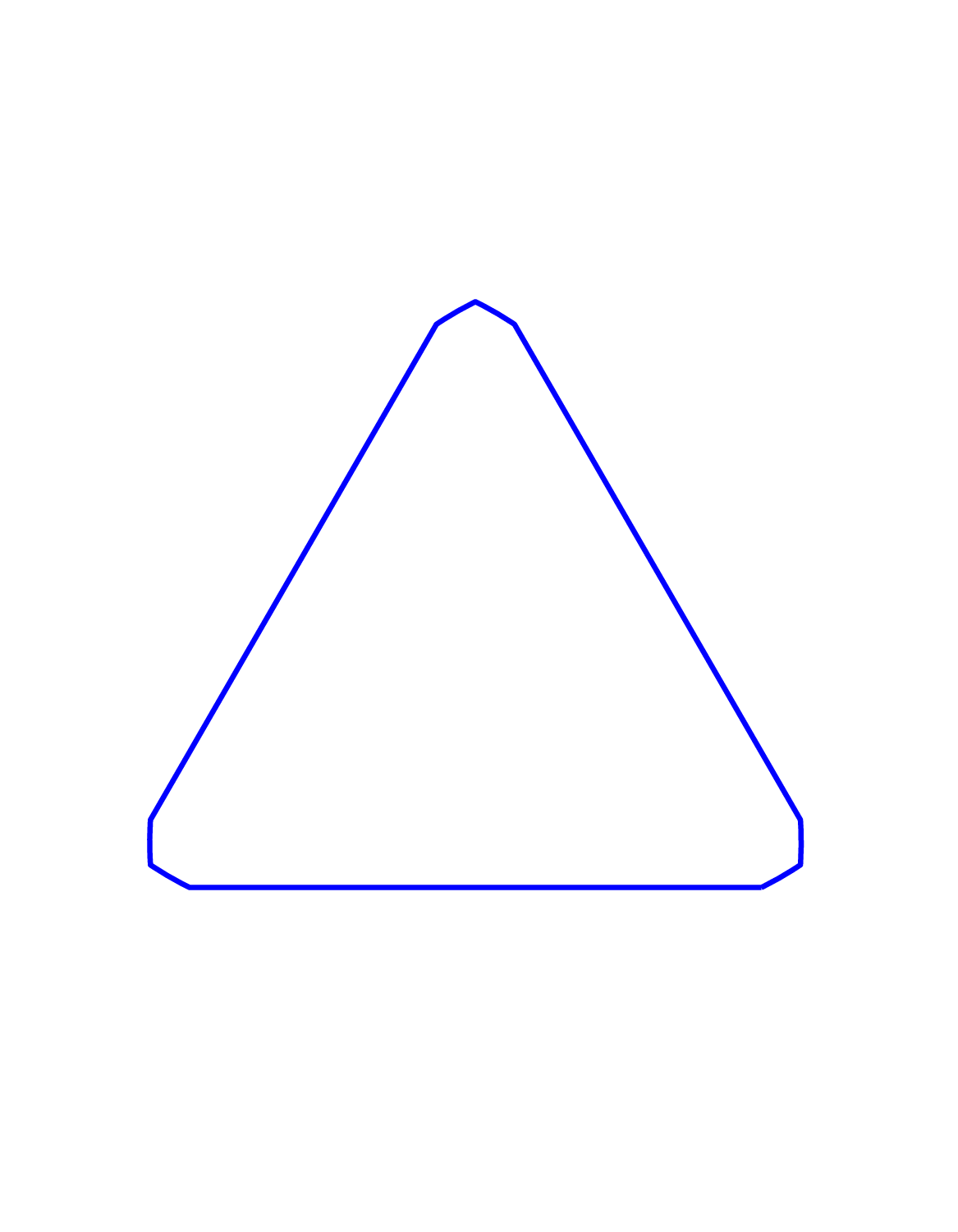}
\includegraphics[scale=0.22]{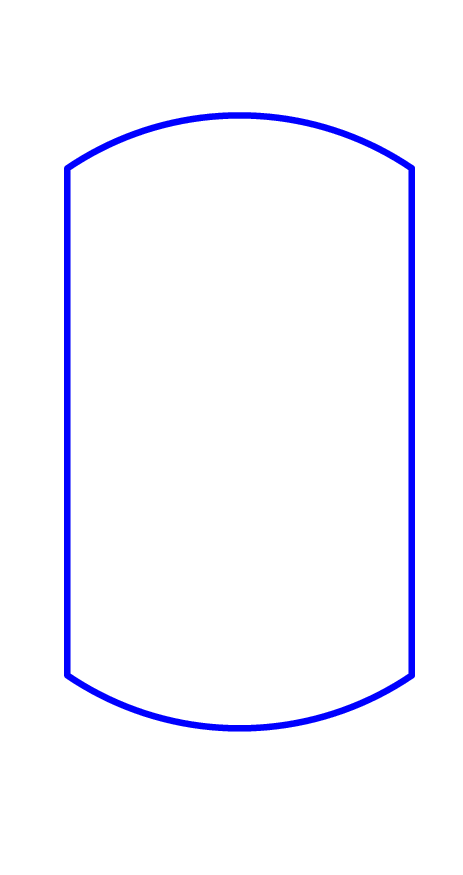}
& \includegraphics[scale=0.27]{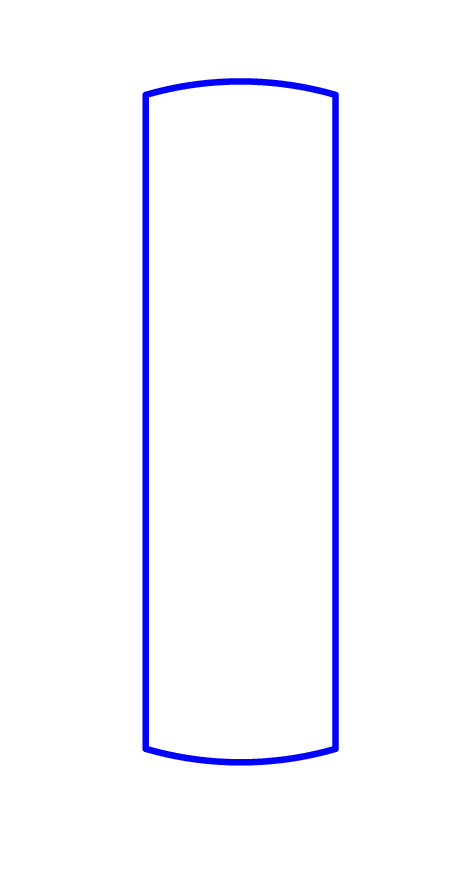}\\
\hline
Lower boundary &\includegraphics[scale=0.16]{disk__.eps}& \includegraphics[scale=0.225]{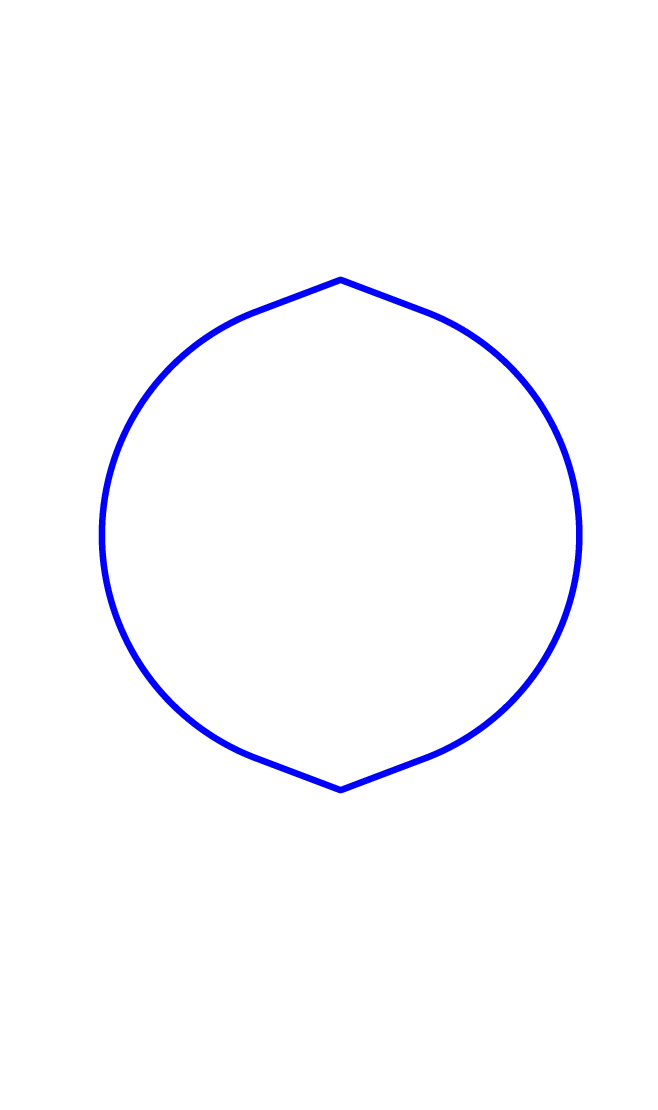}& \includegraphics[scale=0.24]{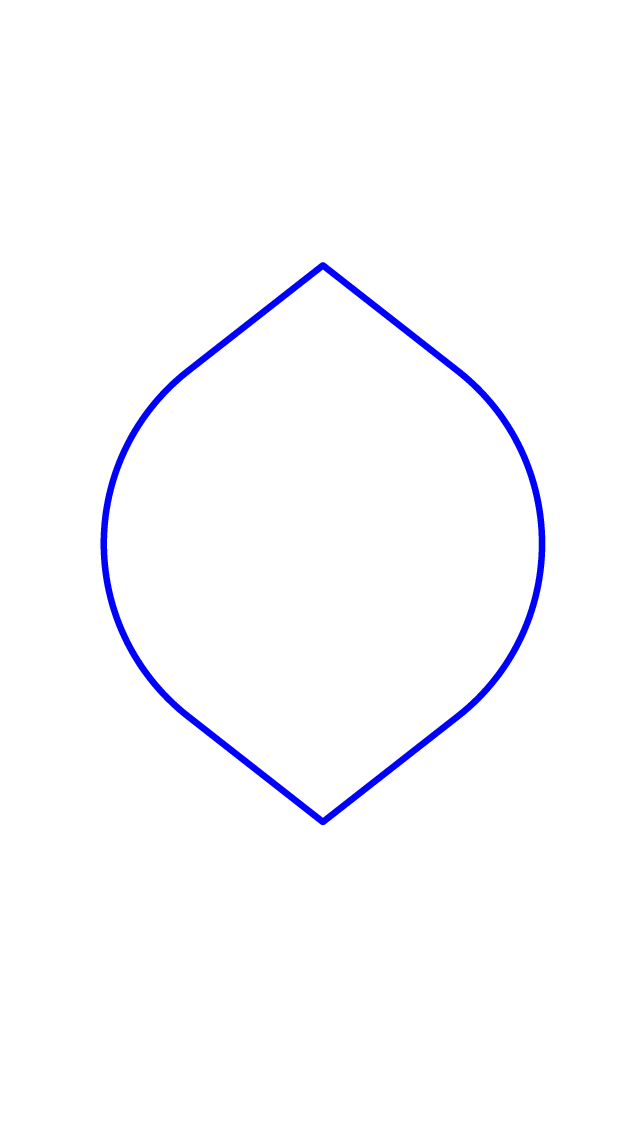}& \includegraphics[scale=0.28]{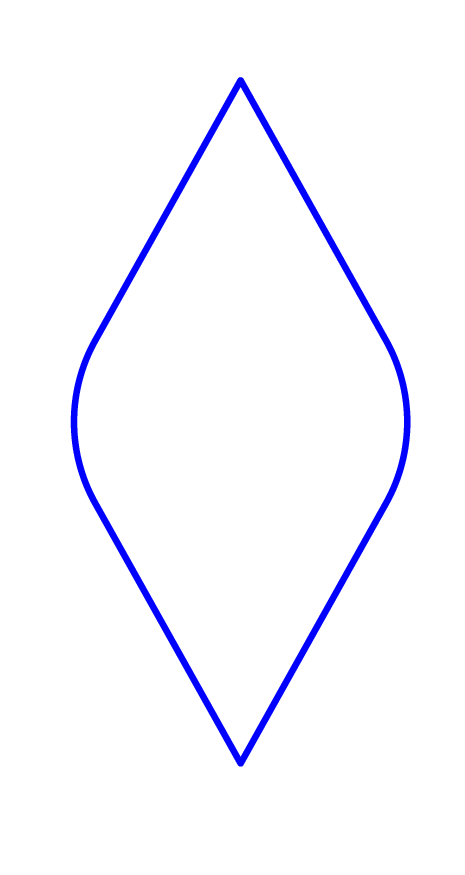}\\
\hline
\end{tabular}
    \caption{The best known shapes corresponding to different values of $d_0$.}
    \label{fig:optimal_DLA}
\end{figure}

Finally, by the vertical convexity of the diagram $\D_3$ stated in Theorem \ref{th:main_2}, we are able to provide an improved approximation of $\D_3$, see Figure \ref{fig:improved_adl}. 

 \begin{figure}[!ht]
 \centering
    \includegraphics[scale=0.6]{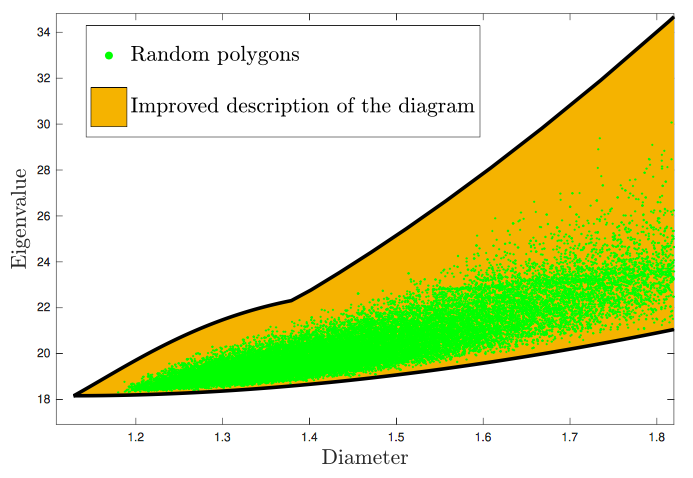}
    \caption{An improved description of the diagram $\D_3$.}
    \label{fig:improved_adl}
\end{figure}

\pagebreak
\color{black}
\subsection{Conclusion }
In the present paper, we propose a method to obtain improved descriptions of Blaschke--Santal\'o diagrams that combines theoretical and numerical methods. The quality of the obtained  results depends on the accuracy of the resolution of the shape optimization problems 
\begin{equation}
    \min\slash \max \{(J_1(\Om)\ |\ \text{$\Omega\in \K$},\ J_2(\Om)=c_0\ \text{and}\ J_3(\Om)=1\}.
\end{equation}
As we have noticed in the present work, the capacity of our algorithms to approximate optimal shapes strongly depends on their regularity and geometry. Therefore, simple shapes such as polygons, as in Section \ref{ss:pla},  will be easier to approximate (provided that we choose an adapted parametrization) than a more complicated shape such as smoothed nonagons as in Section  \ref{ss:proof_dla}. Therefore, it is interesting to develop relevant numerical schemes for the resolution of shape optimization problems under convexity constraint. 

We note that a clever idea to approximate Blaschke--Santal\'o diagrams by using Lloyd's algorithm was introduced recently in \cite{oudet_bogosel}. Nevertheless, it seems that treating the case of spectral functionals seems to present a serious computational challenge.  
\color{black}

\section*{Appendix: Computation of the shape derivative of the diameter}
 In the present appendix, we propose a method to compute the shape derivative of the diameter. This formula is used for numerical shape optimization in the previous sections. 
\begin{theorem}\label{th:diametre}
Let $V\in W^{1,\infty}(\R^2,\R^2)$ be a perturbation vector field. The diameter functional $d$ admits a directional shape derivative in the direction $V$ in any compact set $\Om \subset\R^2$. Moreover, we have
\begin{align*}
\exists (x_\infty,y_\infty) \in \Omega\times \Omega,\ \  d'(\Omega,V)&=\lim\limits_{t\rightarrow0^+}\frac{d(\Omega_t)-d(\Om)}{t}\\&= \sup\left\{\left\langle\frac{x-y}{|x-y|},V(x)-V(y)\right\rangle\ \Big|\ x,y \in \Omega,\ \text{such that}\  |x-y|=d(\Om)\right\}\\
&= \left\langle\frac{x_{\infty}-y_{\infty}}{|x_{\infty}-y_{\infty}|},V(x_{\infty})-V(y_{\infty})\right\rangle,
\end{align*}
where $\Om_t:=(I+tV)(\Om)$ and $I:x\in \R^2\longmapsto x\in \R^2$ is the identity map.
\end{theorem}

\begin{proof}

We want to prove the existence and compute the limit$\lim\limits_{t\rightarrow0^+} \frac{d(\Omega_t)-d(\Om)}{t}$.

For every $t\ge 0$, $\Omega_t$ is compact as  it is the image of the compact $\Omega $ by the continuous map $(I+t V)$. Thus, since $d:(x,y) \in \mathbb{R}^2\times\mathbb{R}^2\rightarrow|x-y|$ is continuous, it is bounded from above and there exists $(x_t,y_t)\in \Om\times \Om$ such that $d(\Om_t)= |(I+t V)(x_t)-(I+t V)(y_t)|$. In what follows, we take $(x,y):=(x_0,y_0)$.\vspace{2mm}

We use $x_t$ and $y_t$ (resp. $x$ and $y$) as test points to bound $d(\Om_t)-d(\Om)$ from above  (resp. below): $$|(I+tV)(x)-(I+tV)(y)|-|x-y|\leq d(\Om_t)-d(\Om) \leq |(I+tV)(x_t)-(I+tV)(y_t)|-|x_t-y_t|.$$ 

For the lower estimate, we have
\begin{align*}
d(\Om_t)-d(\Om) &\ge |(I+t V)(x)-(I+t V)(y)|-|x-y|\\
&= |x+t V(x)-y-t V(y)|-|x-y| \\
&= \Big|(x-y)+t\Big(V(x)-V(y)\Big)\Big|-|x-y| \\
&= \sqrt{\Big|(x-y)+t\Big(V(x)-V(y)\Big)\Big|^2}-|x-y| \\
&= \sqrt{|x-y|^2+2t\left\langle x-y,V(x)-V(y)\right\rangle+o(t)}-|x-y|\\
&= |x-y|\sqrt{1+2t\left\langle\frac{x-y}{|x-y|},\frac{V(x)-V(y)}{|x-y|}\right\rangle+o(t)}-|x-y|\\
&= |x-y|\left(1+t\left\langle\frac{x-y}{|x-y|},\frac{V(x)-V(y)}{|x-y|}\right\rangle+o(t)\right )-|x-y|\\
&= t\left\langle\frac{x-y}{|x-y|},V(x)-V(y)\right\rangle+o(t).
\end{align*}

Thus $$\liminf\limits_{t\rightarrow0^+}\frac{d(\Om_t)-d(\Om)}{t}\ge \sup\left\{\left\langle\frac{x-y}{|x-y|},V(x)-V(y)\right\rangle\ |\ x,y \in \Omega\ \ \text{such that}\  |x-y|=d(\Om)\right\}.$$

Let us now focus on the upper estimate. Let $(x_{t_n})$ and $(y_{t_n})$ be respective sub-sequences of $(x_t)_t$ and $(y_t)_t$ such that  $$\lim\limits_{n\rightarrow+\infty} t_n=0 \ \ \  \text{and}\ \ \ \lim\limits_{n\rightarrow\infty} \frac{d(\Om_{t_n})-d(\Om)}{t_n} = \limsup\limits_{t\rightarrow0^+} \frac{d(\Om_t)-d(\Om)}{t}.$$ By Bolzano--Weirstrass Theorem, we assume without loss of generality that there exists $(x_\infty,y_\infty)\in\Omega\times \Om$ such that the sequences $(x_{t_n})$ and $(y_{t_n})$ respectively converge to $x_\infty$ and $y_\infty$. We then have $|x_\infty-y_\infty|=d(\Om)$. Indeed $$\forall(v,w)\in \Omega\times \Omega,\ \ \  |(I+t_nV)(x_{t_n})-(I+t_nV)(y_{t_n})|\ge |(I+t_nV)(v)-(I+t_nV)(w)|,$$
which is equivalent to $$\forall(v,w)\in \Omega\times \Omega,\ \ \ |x_{t_n}-y_{t_n}+t_n.V(x_{t_n})-t_n.V(y_{t_n})|\ge|v-w+t_n.V(v)-t_n.V(w)|.$$

Finally 
\begin{align*}
d(\Om_{t_n})-d(\Om) &\leq |(I+tV)(x_{t_n})-(I+tV)(y_{t_n})|-|x_{t_n}-y_{t_n}|\\
&= |x_{t_n}+{t_n}.V(x_{t_n})-y_{t_n}-{t_n}V(y_{t_n})|-|x_{t_n}-y_{t_n}| \\
&= \Big|(x_{t_n}-y_{t_n})+t_{n}\Big(V(x_{t_n})-V(y_{t_n})\Big)\Big|-|x_{t_n}-y_{t_n}| \\
&= \sqrt{\Big|(x_{t_n}-y_{t_n})+t_{n}\Big(V(x_{t_n})-V(y_{t_n})\Big)\Big|^2}-|x_{t_n}-y_{t_n}| \\
&= \sqrt{\Big|x_{t_n}-y_{t_n}\Big|^2+2t_{n}\Big\langle x_{t_n}-y_{t_n},V(x_{t_n})-V(y_{t_n})\Big\rangle+o(t_n)}-|x_{t_n}-y_{t_n}|\\
&= |x_{t_n}-y_{t_n}|\sqrt{1+2t_{n}\left\langle\frac{x_{t_n}-y_{t_n}}{|x_{t_n}-y_{t_n}|},\frac{V(x_{t_n})-V(y_{t_n})}{|x_{t_n}-y_{t_n}|}\right\rangle+o(t_n)}-|x_{t_n}-y_{t_n}|\\
&= |x_{t_n}-y_{t_n}|\sqrt{1+2t_{n}\left(\left\langle\frac{x_{\infty}-y_{\infty}}{|x_{\infty}-y_{\infty}|},\frac{V(x_{\infty})-V(y_{\infty})}{|x_{\infty}-y_{\infty}|}\right\rangle+o(1)\right)+o(t_{n})}-|x_{t_n}-y_{t_n}|\\
&= |x_{t_n}-y_{t_n}|\sqrt{1+2t_{n}\left\langle\frac{x_{\infty}-y_{\infty}}{|x_{\infty}-y_{\infty}|},\frac{V(x_{\infty})-V(y_{\infty})}{|x_{\infty}-y_{\infty}|}\right\rangle+o(t_{n})}-|x_{t_n}-y_{t_n}|\\
&= |x_{t_n}-y_{t_n}|\left(1+t_n\left\langle\frac{x_{\infty}-y_{\infty}}{|x_{\infty}-y_{\infty}|},\frac{V(x_{\infty})-V(y_{\infty})}{|x_{\infty}-y_{\infty}|}\right\rangle+o(t_{n})\right)-|x_{t_n}-y_{t_n}|\\
&= t_n\left\langle\frac{x_{\infty}-y_{\infty}}{|x_{\infty}-y_{\infty}|},V(x_{\infty})-V(y_{\infty})\right\rangle+o(t_{n}).
\end{align*}
Thus $$\limsup\limits_{t\rightarrow0^+}\frac{d(\Om_t)-d(\Om)}{t}\leq\left\langle\frac{x_{\infty}-y_{\infty}}{|x_{\infty}-y_{\infty}|},V(x_{\infty})-V(y_{\infty})\right\rangle.$$
By combining the inferior  and superior limits inequalities, we obtain
\begin{align*}
\liminf\limits_{t\rightarrow0^+}\frac{d(\Om_t)-d(\Om)}{t}&\ge \sup\left\{\left\langle\frac{x-y}{|x-y|},V(x)-V(y)\right\rangle\ |\ x,y \in \Omega\ ,\ \text{such that}\  |x-y|=d(\Om)\right\}\\
&\ge \left\langle\frac{x_{\infty}-y_{\infty}}{|x_{\infty}-y_{\infty}|},V(x_{\infty})-V(y_{\infty})\right\rangle\\
&\ge \limsup\limits_{t\rightarrow0^+}\frac{d(\Om_t)-d(\Om)}{t}\\
&\ge \liminf\limits_{t\rightarrow0^+}\frac{d(\Om_t)-d(\Om)}{t}.
\end{align*}
Finally, we deduce that the diameter admits a directional shape derivative in the direction $V$ and 
\begin{align*}
\exists (x_\infty,y_\infty) \in \Omega\times \Omega,\ \  \lim\limits_{t\rightarrow0^+}\frac{d(\Om_t)-d(\Om)}{t}&= \sup\left\{\left\langle\frac{x-y}{|x-y|},V(x)-V(y)\right\rangle\Big|\ x,y \in \Omega,\ \text{such that}\  |x-y|=d(\Om)\right\}\\
&= \left\langle\frac{x_{\infty}-y_{\infty}}{|x_{\infty}-y_{\infty}|},V(x_{\infty})-V(y_{\infty})\right\rangle.
\end{align*}

\end{proof}
\begin{remark}
    For clarity and coherence purposes, Theorem \ref{th:diametre} has been stated in the planar case, Nevertheless, it is not difficult to see that the result holds in higher dimensions. 
\end{remark}
\section*{Data Availability Statement}
The datasets generated during and analyzed during the current study are available from the corresponding author on reasonable request.

\section*{Acknowledgments}
The author extends gratitude to the anonymous referee for their thoughtful comments and meticulous review that helped to considerably improve the quality of the manuscript. 

This work was supported by the Alexander von Humboldt-Professorship program and by the project  ANR-18-CE40-0013 SHAPO financed by the French Agence Nationale de la Recherche (ANR). 
The author is thankful to Jimmy Lamboley and Edouard Oudet for stimulating discussions and comments on the topic of the paper. 

\bibliographystyle{plain}
\bibliography{samplee}

\vspace{1cm}
(Ilias Ftouhi) \textsc{ Friedrich-Alexander-Universit{\"a}t Erlangen-N{\"u}rnberg, Department of Mathematics, Chair of Applied Analysis (Alexander von Humboldt Professorship), Cauerstr. 11, 91058 Erlangen, Germany.}\vspace{1mm}

\textit{Email address:} \textbf{\texttt{ilias.ftouhi@fau.de}}







\end{document}